\documentclass{article}
\usepackage[utf8]{inputenc}
\usepackage{savesym}
\usepackage{amsthm,amssymb,amsfonts,amsmath,bm,xcolor,newtxmath}

\usepackage{faktor}
\savesymbol{bigtimes}
\usepackage{mathabx}
\usepackage{enumitem}
\restoresymbol{NEWbigtimes}{bigtimes}
\usepackage[margin=1in]{geometry}
\usepackage[]{todonotes}
\usepackage{tikz-cd}
\usepackage[hang,flushmargin]{footmisc} % Removes indent from footnote

\newtheorem{theorem}{Theorem}
\newtheorem*{theorem*}{Theorem}
\newtheorem{lemma}[theorem]{Lemma}
\newtheorem{proposition}[theorem]{Proposition}

\newtheorem{corollary*}[theorem]{Corollary}
\newtheorem{thm}{Theorem}

\theoremstyle{definition}
\newtheorem{definition}[theorem]{Definition}
\theoremstyle{remark}
\newtheorem{remark}[theorem]{Remark}

\newtheorem{example}[theorem]{Example}

\numberwithin{theorem}{section}

\newcommand\cB{{\mathcal B}}
\newcommand\cC{{\mathcal C}}

\newcommand\cL{{\mathcal L}}

\newcommand\cP{{\mathcal P}}
\newcommand\cQ{{\mathcal Q}}
\newcommand\cV{{\mathcal V}}
\newcommand\cX{{\mathcal X}}

\newcommand\CC{{\mathbb C}}
\newcommand\DD{{\mathbb D}}
\newcommand\FF{{\mathbb F}}
\newcommand\HH{{\mathbb H}}
\newcommand\KK{{\mathbb K}}
\newcommand\NN{{\mathbb N}}
\newcommand\PP{{\mathbb P}}

\newcommand\RR{{\mathbb R}}
\renewcommand\SS{{\mathbb S}}
\newcommand\TT{{\mathbb T}}
\newcommand\UU{{\mathbb U}}
\newcommand\VV{{\mathbb V}}
\newcommand\ZZ{{\mathbb Z}}

\newcommand\ba{{\bm a}}

\newcommand\bu{{\bm u}}

\newcommand\bM{{\bm M}}

\newcommand\tP{{\widetilde{P}}}
\newcommand\tQ{{\widetilde{Q}}}

\newcommand{\0}{{\vmathbb 0}}
\newcommand{\1}{{\vmathbb 1}}

\newcommand\SetOf[2]{\left\{\left.#1\vphantom{#2}\ \right|\ #2\vphantom{#1}\right\}}
\newcommand\hseries[2]{#1\llbracket t^{#2} \rrbracket}
\newcommand\pseries[1]{#1\{\!\{ t \}\!\}}

\newcommand\iprod[2]{\langle{#1},{#2}\rangle}

\newcommand{\loopsum}{+_{\ell}}
\newcommand{\coloopsum}{+_{co\ell}}

\newcommand{\shortdots}{\!..}

\newcommand\quot{\twoheadrightarrow}
\newcommand\quotient{\twoheadrightarrow}
\DeclareMathOperator{\Gr}{Gr}

\DeclareMathOperator{\conv}{conv}

\DeclareMathOperator{\lc}{lc}
\DeclareMathOperator{\lt}{lt}
\DeclareMathOperator{\lp}{lp}

\DeclareMathOperator{\val}{val}
\DeclareMathOperator{\sval}{sval}
\DeclareMathOperator{\fval}{fval}
\DeclareMathOperator{\phval}{phval}
\DeclareMathOperator{\ph}{ph}

\DeclareMathOperator{\sign}{sign}

\DeclareMathOperator{\MinSupp}{MinSupp}
\DeclareMathOperator{\argmin}{argmin}
\DeclareMathOperator{\sgn}{sgn}
\DeclareMathOperator{\spn}{span}
\DeclareMathOperator{\tspan}{tspan}
\DeclareMathOperator{\rank}{rk}
\DeclareMathOperator{\initial}{in}

\makeatletter
\def\metadata{\xdef\@thefnmark{}\@footnotetext}
\makeatother

\title{Matroids over tropical extensions of tracts}
\author{Ben Smith}
\begin{document}

\maketitle

\metadata{\textsc{Correspondence.} Lancaster University, LA1 4YF, United Kingdom. \texttt{b.smith9@lancaster.ac.uk}}

\metadata{\textsc{MSC Classes.}
	05B35, % Combinatorial aspects of matroids and geometric lattices
	12K99, % Generalizations of fields: None of the above, but in this section
	14T15 % Combinatorial aspects of tropical varieties
  (Primary)
	14M15, % Grassmannians, Schubert varieties, flag manifolds
	52B40, % Matroids in convex geometry (realizations in the context of convex polytopes, convexity in combinatorial structures, etc.)
	52C40 % Oriented matroids in discrete geometry
	(Secondary)
}

\metadata{\textsc{Keywords.} matroids over tracts; initial matroid; tropical extension; tropical linear space; hyperfield; flag matroid; positroid.}

\begin{abstract}
A tract $F$ is an algebraic structure where multiplication is defined but addition is only partially defined.
They were introduced by Baker and Bowler as a unified framework to study generalisations of matroids, including oriented and valuated matroids.
A \emph{tropical extension} $F[\Gamma]$ is a tract obtained by extending a tract $F$ by an ordered abelian group $\Gamma$.
Key examples include the tropical hyperfield as a tropical extension of the Krasner hyperfield, and the signed tropical hyperfield as a tropical extension of the sign hyperfield.

We study matroids over tropical extensions of tracts, including valuated matroids and oriented valuated matroids.
We generalise the correspondence between valuated matroids and their initial matroids, showing that $M$ is an $F[\Gamma]$-matroid if and only if every initial matroid $M^\bu$ is an $F$-matroid.
We also show analogous results for flag matroids and positroids over tropical extensions, utilising the circuit and covector descriptions we derive for $F[\Gamma]$-matroids.

We conclude by studying images of linear spaces in \emph{enriched valuations}, valuation maps enriched with additional data.
These give rise to \emph{enriched tropical linear spaces}, including signed tropical linear spaces as a key examples.
As an application of our results, we prove a structure theorem for enriched tropical linear spaces, generalising the characterisation of projective tropical linear spaces of Brandt-Eur-Zhang.
\end{abstract}

\section{Introduction}

A \emph{tract} $F = (F^\times, N_F)$ is generalisation of a field where multiplication is defined but summations are only remembered if they `sum to zero'.
Explicitly, $F^\times$ is a multiplicative abelian group and $N_F \subset \NN[F^\times]$ is a collection of formal sums known as the null set.
Tracts are a generalisation of both hyperfields and partial fields, both of whom have proven fundamental within matroid theory.
Partial fields are fields where addition is only partially defined, and were introduced to study problems of matroid representability~\cite{Semple+Whittle:1996,Pendavingh+vanZwam:2010}.
Moreover, Baker and Bowler recently demonstrated that matroids, oriented matroids and valuated matroids can all be viewed as matroids over hyperfields, specifically the Krasner hyperfield $\KK$, the sign hyperfield $\SS$ and the tropical hyperfield $\TT$ respectively~\cite{Baker+Bowler:2017}.
Soon after, matroids over tracts were introduced as a common generalisation of these frameworks~\cite{Baker+Bowler:2019}.

One definition is as follows.
Given a tract $F$, a (strong) \emph{$F$-matroid} $M$ on ground set $E$ of rank $r$ is a function $P_M \colon {E \choose r} \rightarrow F$ that satisfies the following `Pl\"ucker relations' over $F$
\begin{equation}\label{eq:GP}
\sum_{j \in J \setminus I} \sign(j;I,J) \cdot P_M(I+j)\cdot P_M(J-j) \in N_F \, , \quad \forall I \in {E \choose r-1} \, , \, J \in {E \choose r+1} \, .
\end{equation}
We call this function $P_M$ the \emph{Pl\"ucker vector} of $M$, and can be thought of as the `bases' definition of $F$-matroids.
In fact when $F= \KK$, equation \eqref{eq:GP} is precisely the strong basis exchange axiom for ordinary matroids.
As with ordinary matroids, $F$-matroids can be described by a number of cryptomorphic axiom systems, including circuits and (co)vectors~\cite{Anderson:2019, Bowler+Pendavingh:2019}.
%There has been much progress since this point, including axiom systems for . %, a theory of flag matroids over tracts~\cite{Jarra+Lorscheid:2022} and a construction of a moduli space of $F$-matroids~\cite{Baker+Lorscheid:2021}.

Our key object throughout will be a tropical extension of a tract.
Given a tract $F$ and an ordered abelian group $\Gamma$, we define the \emph{tropical extension} $F[\Gamma]$ as the tract with multiplicative abelian group $F^\times \times \Gamma$, and null set
    \[
    N_{F[\Gamma]} = \left\{ \sum_{i\in I} (a_i, \gamma_i) \, \Big| \, \sum_{i \in I_{\min}} a_i \in N_F \text{ where } I_{\min} = \{i\in I \mid \gamma_i \leq \gamma_j \; \forall j \in I\}\right\} \cup \{\infty\} \, ,
    \]
where $\infty$ is the zero element of $F[\Gamma]$.
The tropical hyperfield $\TT \cong \KK[\RR]$ is itself a tropical extension, namely the tropical extension of the Krasner hyperfield $\KK$ by the ordered abelian group $(\RR,+)$.
Another key example is the signed (or real) tropical hyperfield $\TT_\pm \cong \SS[\RR]$, obtained as the tropical extension of the sign hyperfield $\SS$ by $(\RR,+)$.
The matroids over these special tropical extensions are valuated matroids ($\TT$-matroids) and oriented valuated matroids ($\TT_\pm$-matroids) respectively.
We demonstrate that many fundamental properties of valuated matroids extend to $F[\Gamma]$-matroids over arbitrary tropical extensions of tracts, already recovering new results for oriented valuated matroids.

Our jumping off point is the following property of valuated matroids that states they can be understood globally in terms of their local pieces called initial matroids.
Given a valuated matroid $M$ with Pl\"ucker vector $P \colon {E \choose r} \rightarrow \TT$ and some $\bu \in \RR^E$, its \emph{initial matroid} $M^\bu$  is the matroid with bases
\begin{equation} \label{eq:initial}
\cB(M^\bu) = \SetOf{B \in {E \choose r}}{P(B) - \sum_{i \in B} u_i \leq P(B') - \sum_{j \in B'} u_j \; \forall B' \in {E \choose r}} \, , \quad \bu \in \RR^n \, .
\end{equation}
Note that we can associate to $M^\bu$ the Pl\"ucker vector $P^\bu\colon {E \choose r} \rightarrow \KK$ defined as $P^\bu(B) = \1$ if and only if $B \in \cB(M^\bu)$.

Implicit in the work of Kapranov~\cite{Kapranov:1993}, it was shown independently by Murota and Speyer that valuated matroids are characterised by their initial matroids.
\begin{theorem*}\cite{Murota:2003, Speyer:2008} \label{thm:vm}
%Let $P \colon {E \choose r} \rightarrow \TT$ be an arbitrary function.
$M$ is a valuated matroid if and only if $M^\bu$ is a matroid for all $\bu \in \RR^E$.
\end{theorem*}
Rephrasing this theorem in the language of tracts, we have that $M$ is a $\KK[\RR]$-matroid if and only if $M^\bu$ is a $\KK$-matroid for all $\bu \in \RR^E$.
%Our first main result is that this theorem holds for general tropical extensions, where we replace $\KK$ with an arbitrary perfect\footnote{There are two notions of matroids over tracts, namely weak and strong matroids.
%Over a \emph{perfect tract}, these notions coincide so we do not have to worry about such technicalities.}\todo{Make sure to mention somewhere what happens for non-perfect tracts} tract $F$ and $\RR$ with any ordered abelian group.
Our first main result is that this theorem holds for general tropical extensions, where we replace $\KK$ with an arbitrary (perfect) tract $F$ and $\RR$ with any ordered abelian group.
This requires defining the initial $F$-matroid $M^\bu$ of an $F[\Gamma]$-matroid $M$ with respect to some $\bu \in \Gamma^E$, analogously to~\eqref{eq:initial}.
This is given in Definitions \ref{def:toric+initial+tuple} and \ref{def:initial+tuple}.
%We then give an affirmative answer to the question with the following theorem.

\begin{thm} \label{thm:A}
Let $F$ be a perfect tract and $\Gamma$ an ordered abelian group.
$M$ is an $F[\Gamma]$-matroid if and only if $M^\bu$ is an $F$-matroid for all $\bu \in \Gamma^E$.
\end{thm}
As stated previously, this recovers the theorems of Murota and Speyer when $F = \KK$ and $\Gamma = \RR$.
When $F = \SS$ and $\Gamma = \RR$, we get a special case of this theorem stating that $M$ is an oriented valuated matroid if and only if each $M^\bu$ is an oriented matroid for $\bu \in \Gamma^E$.
This result can be deduced from~\cite{Celaya+Loho+Yuen:2024}, where it is proved in the language of polyhedral subdivisions rather than initial matroids.

Working over arbitrary tracts introduces technicalities where we must differentiate between weak and strong matroids.
Theorem~\ref{thm:A} is restricted to \emph{perfect tracts}, tracts where these notions coincide and so we do not have to worry about such technicalities.
We give analogous results over non-perfect tracts in Theorem~\ref{thm:matroid+subdivision} and Proposition~\ref{prop:ext+implies+initial}, along with a detailed example demonstrating failings over non-perfect tracts.

Jarra and Lorscheid extended the framework of matroids over tracts to flag matroids~\cite{Jarra+Lorscheid:2022}.
A \emph{flag $F$-matroid} $\bM = (M_1, \dots, M_k)$ is a tuple of $F$-matroids $M_i$ satisfying the `Pl\"ucker incidence relations' over $F$ (see Definition~\ref{def:flag+matroid}).
Key examples include flag $\KK$-matroids, which agree with the usual definition of flag matroids, and flag $\TT$-matroids, also known as valuated flag matroids.
The latter were studied in~\cite{Haque:2012} and \cite{Brandt+Eur+Zhang:2021}, where it was shown that $\bM$ is a valuated flag matroid if and only if all of its initial flag matroids $\bM^\bu$ are flag matroids.
Note that an \emph{initial flag matroid} $\bM^\bu = (M_1^\bu, \dots, M_k^\bu)$ is the tuple of initial matroids of each of its constituents.
Our second main theorem generalises this result (and extends Theorem~\ref{thm:A}) to flag matroids over tropical extensions of tracts.
%This is the content of Section~\ref{sec:flags}.

\begin{thm}\label{thm:B}
Let $F$ be a perfect tract and $\Gamma$ an ordered abelian group.
Then $\bM = (M_1, \dots, M_k)$ is a flag $F[\Gamma]$-matroid if and only if $\bM^\bu = (M_1^\bu, \dots, M_k^\bu)$ is a flag $F$-matroid for all $\bu \in (\Gamma \cup \{\infty\})^E$.
\end{thm}
%Unlike with Theorem~\ref{thm:A}, Over a non-perfect tract $F$, we can show the `only if' direction of this theorem, but emphasise that the `if' direction is not true in general.
Again in the special case where $F = \SS$ and $\Gamma = \RR$, we obtain that $\bM$ is an flag oriented valuated matroid if and only if each $\bM^\bu$ is a flag oriented matroid.
To our knowledge, this is a new result on oriented valuated matroids.

The proof of Theorem~\ref{thm:B} requires us to know far more about the initial matroids than just its Pl\"ucker vector.
As such, Section~\ref{sec:initial+circuits} characterises duality and circuits for initial matroids (Proposition~\ref{prop:initial+dual} and Proposition~\ref{prop:initial+circuits} respectively).
This culminates in characterising the (co)vectors of $F[\Gamma]$-matroids in terms of (co)vectors of its initial matroids (Proposition~\ref{prop:covectors}).

A \emph{positroid} is a matroid representable by a real matrix, all of whose maximal minors are non-negative.
It can equivalently be defined as a positively oriented matroid~\cite{Ardila+Rincon+Williams:2017}, an $\SS$-matroid $M$ whose Pl\"ucker vector $P_M$ is `non-negative', i.e. $P_M(I) \in \SS_{\geq \0} := \{\0, \1\}$ for all subsets $I$.
We use this notion to define \emph{positroids over ordered tracts}.
We say a tract $F$ is \emph{ordered} if there exists a distinguished subset $F_{\geq \0}$ of `non-negative' elements satisfying certain properties of positive numbers (Definition~\ref{def:ordering}).
Given such an ordered tract, an \emph{$F$-positroid} $M$ is an $F$-matroid whose Pl\"ucker vector $P_M$ is non-negative, i.e. $P_M(I) \in F_{\geq \0}$ for all subsets $I$.
When $F = \SS$, we recover the usual definition of positroid, and when $F = \TT_{\pm}$, we recover positive valuated matroids (or positive tropical Pl\"ucker vectors).
Similarly, $\bM = (M_1, \dots, M_k)$ is a \emph{flag $F$-positroid} if it is a flag $F$-matroid, all of whose constituents $M_i$ are $F$-positroids.

Utilising the groundwork we have laid with Theorems~\ref{thm:A} and~\ref{thm:B}, we obtain positroid variants of both theorems.
%Note that initial (flag) positroids are defined identically to initial (flag) matroids.

\begin{thm} \label{thm:C}
Let $F$ be an ordered perfect tract and $\Gamma$ an ordered abelian group.
\begin{enumerate}[label=(\roman*)]
\item $M$ is an $F[\Gamma]$-positroid if and only if $M^\bu$ is an $F$-positroid for all $\bu \in \Gamma^E$.
\item $\bM = (M_1, \dots, M_k)$ is a flag $F[\Gamma]$-positroid if and only if $\bM^\bu = (M_1^\bu, \dots, M_k^\bu)$ is a flag $F$-positroid for all $\bu \in (\Gamma \cup \{\infty\})^E$.
\end{enumerate}
\end{thm}

In the case where $F[\Gamma] = \SS[\RR] = \TT_\pm$, this theorem recovers a number of results regarding the totally non-negative Dressian.
The claim that $\TT_\pm$-positroids are precisely those oriented valuated matroids whose initial matroids are $\SS$-positroids was proved independently by~\cite{Arkani-Hamed+Lam+Spradlin:2020} and~\cite{Lukowski+Parisi+Williams:2020}.
The claim for flag $\TT_\pm$-positroids was can be deduced for full flags from~\cite{Joswig+Loho+Luber+Olarte:2023} and for partial flags from~\cite{Boretsky+Eur+Williams:2023}.

As an application of our results, we dedicate Section~\ref{sec:tropical+linear+spaces} to the study of \emph{enriched tropical linear spaces}.
These are sets $\cL_M \subseteq F[\Gamma]^E$, defined as those points orthogonal to the circuits of a $F[\Gamma]$-matroid $M$.
When $F=\KK$, these are ordinary tropical linear spaces, hence they can be considered as tropical linear spaces enriched with additional data coming from the tract $F$.
For example, when $F = \SS$ these are precisely \emph{signed tropical linear spaces}, linear spaces over the signed tropical hyperfield $\TT_\pm$.

Utilising our previous work and known cryptomorphisms of matroids over tracts, we obtain a structure theorem for enriched tropical linear spaces.
This theorem exactly recovers the characterisation of projective tropical linear spaces of Brandt, Eur and Zhang \cite[Theorem B (i)-(iv)]{Brandt+Eur+Zhang:2021} in the case where $F = \KK$.
Moreover, it gives a number of new characterisations of signed tropical linear spaces in the case where $F = \SS$.
\begin{thm} \label{thm:D}
Let $M$ be an $F[\Gamma]$-matroid on $E$ where $F$ is a perfect tract.
The enriched tropical linear space $\cL_M$ corresponding to $M$ is equal to any of the following sets in $F[\Gamma]^E$:
\begin{enumerate}[label=(\roman*)]
\item $\bigcup\limits_{\emptyset \subseteq A \subseteq E} \cL_{M/A}^\circ \times \{\infty\}^A \, ,$ the set of tuples contained in the finite part of a contraction of $\cL_M$;
\item $\cC(M)^\perp \, ,$ the set of tuples orthogonal to the circuits $\cC(M)$ of $M$;
\item $\left\{X \in F[\Gamma]^E \mid \theta(X) \in \cV^*(M^{|X|}) \right\} \, ,$ the set of tuples whose initial $F$-matroid $M^{|X|}$ contains the $F$-part $\theta(X)$ of $X$ as a covector;
\item $\bigcap\limits_{B \in \cB(\underline{M})} {\rm{span}}(\cC^*_B(M)) \, ,$ the set of tuples in the span of all fundamental cocircuits of $M$.
\end{enumerate}
\end{thm}

To justify our definition, we also consider `tropicalisations' of linear spaces over a field $\FF$ via \emph{enriched valuations}.
These are homomorphisms $\nu \colon \FF \rightarrow F[\Gamma]$ to a tropical extension that can be considered as a valuation map that records additional $F$-data about field elements.
The prototypical example is the signed valuation $\sval \colon \FF \rightarrow \TT_\pm$ that records the sign of an element.
However, there are many other natural examples including the \emph{fine valuation} $\fval\colon \pseries{\FF} \rightarrow \FF[\RR]$ that maps a Puiseux series to its leading term.
We show that for any linear space $L$ over $\FF$ and enriched valuation $\nu \colon \FF \rightarrow F[\Gamma]$, the image $\nu(L)$ is naturally contained in an enriched tropical linear space.
Moreover, we show that for a number of key examples of enriched valuations, the image $\nu(L)$ is precisely an enriched tropical linear space (Proposition~\ref{prop:tropicalisation}).

We conclude with a number of extended examples of enriched tropical linear spaces, and how to view and compute them.

\subsection{Motivation and related work}\label{sec:related+work}

\paragraph{Valuated matroids and their initial matroids}
Valuated matroids were introduced by Dress and Wenzel~\cite{Dress+Wenzel:1992-valuated} as an abstraction of matroids representable over a valued field, and have applications in a number of fields including optimisation~\cite{Murota:2003} and economics~\cite{Ostrovsky+PaesLeme:2015}.
They have received significant attention within tropical geometry due to their equivalence with tropical linear spaces~\cite{Speyer:2008,Rincon:2013}.
In particular, the valuated basis exchange axiom is exactly the Pl\"ucker relations over the tropical hyperfield, defining a moduli space of valuated matroids known as the \emph{Dressian}.
Understanding and computing the Dressian and its fan structure is a major area of research within tropical geometry~\cite{Herrmann+Joswig+Speyer:2014, Olarte+Panizzut+Schroter:2019, Brandt+Speyer:2022, Bendle+Bohm+Ren+Schroter:2024}.
Valuated matroids have also received significant attention from tropical scheme theory, as they are central to the definition of tropical ideals~\cite{Maclagan+Rincon:2018}.

A useful property of valuated matroids is that they can be understood globally by their local pieces known as initial matroids.
As discussed previously, $M$ is a valuated matroid if and only if all of its initial matroids $M^\bu$ are matroids.
This paradigm can be extended to various families of valuated matroids.
Motivated by the study of Stiefel tropical linear spaces initiated in~\cite{Fink+Rincon:2015}, Fink and Olarte showed that $M$ is a transversal valuated matroid if and only if all of its initial matroids $M^\bu$ are transversal matroids~\cite{Fink+Olarte:2022}.
In their study of the positive tropical Grassmannian, it was shown independently by~\cite{Arkani-Hamed+Lam+Spradlin:2020} and~\cite{Lukowski+Parisi+Williams:2020} that $M$ is a valuated positroid if and only if each of its initial matroids are positroids.
This paradigm has also been extended to flag matroids, where~\cite{Haque:2012} and~\cite{Brandt+Eur+Zhang:2021} showed that $M$ is a valuated flag matroid if and only if each of its initial flag matroids are flag matroids.
This has also been generalised recently to linear degenerate valuated flag matroids~\cite{Borzi+Schleis:2023}. 
These results demonstrate that initial matroids are powerful classifiers of families of valuated (flag) matroids.

\paragraph{Extensions of tropical geometry}
Tropical geometry is the study of images of varieties under a non-archimedean valuation.
Equivalently, it can be viewed as the geometry of varieties over the tropical semiring $(\RR \cup \{\infty\}, \min, +)$.
As the theory of tropical geometry has progressed, there have been attempts to work over richer structures than the tropical semiring.
One such approach is to replace semirings with \emph{hyperfields}, field-like objects whose addition operation may be multi-valued, as pioneered in the articles of Viro~\cite{Viro:2010,Viro:2011}.
In particular, the tropical semiring can be enriched to the \emph{tropical hyperfield} by replacing $\min$ with the hyperoperation
\[
a \boxplus b = \begin{cases} \min(a,b) & a \neq b \\ [a, \infty] & a = b \end{cases} \, .
\]
This makes a number of facets of tropical geometry more natural: tropical varieties are now genuine algebraic varieties over $\TT$, and valuation maps are precisely hyperfield homomorphisms to $\TT$.
Further benefits of the hyperfield structure are given in~\cite{Lorscheid:2022} as an approach to tropical scheme theory.

Hyperfields have the additional benefit that they give a flexible framework to work with finer maps than a non-archimedean valuation, or equivalently to work over tropical hyperfields with additional data.
The most notable of these is the \emph{signed valuation} of an ordered valued field, which records the sign of an element in addition to its valuation.
Studying images of objects in this map results in geometry over the signed (or real) tropical hyperfield $\TT_\pm$,
with close ties to real algebraic geometry via Mazlov dequantization and Viro's patchworking method~\cite{Viro:1984}. 
Such objects studied this way includes real algebraic varieties~\cite{Tabera:2015,Rau+Renaudineau+Shaw:2023} and semialgebraic sets~\cite{Allamigeon+Gaubert+Skomra:2020, Jell+Scheiderer+Yu:2022, Loho+Skomra:2024}.
In particular, there has been significant contributions to the study of signed tropical linear spaces~\cite{Jurgens:2018, Celaya:2019, Rau+Renaudineau+Shaw:2022, Loho+Skomra:2024,Celaya+Loho+Yuen:2024}, which we review in detail in Remark~\ref{rem:signed+survey}.
Other examples of richer valuation maps includes the complexified valuation map which records both the valuation and the phase of an element.
This was used by Mikhalkin to enumerate curves in toric surfaces~\cite{Mikhalkin:2005}, and motivated a complex analogue of Mazlov dequantization introduction by Viro~\cite{Viro:2011}.
One can also consider a fine valuation map on the field of Puiseux series that records the leading term rather than just the leading power: such maps are implicitly used in constructing initial solutions in polyhedral homotopy theory~\cite{Leykin+Yu:2019}.
These valuation-like maps were dubbed \emph{enriched valuations} in~\cite{Maxwell+Smith:2023}, defined as homomorphisms to tropical extensions of hyperfields.

The tropical extension of a tract is not a novel concept: the construction had been defined and studied for a number of algebraic structures, often under the guise of \emph{layering}.
For single-valued structures, it has been used to study idempotent semirings~\cite{Akian+Gaubert+Guterman:2009, Akian+Gaubert+Guterman:2014, Izhakian+Knebusch+Rowen:2014}, motivated by methods from tropical linear algebra and solving polynomial systems.
These applications have been extended recently to \emph{semiring systems} as a bridge between semirings and hyperstructures~\cite{Rowen:2021, Akian+Gaubert+Rowen:2022, Akian+Gaubert+Tavakolipour:2023}.
For multivalued structures, tropical extensions of hyperfields were defined in~\cite{Bowler+Su:2021} to classify various families of hyperfields, including stringent and doubly distributive hyperfields.
These classifications were then used in~\cite{Bowler+Pendavingh:2019} to prove such hyperfields were perfect.
The construction was also defined in~\cite{Gunn:2022} for \emph{idylls}, tracts whose null set form an ideal of the group semiring.
This was used to extend results from~\cite{Baker+Lorscheid:2021-polynomial} on  multiplicities of roots of polynomials over $\KK, \SS$ and $\TT$ to all stringent hyperfields and beyond.

\paragraph{Matroids over tracts}
Matroids over tracts were introduced in the articles of Baker and Bowler~\cite{Baker+Bowler:2017,Baker+Bowler:2019} with heavy inspiration from Anderson and Delucchi's work on complex matroids~\cite{Anderson+Delucchi:2012}.
This has lead to an explosion of developments in the past five years, of which the following is a subset.
A cryptomorphic axiom system in terms of (co)vectors was developed for arbitrary tracts by Anderson~\cite{Anderson:2019}, with simplified axioms via (co)vector composition developed for stringent hyperfields by Bowler and Pendavingh \cite{Bowler+Pendavingh:2019}.
The initial advances in matroid representability via partial fields~\cite{Pendavingh+vanZwam:2010} have been developed further in this framework~\cite{Baker+Lorscheid:2021-lift}, along with further progress via the introduction of foundations of matroids~\cite{Baker+Lorscheid:2024}. 
Moduli spaces of $F$-matroids were first approached topologically in~\cite{Anderson+Davis:2019}, and then scheme-theoretically in~\cite{Baker+Lorscheid:2021-moduli}.
The framework of matroids over tracts has been since generalised to both flag matroids~\cite{Jarra+Lorscheid:2022} and orthogonal matroids~\cite{Jin+Kim:2023, Baker+Jin:2023}.

\subsection{Structure of the paper}

In Section~\ref{sec:preliminaries}, we review the necessary preliminaries of tracts and matroids over them.
The majority of this material is a retread of known concepts, but we do introduce one new operation, namely extension by (co)loops.

In Section~\ref{sec:trop+ext+matroids}, we study matroids over tropical extensions of tracts.
After introducing initial matroids, we prove Theorem~\ref{thm:A} for perfect tracts and its variants for non-perfect tracts, namely for strong matroids (Proposition~\ref{prop:ext+implies+initial}) and weak matroids (Theorem~\ref{thm:matroid+subdivision}).
In addition, we give a detailed example demonstrating there exists a weak $F[\Gamma]$-matroid whose initial $F$-matroids are all strong (Example~\ref{ex:strong+counterexample}).
In Section~\ref{sec:initial+circuits}, we study the structure of initial matroids in more detail.
We describe their duals (Proposition~\ref{prop:initial+dual}) and circuits (Proposition~\ref{prop:initial+circuits}).
Utilising these characterisations, we end with a description of the covectors of matroids over tropical extensions via initial matroids (Proposition~\ref{prop:covectors}).

In Section~\ref{sec:flags}, we study flag matroids over tropical extensions of tracts.
Using the structure of initial matroids developed in the previous section, we prove Theorem~\ref{thm:B} characterising flag matroids via their initial flag matroids.
In Section~\ref{sec:positroids}, we introduce the notion of a positroid over an ordered tract.
Our tract setup gives us a quick proof of Theorem~\ref{thm:C}, a characterisation of positroids over tropical extensions of ordered tracts in terms of their initial matroids.

In Section~\ref{sec:tropical+linear+spaces}, we finish by applying the theory developed in the previous sections to tropical linear spaces.
In Section~\ref{sec:tropical+geometry}, we first review some necessary preliminaries of (extensions of) tropical geometry and enriched valuations.
We then define enriched tropical linear spaces in Section~\ref{sec:enriched+tls} via matroids over tropical extensions.
We first prove Theorem~\ref{thm:D}, four different characterisations of enriched tropical linear spaces.
We then motivate our definition further by showing that the tropicalisation of a linear space is always contained in an enriched tropical linear space, with equality in multiple key examples (Proposition~\ref{prop:tropicalisation}).
We finish with a number of extended examples computing enriched tropical linear spaces in Section~\ref{sec:enriched+examples}.

\section*{Acknowledgements}
Thanks to Alex Fink, Jorge Alberto Olarte, Victoria Schleis and Chi Ho Yuen for helpful conversations, and Georg Loho and James Maxwell for useful comments and feedback on an earlier draft.
We also thank the organisers and speakers of the MaTroCom workshop at Queen Mary University of London, where discussions prompted this project.
The author is supported by EPSRC grant EP/X036723/1.

\section{Preliminaries} \label{sec:preliminaries}

We begin by recalling the necessary preliminaries for tracts and matroids over them.
Much of the material given can be found in~\cite{Baker+Bowler:2019, Anderson:2019,Jarra+Lorscheid:2022}.

\subsection{Tracts}

\begin{definition}
A \emph{tract} $F = (F^\times, N_F)$ is a multiplicative abelian group $(F^\times,\cdot)$ with a subset $N_F$ of the group semiring $\NN[F^\times]$, called the \emph{null set}, satisfying all of the following:
\begin{enumerate}[label=(T\arabic*)]
    \item the zero element of $\NN[F^\times]$, denoted $\0$, belongs to $N_F$,
    \item the identity element of $F^\times$, denoted $\1$, is not in $N_F$,
    \item there is a unique element $-\1$ such that $\1 + (-\1) \in N_F$,
    \item if $\sum a_i \in N_F$, then $\sum c \cdot a_i \in N_F$ for any $c \in F^\times$ .
%    $N_F$ is closed under the natural action of $F^\times$ on $\NN[F^\times]$. \todo{Write this out}
\end{enumerate}
\end{definition}
We will often refer to the tract as the underlying set $F = F^\times \cup \{\0\}$, where multiplication can be extended to $F$ by setting $\0 \cdot g = \0$.
The multiplicative group $F^\times$ can be thought of as the multiplicative units of a `field'.
While tracts do not have a formally defined addition, one should think of the null set as summations `equivalent' to $\0$.
%We will also usually refer to $\epsilon$ as $-\1$ as it is uniquely defined.

\begin{example}[Hyperfields]
Our main source of examples will be \emph{hyperfields} $(\HH, \boxplus, \odot)$, field-like objects with multi-valued addition.
The hyperaddition on $\HH$ is a commutative binary operation $\boxplus\colon \HH \times \HH \rightarrow 2^\HH \setminus \emptyset$ that maps $(a,b)$ to a non-empty set $a\boxplus b$.
This hyperaddition is extended to strings of elements recursively via
\[
a_1 \boxplus \cdots \boxplus a_m = \bigcup_{a \in a_1 \boxplus \cdots \boxplus a_{m-1}} a \boxplus a_m \, ,
\]
and is required to be associative.
The hyperfield $\HH$ then satisfies the usual axioms of a field with an alteration to the additive inverses axiom, and an additional reversibility axiom:
\begin{itemize}
\item (Inverses) For all $a \in \HH^\times$, there exists a unique $(-a) \in \HH^\times$ such that $\0 \in a \boxplus -a$,
\item (Reversibility) $a \in b \boxplus c$ if and only if $c \in a \boxplus -b$ for all $a,b,c \in \HH$.
\end{itemize}
Hyperfields can be viewed as a tract by taking $F^\times = (\HH^\times, \odot)$ and setting the null set to be
\[
N_{\HH} = \SetOf{\sum_{i \in I} a_i \in \NN[\HH^\times]}{\0 \in \bigboxplus_{i \in I} a_i} \, .
\]
As the null set completely determines $\boxplus$, the tract structure and hyperfield structure on $\HH$ are equivalent.
However, the latter is frequently easier to work with in explicit examples and so we will often do so.
\end{example}

\begin{example}
The \emph{Krasner hyperfield} is the set $\KK := \{\0,\1\}$ with the standard multiplication and hyperaddition defined as
\[
\0 \boxplus \0 = \0 \quad , \quad \0 \boxplus \1 = \1 \boxplus \0 = \1 \quad , \quad \1 \boxplus \1 = \{ \0 , \1\} \, .
\]
As a tract, its abelian group is the trivial group $\KK^\times = \langle \1\rangle$ with null set
\[
N_\KK = \{\0, \1+\1, \1+\1+\1, \dots\} = \SetOf{k\cdot \1}{k \geq 2} \cup \{\0\} \, .
\]
\end{example}
\begin{example}
The \emph{sign hyperfield} is the set $\SS :=\{\1,\0,-\1\}$ with standard multiplication and hyperaddition defined as
\[
\0 \boxplus a = a \quad , \quad a \boxplus a = a \quad , \quad \1 \boxplus -\1 = \{\1,\0,-\1\} \quad \forall a \in \SS .
\]
As a tract, its abelian group is $\SS^\times = \langle +\1, -\1\rangle \cong C_2$ with null set
\[
N_\SS = \SetOf{k\cdot \1 + \ell\cdot(-\1)}{k, \ell \geq 1} \cup \{\0\} \, .
\]
\end{example}
\begin{example}
The \emph{tropical hyperfield} is the set $\TT := \RR \cup \{\infty\}$ with hyperfield operations
\begin{align*}
\gamma \boxplus \eta &= 
\begin{cases}
\min(\gamma,\eta)  &\text{if} \quad \gamma \neq \eta\\
\{ \gamma' \in \TT \, | \, \gamma' \geq \gamma\} &\text{if} \quad \gamma = \eta
\end{cases} \, , \\
\gamma \odot \eta &= \gamma + \eta \, .
\end{align*}
As a tract, its abelian group is $\TT^\times \cong (\RR, +)$ with null set
\[
N_\TT = \SetOf{\sum_{i \in I} \gamma_i}{\exists i, j \in I \text{ s.t. } \gamma_i = \gamma_j \leq \gamma_k \forall k \in I} \, ,
\]
i.e. $\gamma_1 + \cdots + \gamma_k \in N_\TT$ if and only if the minimum of the $\gamma_i$'s is attained at least twice.
Note that the additive and multiplicative identity elements are $\0 = \infty$ and $\1 = 0$.
\end{example}

\begin{example}\label{ex:phase}
The \emph{phase hyperfield} is the set $\Theta := S^1 \cup \{0\}$ where $S^1 = \SetOf{z \in \CC}{|z| = 1}$ with hyperfield with operations
\begin{align*}
    z_1 \odot z_2 &= z_1\cdot z_2 \\
    z_1 \boxplus z_2 &= \begin{cases}
        z_1 & z_1 = z_2 \\
        \{z_1, 0, -z_1\} & z_1 = -z_2 \\
        \text{shortest open arc between } z_1, z_2 & \text{otherwise}
    \end{cases}
\end{align*}
As a tract, its abelian group is $S^1$ with null set
\[
N_\Theta = \SetOf{\sum_{i \in I} z_i}{\exists w_i \in \CC^\times \text{ s.t } \sum_{i \in I} w_i = 0_\CC \text{ and } \arg(w_i) = z_i} \, .
\]
\end{example}

\begin{example}[Partial fields]
A \emph{partial field} $P = (G,R)$ consists of a commutative ring $R$ with multiplicative identity $\1$ and a subgroup $G \subseteq R^\times$ such that $-\1 \in G$ and $G$ generates the ring $R$.
The tract corresponding to $P$ has $G$ as its abelian group, with null set
\[
N_P = \SetOf{\sum k_i a_i \in \NN[G]}{\sum k_i a_i = 0 \text{ in } R} \, .
\]

Key examples partial fields include
\begin{itemize}
\item The \emph{regular partial field} $\UU_0 := (\{\pm 1\}, \ZZ)$,
\item The \emph{dyadic partial field} $\DD := (\langle -1, 2\rangle, \ZZ[\frac{1}{2}])$.
\end{itemize}
\end{example}

Our key construction for creating `tropical-like' tracts is as follows.

\begin{definition}\label{def:tropical+extension}
    Let $F$ be a tract and $\Gamma$ an ordered abelian group.
    The \emph{tropical extension} $F[\Gamma]$ of $F$ by $\Gamma$ is the tract with multiplicative abelian group $F^\times \times \Gamma$, and null set
    \[
    N_{F[\Gamma]} = \left\{ \sum_{i\in I} (a_i, \gamma_i) \, \Big| \, \sum_{i \in I_{\min}} a_i \in N_F \text{ where } I_{\min} = \{i\in I \mid \gamma_i \text{ is minimal}\}\right\} \cup \{\infty\} \, ,
    \]
    where $\infty$ is label we give to the zero element of $F[\Gamma]$.
    %It will also be convenient to enrich $\Gamma$ with the maximal absorbing element $\infty$.
   % We can consider the zero element in $F[\Gamma]$ as $(0, \infty)$ where $0 \in F$ the zero element of $F$. \todo{Try to avoid this!!}
\end{definition}

\begin{example}\label{ex:trop-hyp-trop-ext}
    The tropical hyperfield can be realised as the tropical extension $\TT \cong \KK[\RR]$ where $(\RR, +)$ is viewed as an ordered abelian group.
    This isomorphism is given by identifying $\gamma \in \TT$ with $(\1, \gamma) \in \KK[\RR]$.
    We can extend this correspondence to the \emph{rank k tropical hyperfield} $\TT^{(k)} \cong \KK[\RR^{(k)}]$ where $(\RR^{(k)}, +)$ is $k$-tuples of reals ordered lexicographically.
    
    Given an arbitrary ordered abelian group $(\Gamma,+)$, the tract $\KK[\Gamma]$ can be similarly identified with a `tropical' hyperfield whose underlying set is $\Gamma \cup \{\infty\}$.
    Precisely, Hahn's embedding theorem states we can embed any ordered abelian group $\Gamma$ into $(\RR^k, +)$ ordered lexicographically for sufficiently large $k$.
    As such, we can view $\KK[\Gamma]$ as a subhyperfield of a sufficiently high rank tropical hyperfield.
\end{example}

\begin{example}\label{ex:sgn-trop-hyp-trop-ext}
The \emph{signed tropical hyperfield} is the set $\TT_\pm:= (\{\pm 1\} \times \RR) \cup \{\infty\}$, where $\{-1\} \times \RR$ is a `negative' copy of the tropical numbers.
Its hyperfield operations are
\begin{align*}
    (s_1,\gamma_1) \boxplus (s_2,\gamma_2) &=
    \begin{cases}
        (s_1,\gamma_1), & \text{if} \, \gamma_1 < \gamma_2,\\
        (s_2,\gamma_2), & \text{if} \, \gamma_2 < \gamma_1,\\
        (s_1,\gamma_1), & \text{if} \, s_1 = s_2, \, \text{and} \, \gamma_1 = \gamma_2,\\
        \{(\pm 1, \eta) \mid \eta \geq \gamma_1\} \cup \{\infty\}, & \text{if} \, s_1 = -s_2, \, \text{and} \, \gamma_1 = \gamma_2 ,
    \end{cases} \\
    (s_1,\gamma_1) \odot (s_2,\gamma_2) &= (s_1 \cdot s_2, \gamma_1 + \gamma_2) \, .
\end{align*}
    The signed tropical hyperfield can be viewed as the tropical extension $\TT_\pm \cong \SS[\RR]$.
    As such, its null set is
    \[
N_{\SS[\RR]} = \SetOf{\sum_{i\in I}(s_i, \gamma_i)}{\sum_{i \in I_{\min}} s_i = k\cdot 1 + \ell\cdot(-1) \, , \, k, \ell \geq 1} \cup \{\infty\} \, .
\]
\end{example}

Examples coming from extending hyperfields have already been studied in a number of guises.
However, there are new interesting examples we can view by moving to the category of tracts.

\begin{example}
The \emph{regular tropical tract} $\UU_0[\RR]$ has the same underlying abelian group as the signed tropical hyperfield, but a different null set.
Namely, it is the abelian group $\langle \pm 1 \rangle \times \RR$ with null set
\[
N_{\UU_0[\RR]} = \SetOf{\sum_{I \in I}(a_i, \gamma_i)}{\sum_{i \in I_{\min}} a_i = k\cdot 1 + k\cdot(-1) \, , \, k \geq 1} \cup \{\infty\} \, .
\]
\end{example}

\begin{definition}
    A homomorphism of tracts is a group homomorphism $f\colon F_1^\times \rightarrow F_2^\times$ such that the induced map $\tilde{f}$
    \begin{align*}
\tilde{f}\colon \NN[F_1^\times]&\rightarrow\NN[F_2^\times] \\
\sum k_ig_i &\mapsto \sum k_if(g_i)
    \end{align*}
    satisfies $\Tilde{f}(N_{F_1}) \subseteq N_{F_2}$.
    We extend $f$ to a map from $F_1$ to $F_2$ by setting $f(\0_{F_1}) = \0_{F_2}$.
\end{definition}

\begin{example}\label{ex:trivial+hom}
    Every tract $F$ has a trivial homomorphism $t \colon F \rightarrow \KK$ to the Krasner hyperfield, given by sending $g \to \1$ for all $g \in F^\times$.
\end{example}

\begin{example}\label{ex:maps}
There are a number of maps that will be extremely useful throughout when dealing with tropical extensions.
These are the `modulus' map $|\cdot|$ that forgets the $F$-part of an element, and the `phase' map $\theta$ that forgets the $\Gamma$-part of an element.
These are defined respectively as
\begin{align*}
|\cdot| \colon F[\Gamma] &\longrightarrow \KK[\Gamma] & \theta \colon F[\Gamma] &\longrightarrow F \\
(a, \gamma) &\longmapsto (\1, \gamma) & (a, \gamma) &\longmapsto a \\
\infty &\longmapsto \infty & \infty &\longmapsto \0 \, .
\end{align*}
It is straightforward to check that $|\cdot|$ is homomorphism of tracts.
However, $\theta$ is not a homomorphism: it commutes with multiplication, but may not preserve the null set, as $\sum_I (a_i, \gamma_i) \in N_{F[\Gamma]}$ does not imply $\sum_I a_i \in N_F$.
However, if restrict the null set to sums where $I_{\min} = I$, then this does commute: we will use this to our advantage later.
\end{example}
%The maps from Examples~\ref{ex:trivial+hom} and~\ref{ex:maps} are related by the following diagram:
%\begin{equation}\label{eq:diagram}
%\begin{tikzcd}
%F[\Gamma] \arrow[d, "\theta_F"] \arrow[r, "|\,\cdot|"] & \KK[\Gamma] \arrow[d, "\theta_\KK"] \\
%F \arrow[r, "t"] & \KK
%\end{tikzcd}
%\end{equation}

Finally, we note that our tracts will require an involution for certain constructions later.
An \emph{involution} of a tract is a homomorphism $\tau \colon F \rightarrow F$ such that $\tau^2(a) = a$ for all $a \in F$.
Natural examples of involutions are the identity and complex conjugation.
We will assume a tract $F$ comes with some fixed involution $\tau$ from now on, and write $\overline{X} = \tau(X)$.

%If $f\colon F \rightarrow F'$ is a homomorphism of tracts, there is an induced homomorphism $f_\Gamma \colon F[\Gamma] \rightarrow F'[\Gamma]$ between its tropical extensions.
%There is also a natural forgetful morphism $\pi_F\colon F[\Gamma] \rightarrow F$.
%These naturally commute as shown in Figure~\ref{}.
%Moreover, we will be particularly interested in the case when $f$ is the trivial homomorphism to the Krasner hyperfield.
%\begin{figure}
%\centering
%\begin{tikzcd}
%F[\Gamma] \arrow[r, "f_\Gamma"] \arrow[d,"\pi_F"] & F'[\Gamma] \arrow[d,"\pi_{F'}"] \\
%F \arrow[r, "f"] & F'
%\end{tikzcd}
%\quad\quad\quad\quad
%\begin{tikzcd}
%F[\Gamma] \arrow[r, "t_\Gamma"] \arrow[d,"\pi_F"] & \KK[\Gamma] \arrow[d,"\pi_{\KK}"] \\
%F \arrow[r, "t"] & \KK
%\end{tikzcd}
%\end{figure}

\subsection{Matroids over tracts} \label{sec:matroids}

Let $F$ be a tract, $E$ a set and $1 \leq r \leq |E|$ a positive integer.
A \emph{(strong) Grassmann-Pl\"ucker function} of rank $r$ on ground set $E$ is a function $\phi:E^r \rightarrow F$ satisfying:
\begin{enumerate}[label=(GP\arabic*)]
\item \label{eq:GP1}$\phi$ is not the zero function,
\item \label{eq:GP2}$\phi$ is alternating, i.e. $\phi(i_1,\shortdots, i_a, \shortdots, i_b, \shortdots, i_r) = -\phi(i_1, \shortdots, i_b, \shortdots, i_a, \shortdots, i_r)$, and $\phi(i_1, \dots, i_r) = 0$ if $i_a = i_b$ for some $a \neq b$,
\item \label{eq:GP3}For any two ordered subsets $\{i_1, \dots, i_{r-1}\}$ and $\{j_1, \dots, j_{r+1}\}$ of $E$
\begin{equation} \label{eq:strong+GP}
\sum_{k=1}^{r+1} (-1)^k \phi(j_k, i_1, \dots, i_{r-1}) \cdot \phi(j_1, \dots,\hat{j}_k, \dots, j_{r+1}) \in N_F \, .
\end{equation}
\end{enumerate}
A \emph{weak Grassmann-Pl\"ucker function} satisfies \ref{eq:GP1}, \ref{eq:GP2} and the following weaker version of \ref{eq:GP3}:
\begin{enumerate}
\item[(GP3')] For any two subsets $I = \{i_1, \dots, i_{r-1}\}$ and $J = \{j_1, \dots, j_{r+1}\}$ with $|J \setminus I| = 3$, $\phi$ satisfies \eqref{eq:strong+GP}.
\end{enumerate}

As $\phi$ is alternating, we only need to record the value it takes on one ordering of a subset to deduce the values on all other orderings.
%It will generally be more convenient for us to consider $F$-matroids in terms of \emph{Pl\"ucker vectors}.
We define its associated \emph{Pl\"ucker vector} as the non-zero function $P_\phi \colon {E \choose r} \rightarrow F$ defined by
\[
P_\phi(I) = \phi(i_1, \dots, i_r) \, , \quad I = \{i_1 < i_2 < \dots < i_r\} \, .
\]
If $\phi$ is a strong Grassmann-Pl\"ucker function, then \ref{eq:GP3} states that $P_\phi$ must satisfy the \emph{Pl\"ucker relations} $\cP^r_{I;J}$ for all $I \in {E \choose r-1}, J \in {E \choose r+1}$,
\begin{equation} \label{eq:plucker+relation}
\cP^r_{I;J} \, \colon \, \sum_{j \in J} \sign(j;I,J) \cdot P(I+j)\cdot P(J-j) \in N_F
\end{equation}
where $\sign(j;I,J) = (-1)^\ell$ with $\ell$ equal to the number of elements $j' \in J$ with $j< j'$ plus the number of elements $i \in I$ with $j<i$.
In this case, we call the $P_\phi$ strong.

If $\phi$ is a weak Grassmann-Pl\"ucker function, then (GP3') states that $P_\phi$  must satisfy the \emph{three-term Pl\"ucker relations}, i.e. $\cP^r_{I;J}$ where $|J \setminus I| = 3$.
In this case we call $P_\phi$ weak.
As $P_\phi$ completely determines $\phi$, we will mostly drop Grassmann-Pl\"ucker functions and work entirely with Pl\"ucker vectors from now on.

\begin{definition}
We say two Pl\"ucker vectors $P_1, P_2$ are \emph{equivalent} if there exists some $\alpha \in F^\times$ such that $P_1 = \alpha \cdot P_2$.
A \emph{weak (resp. strong) $F$-matroid} $M$ of rank $r$ on ground set $E$ is an equivalence class of weak (resp. strong) Pl\"ucker vectors $M = [P_M \colon {E \choose r} \rightarrow F]$.
\end{definition}

%Let $F$ be a tract and denote the set of $k$-subsets of $[n] = \{1,\dots, n\}$ by ${n \choose k}$.
%Given $0< r < n$ and subsets $I \in {n \choose r-1}$ and $J \in {n \choose r+1}$, the \emph{Pl\"ucker relation} $\cP^r_{I;J}$ is defined to be the function 
%\begin{equation}\label{eq:PR}
%\begin{split}
%\cP^r_{I;J} \colon F^{n \choose r} &\longrightarrow \NN[F^\times] \\
%P &\longmapsto \sum_{j \in J} \sign(j;I,J) \cdot p_{I+j}\cdot p_{J-j}
%\end{split}
%\end{equation}
%where $\sign(j;I,J) = (-1)^\ell$ with $\ell$ equal to the number of elements $j' \in J$ with $j< j'$ plus the number of elements $i \in I$ with $j<i$.
%If we wish to talk about $\cP^r_{I;J}$ abstractly without evaluating, we may use the expression $\sum_{j \in J} \sign(j;I,J) X_{I+j}\cdot X_{J-j}$, where $X_S$ can be considered formal variables that we later `evaluate'.
%We will often drop the superscript $r$ if it is clear from context.
%
%%The tuple $P \in F^{n \choose r}$ is said to \emph{satisfy} the Pl\"ucker relation $\cP_{I,J}$ if $\cP_{I,J}(P)$ is contained in the null set $N_F$.
%\begin{definition}
%    A tuple $P \in F^{n \choose r}$ is a \emph{strong $F$-matroid} of rank $r$ if $\cP^r_{I,J}(P) \in N_F$ for all $I \in {n \choose r-1}$ and $J \in {n \choose r+1}$.
%\end{definition}

\begin{example}
$\KK$-matroids are precisely ordinary matroids in the following way.
We can consider $P \colon {E \choose r} \rightarrow \KK$ as an indicator function of a set system $\cB \subseteq {E\choose r}$, where $B \in \cB$ if and only if $P(B) = \1$.
    In this case, the Pl\"ucker relations \eqref{eq:plucker+relation} precisely give the strong basis exchange axiom for matroids.
\end{example}

\begin{example}
$\SS$-matroids are precisely oriented matroids, as a Grassmann-Pl\"ucker function $\phi\colon E^r \rightarrow \SS$ is exactly a \emph{chirotope}.
Explicitly, it is a non-zero, alternating function that satisfies the exchange axiom derived from \eqref{eq:strong+GP}, see~\cite{Björner+LasVergnas+Sturmfels+White+Ziegler:1999}.
\end{example}

\begin{example}
$\TT$-matroids are precisely valuated matroids as defined in~\cite{Dress+Wenzel:1992-valuated}.
It is immediate to see that the Pl\"ucker relations precisely give the valuated basis exchange axiom.
\end{example}

\begin{example}
Matroids over partial fields are closely related to questions of representability of matroids, see~\cite{Semple+Whittle:1996,Pendavingh+vanZwam:2010}.
For example, $\UU_0$-matroids are precisely regular matroids, matroids representable over any field.
Similarly, $\DD$-matroids are precisely dyadic matroids, matroids representable over every field of characteristic different from two.
\end{example}

In all of these examples, we did not differentiate between weak and strong matroids.
It turns out that $\KK, \SS, \TT$ and partial fields are all examples of perfect tracts, tracts over which every weak matroid is also a strong matroid~\cite[Section 3.13]{Baker+Bowler:2019}.
As a non-example, the phase hyperfield $\Theta$ is not perfect, and so there exist weak $\Theta$-matroids that are not strong.
We shall characterise perfect tracts at the end of the section when we have a few more notions to work with.

Every $F$-matroid has an ordinary matroid associated to it in the following way.
Given some tuple $X \in F^E$, its \emph{support} is the set $\underline{X} := \{i \in E \, | \, X_i \neq \0\}$.
Given an $F$-matroid $M$ with Pl\"ucker vector $P$ viewed as a tuple in $F^{E \choose r}$, its \emph{underlying matroid} $\underline{M}$ is the matroid with whose bases are the support of $P$:
\[
\cB(\underline{M}) := \underline{P} = \left\{B \in {E \choose r} \, \Big| \, P(B) \neq \0 \right\} \, .
\]
We will see that many axioms systems and operations for $F$-matroids will reduce to the usual matroid axioms and operations on their underlying matroids.

As with matroids, $F$-matroids have multiple cryptomorphic axiom systems.
While we will not actively define these equivalent axiom systems, we will utilise their alternative formulations.

We begin with the \emph{$F$-circuits} of $M$.
For each $\tau = \{i_0, \dots, i_r\} \in {E \choose r+1}$ with $i_0 < \cdots < i_r$, we define $C_\tau \in F^E$ to be
\begin{equation} \label{eq:circuits}
(C_\tau)_i := \begin{cases}
(-1)^s P(\tau - i_s) & i = i_s \\
\0 & \text{otherwise}
\end{cases}
\end{equation}
The \emph{$F$-circuits of $M$} are scalings of the non-zero $C_\tau$, i.e.
\[
\cC(M) = \SetOf{\alpha \cdot C_\tau \in F^E}{\alpha \in F^\times \, , \, \tau \in {E \choose r+1}} \setminus \{\underline{\0}\} \, ,
\]
where $\underline{\0} = (\0, \dots, \0) \in F^E$.
The circuits of the underlying matroid $\underline{M}$ are the supports of the $F$-circuits of $M$, i.e.
\[
\cC(\underline{M}) = \SetOf{\underline{C} \subseteq E}{C \in \cC(M)} \, .
\]
One property we can immediately deduce from this for $F$-circuits is that they have incomparable supports: if $C, D \in \cC(M)$ with $\underline{C} \subseteq \underline{D}$ then $\underline{C} = \underline{D}$.
This is in fact one of the axioms for $F$-circuits; see~\cite{Baker+Bowler:2019} for a full description. 

\begin{remark}
If $\tau$ does not contain a basis of $\underline{M}$, then $C_\tau$ is just the zero vector, and so we do not count it as an $F$-circuit.
If $\tau$ does contain a basis $B \in \cB(\underline{M})$, then $\tau$ contains a unique circuit of $\underline{M}$.
This is the fundamental circuit corresponding to $B$ and $e = \tau \setminus B$.
\end{remark}

We next define a notion of orthogonality over $F$.
Let $X,Y \in F^E$, the \emph{inner product} of $X$ and $Y$ is
\[
\iprod{X}{Y} = \sum_{i\in E} X_i \cdot \overline{Y_i} \in \NN[F^\times] \, .
\]
We say that $X$ and $Y$ are \emph{orthogonal} if $\iprod{X}{Y} \in N_F$, in which case we write $X \perp Y$.
Given a set $\cX \subseteq F^E$, we define the set of tuples orthogonal to all elements of $\cX$ by
\[
\cX^\perp = \SetOf{Y \in F^E}{X \perp Y \text{ for all } X \in \cX} \, .
\]
We define the \emph{$F$-covectors} of $M$ to be the set of tuples $\cV^*(M) = \cC(M)^\perp$ in $F^E$ that are orthogonal to all circuits of $M$.

We now define the notion of duality for matroids.
This will allow us to define $F$-cocircuits and $F$-vectors of a matroid.

\begin{definition}
Let $M$ be a strong (resp. weak) $F$-matroid of rank $r$ with Pl\"ucker vector $P \colon {E \choose r} \rightarrow F$.
The \emph{dual $F$-matroid} $M^*$ is the strong (resp. weak) $F$-matroid of rank $|E|-r$ with dual Pl\"ucker vector
\[
P^* \colon {E \choose |E|-r} \rightarrow F \, , \quad P^*(I) = \sign(I,I^c)\cdot\overline{P(I^c)}
\]
where $I^c = E \setminus I$ and $\sign(I,I^c)$ is the sign of $(I, I^c)$ viewed as a permutation.
\end{definition}

This definition follows from the Grassmann-Pl\"ucker characterisation in the following theorem.

\begin{theorem}[\cite{Baker+Bowler:2019}] \label{thm:dual+matroid}
Let $M$ be a strong (resp. weak) $F$-matroid with dual matroid $M^*$.
\begin{itemize}
\item The $F$-circuits of $M^*$ are the elements
\[
\cC(M^*) = \MinSupp(\cC(M)^\perp) \, ,
\]
where $\MinSupp(\cX)$ denotes the non-zero elements of $\cX$ of minimal support,
\item The Grassmann-Pl\"ucker function $\phi^*$ for $M^*$ is defined by the formula
\[
\phi^*(x_1, \dots, x_{n-r}) = \sign(x_1, \dots, x_{n-r}, x_1', \dots, x_r')\overline{\phi(x_1', \dots, x_r')} \, ,
\]
where $x_1', \dots, x_r'$ is any ordering of $E \setminus \{x_1, \dots, x_{n-r}\}$.
\item The underlying matroid of $M^*$ is the dual of the underlying matroid of $M$, i.e. $\underline{M^*} = \underline{M}^*$,
\item $M^{**} = M$.
\end{itemize}
\end{theorem}
The \emph{$F$-cocircuits} of $M$ are $\cC^*(M) = \cC(M^*)$, the circuits of the dual matroid $M^*$.
The \emph{$F$-vectors} of $M$ are the set of tuples  $\cV(M) = \cC^*(M)^\perp = \cV^*(M^*)$ orthogonal to all cocircuits of $M$.

It follows from these definitions and theorems that circuits are the vectors of minimal support, and cocircuits are the covectors of minimal support.
As such, we always have
\[
\cV(M)^\perp \subseteq \cC(M)^\perp = \cV^*(M) \quad , \quad \cV^*(M)^\perp \subseteq \cC^*(M)^\perp = \cV(M) \, ,
\]
but these inclusions may be strict.
They become equalities when the tract $F$ is perfect.

\begin{definition}[\cite{Baker+Bowler:2019, Dress+Wenzel:1992-perfect}]
A tract $F$ is \emph{perfect} if $\cV(M)^\perp = \cV^*(M)$ for every $F$-matroid $M$.
\end{definition}

Recall that over perfect tracts, every weak $F$-matroid is also a strong $F$-matroid, and so we do not have to distinguish.
Restricting to perfect tracts shall be necessary for some of our results later.

We next note that $F$-matroids are preserved under tract homomorphisms, as shown in \cite[Lemma 3.39]{Baker+Bowler:2019}.
\begin{definition} \label{def:matroid+hom}
Let $M$ be a weak (resp. strong) $F_1$-matroid with Pl\"ucker vector $P\colon {E \choose r} \rightarrow F_1$ and circuits $\cC(M) \subseteq F_1^E$.
Let $f\colon F_1 \rightarrow F_2$ be a homomorphism of tracts.
The \emph{push-forward} of $M$ is the weak (resp. strong) $F_2$-matroid denoted $f_*(M)$ with Pl\"ucker vector and circuits:
\begin{align*}
f_*(P)&\colon{E \choose r} \rightarrow F_2 \quad , \quad I \mapsto f(P(I)) \, , \\
\cC(f_*(M)) &= \SetOf{\alpha \cdot f(C) \in F_2^E}{C \in \cC(M) \, , \, \alpha \in F_2^\times} \, .
\end{align*}
\end{definition}

\begin{example}
Recall that every tract has a trivial homomorphism $t\colon F \rightarrow \KK$ to the Krasner hyperfield.
As $\KK = \{\0, \1\}$, we can view tuples over $\KK$ as indicator vectors of set systems.
%\[
%V \in \KK^E \, \longleftrightarrow \, S_V := \SetOf{i \in E}{V_i = \1} \, .
%\]
%In particular, for any tuple $X \in F^E$ its image $t(X)$ in the trivial homomorphism can be viewed as the indicator vector of its \emph{support} $\underline{X} := \{i \in E \, | \, X_i \neq \0\}$.
%
%We can do something similar with $F$-matroids.
Given an $F$-matroid $M$ with Pl\"ucker vector $P$,
its push-forward in the trivial homomorphism $t_*(M)$ is a $\KK$-matroid whose Pl\"ucker vector $t_*(P)$ can be viewed as an indicator vector of $\underline{P} = \cB(\underline{M})$.
As such, we will often identify $t_*(M)$ with the underlying matroid $\underline{M}$ of $M$.
\end{example}

\begin{example}
Consider the tropical extension $F[\Gamma]$, and recall that there is a tract homomorphism $|\cdot| \colon F[\Gamma] \rightarrow \KK[\Gamma]$ that sends $(a, \gamma)$ to $(\1, \gamma)$.
Given an $F[\Gamma]$-matroid $M$, we denote $|M|$ as the $\KK[\Gamma]$-matroid obtained as the push-forward under this homomorphism.
In particular, if $\Gamma = \RR$ then $|M|$ is a valuated matroid.
\end{example}

We recall a notion of minors for an $F$-matroid $M$.
For $X \in F^E$ and $A \subseteq E$, we define
\[
X \setminus A \in F^{E \setminus A} \quad , \quad  (X \setminus A)_j = X_j \quad \forall j \notin A \, .
\]
Given a subset $\cX \subseteq F^E$, we define
\[
\cX \setminus A = \SetOf{X \setminus A \in F^{E\setminus A}}{X \in \cX \, , \, \underline{X} \cap A = \emptyset} \quad , \quad
\cX / A = \MinSupp\left(\SetOf{X \setminus A \in F^{E\setminus A}}{X \in \cX}\right) \, .
\]
The existence of the following definition follows from~\cite[Theorem 3.29]{Baker+Bowler:2019}.
\begin{definition}\label{def:minor}
Let $M$ be a strong (resp. weak) $F$-matroid on $E$ and let $A \subseteq E$.
\begin{itemize}
\item $\cC(M)/A$ is the set of circuits of a strong (resp. weak) $F$-matroid $M / A$ on $E\setminus A$, called the \emph{contraction} of $M$ with respect to $A$, whose underlying matroid is $\underline{M}/ A$.
\item $\cC(M) \setminus A$ is the set of circuits of a strong (resp. weak) $F$-matroid $M \setminus A$ on $E\setminus A$, called the \emph{deletion} of $M$ with respect to $A$, whose underlying matroid is $\underline{M}\setminus A$.
\end{itemize}
Moreover, $(M\setminus A)^* = M^*/A$ and $(M / A)^* = M^*\setminus A$.
\end{definition}
The Pl\"ucker vectors of $M/A$ and $M \setminus A$ are rather more convoluted to describe.
As such, we will only use the circuit definition, but refer the interested reader to~\cite[Lemma 4.4]{Baker+Bowler:2019} for further details.

%We can also give the Pl\"ucker vectors of $M/A$ and $M\setminus A$ respectively, using~\cite[Lemma 4.4]{Baker+Bowler:2019}.
%Let $\ell$ be the rank of $A$ in $\underline{M}$, and $A' \subseteq A$ a maximal independent subset.
%Then the Pl\"ucker vector of $M/A$ is
%\[
%P_{M/A}(I) = P_M(I \sqcup A') \quad , \quad \forall I \subseteq {E \setminus A \choose r - \ell} \, .
%\]
%Let $k$ be the rank of $A^c$ in $\underline{M}$, and choose $B\subseteq A$ a basis for $\underline{M}/A^c$.
%Then the Pl\"ucker vector of $M \setminus A$ is
%\[
%P_{M \setminus A}(I) = P_M(I \sqcup B) \quad , \quad \forall I \subseteq {E \setminus A \choose r - k} \, .
%\]

We recall the notion of loops and coloops of an $F$-matroid as defined by \cite{Anderson:2019}.
\begin{definition}\label{defn:loop+coloop}
Let $M$ be an $F$-matroid on ground set $E$.
An element $e \in E$ of $M$ is a \emph{loop} (resp. \emph{coloop}) if it is a loop (resp. coloop) of the underlying matroid.
Equivalently,
\begin{itemize}
\item $e$ is a loop of $M$ if and only if $X_e = \0$ for every $X \in \cV^*(M)$,
\item $e$ is a coloop of $M$ if and only if $\cV^*(M) = \cV^*(M\setminus e) \times F$.
\end{itemize}
\end{definition}

We end by introducing the notion of extending an $F$-matroid by loops and coloops.
In the following, we write $\underline{\0}_A \in F^A$ for the all zeroes vector with coordinates in $A$, and $\chi_i \in F^E$ for the tuple $(\chi_i)_j = \delta_{ij}$ with $\0$ in every coordinate except $\1$ in the $i$-th coordinate.

\begin{proposition} \label{prop:loop+coloop}
Let $M$ be an $F$-matroid of rank $r$ on ground set $E$, and $A$ a set disjoint from $E$.
There exists a unique $F$-matroid $M \loopsum A$ of rank $r$ on ground set $E \sqcup A$ such that all elements of $A$ are loops, and $M = (M \loopsum A)\setminus A$.
Moreover, it has Pl\"ucker vector and circuits
\[
P_{M\loopsum A}(I) = \begin{cases}
P_M(I) & I \cap A = \emptyset \\ \0 & \text{otherwise}
\end{cases} \quad , \quad
\cC(M\loopsum A) = (\cC(M) \times \{\underline{\0}_A\}) \cup \SetOf{\alpha \cdot \chi_i \in F^{E \sqcup A}}{i \in A \, , \, \alpha \in F^\times} \, .
\]
There exists a unique $F$-matroid $M \coloopsum A$ of rank $r + |A|$ on ground set $E \sqcup A$ such that all elements of $A$ are coloops, and $M = (M \coloopsum A)\setminus A$.
Moreover, it has Pl\"ucker vector and circuits
\[
P_{M\coloopsum A}(I) = \begin{cases}
P_M(I\setminus A) & A \subseteq I \\ \0 & \text{otherwise}
\end{cases} \quad , \quad
\cC(M\coloopsum A) = (\cC(M) \times \{\underline{\0}_A\}) \, .
\]
Finally, we have $(M \loopsum A)^* = M^* \coloopsum A$.
\end{proposition}
\begin{proof}
Let $N$ be an $F$-matroid on $E \sqcup A$ such that all elements of $A$ are loops and $N \setminus A = M$.
The latter coupled with Definition~\ref{def:minor} implies that $\cC(M) \times \{\underline{\0}_A\}$ is a subset of the circuits of $N$.
As each $a \in A$ is a loop, $X_a = \0$ for all covectors $X \in \cV^*(N)$ and so $\chi_a$ and its scalings are contained in $\cV^*(N)^\perp$.
Moreover, they are support minimal and so contained in $\cC(N)$.
We show this gives the full list of circuits in $\cC(N)$.
For any circuit $C$, if $a \in \underline{C} \cap A$, then $\underline{\chi_a} \subseteq \underline{C}$ and hence $C = \alpha \cdot \chi_a$ for some scalar $\alpha \in F^\times$.
If $\underline{C} \cap A = \emptyset$, then $C \setminus A \in \cC(M)$ by Definition~\ref{def:minor}.
As such, $\cC(N) = \cC(M\loopsum A)$, and so $M\loopsum A$ is the unique $F$-matroid with this property.

Now let $N$ be an $F$-matroid on ground set $E \sqcup A$ such that all elements of $A$ are coloops and $N \setminus A = M$.
Again, we immediately have that $\cC(M) \times \{\underline{\0}_A\}$ is a subset of the circuits of $N$.
Suppose $C \in \cC(N)$ with $a \in \underline{C} \cap A$. %, for any covector $X \in \cV^*(M)$ we have $\sum_{i \in E \sqcup A} C_i \cdot X_i \in N_F$.
For any tuple $Y \in F^{E\sqcup A}$, define $\widetilde{Y}$ by
\[
\widetilde{Y}_i = \begin{cases} Y_i & i \neq a \\ \0 & i = a \end{cases} \, .
\]
As $a$ is a coloop, by Definition~\ref{defn:loop+coloop} for any covector $X \in \cV^*(N)$, we have $\widetilde{X}\in \cV^*(N)$ a covector also.
This implies that $\widetilde{C} \in F^{E \sqcup A}$ satisfies
\[
\langle\widetilde{C}, X \rangle = \sum_{i \in E \sqcup A} \widetilde{C}_i \cdot \overline{X}_i = \sum_{i \in E \sqcup A - a} C_i \cdot \overline{X}_i = \langle C, \widetilde{X} \rangle \in N_F \, .
\]
for all $X \in \cV^*(M)$.
Moreover, $\underline{\widetilde{C}} \subsetneq \underline{C}$, contradicting that $C$ was a circuit.
This gives that $\cC(N) = \cC(M\coloopsum A)$.

We show that $P_{M\loopsum A}$ gives rise to the circuits $\cC(M\loopsum A)$ using~\eqref{eq:circuits}.
First note that $\rank(M\loopsum A) = \rank(M)$, and so the circuits are scalings of $C_\tau$ for $\tau \in {E \choose r+1}$.
When $\tau \cap A = \emptyset$, we immediately get $C_\tau \in \cC(M) \times \{\underline{\0}_A\}$.
When $\tau \cap A = a$, we get that $(C_\tau)_a = P(\tau - a)$ and $(C_\tau)_i = \0$ for any $i \in \tau - a$.
As such, $C_\tau$ is a scaling of $\chi_a$, or the zero vector if $\tau - a$ is not a basis.
When $|\tau \cap A| \geq 2$, we have $(\tau - a) \cap A \neq \emptyset$ for all $a \in \tau$, hence $C_\tau$ is the zero vector.
These three cases cover $\cC(M\loopsum A)$, and so this is the Pl\"ucker vector of $M \loopsum A$.

We now show that $P_{M\coloopsum A}$ gives rise to the circuits $\cC(M\coloopsum A)$ using~\eqref{eq:circuits}.
First note that $\rank(M\coloopsum A) = \rank(M) + |A|$, and so the circuits are scalings of $C_\tau$ for $\tau \in {E \choose r+|A|+1}$.
When $A \subseteq \tau$, we have that $C_\tau$ is equal to $C_{\tau\setminus A} \times \{\underline{\0}_A\}$ where $C_{\tau\setminus A} \in \cC(M)$.
When $A \nsubseteq \tau$, we have that $C_\tau$ is the zero vector.
These cases cover $\cC(M\coloopsum A)$, and so this is the Pl\"ucker vector of $M \coloopsum A$.

Finally, we show that $(M \loopsum A)^* = M^* \coloopsum A$.
We do this by showing the cocircuits $\cC^*(M \loopsum A)$ are exactly $\cC^*(M) \times \{\underline{\0}_A\}$.
It is straightforward to show that any element of $\cC^*(M) \times \{\underline{\0}_A\}$ is orthogonal to $C \times \{\underline{\0}_A\}$ for all $C \in \cC(M)$, and to all $\chi_a$ for $a \in A$.
Conversely, any element $D \in \cC^*(M \loopsum A)$ must have $D_a = \0$ for all $a \in A$ to be orthogonal to $\chi_a$.
%As such, $D \setminus A \in \cC^*(M)$ to be orthogonal to $\cC(M)$ and of minimal support.
Definition~\ref{def:minor} tells us that the deletion $D \setminus A \in \cC^*(M)$ is a cocircuit of $M$.
\end{proof}

\section{Matroids over tropical extensions of tracts} \label{sec:trop+ext+matroids}
In this section, we will focus on matroids over tropical extensions of tracts.
Throughout, we will consider the tract $F[\Gamma]$, the tropical extension of the tract $F$ by the ordered abelian group $\Gamma$, with zero element $\infty$ and identity $0$.
We will regularly utilise the natural maps between these tracts introduced in Example~\ref{ex:maps},
\begin{align*}
|\cdot| \colon F[\Gamma] &\longrightarrow \KK[\Gamma] & \theta \colon F[\Gamma] &\longrightarrow F \\
(a, \gamma) &\longmapsto (\1, \gamma) & (a, \gamma) &\longmapsto a \\
\infty &\longmapsto \infty & \infty &\longmapsto \0 \, .
\end{align*}
the former of which is a homomorphism.
We can also view $|\cdot|$ as a map to $\Gamma \cup \{\infty\}$ by identifying $(\1,\gamma) \in \KK[\Gamma]$ with $\gamma \in \Gamma$.
In particular, we will frequently use the ordered additive group structure on $\Gamma$.

We introduce some notation for tuples in $(\Gamma\cup \{\infty\})^E$.
Given some $\bu \in (\Gamma\cup \{\infty\})^E$, we define
\[
Z_\bu = \SetOf{i \in E}{u_i = \infty} \quad , \quad \bu^\circ = \bu \setminus Z_\bu \in \Gamma^{E\setminus Z_\bu} \, .
\]
Note that $Z_\bu$ is the set of (tropical) zeroes of $\bu$, and that we can always write $\bu$ uniquely as $\bu = \bu^\circ \times \{\infty\}^{Z_\bu}$.

\subsection{Initial matroids} \label{sec:initial+matroid}
The key notion in this section will be the concept of an \emph{initial matroid}, also know as a \emph{residue matroid}~\cite{Bowler+Pendavingh:2019}.
We first define the special case when $\bu \in \Gamma^E$ in Definition~\ref{def:toric+initial+tuple}, then in full generality in Definition~\ref{def:initial+tuple}.

\begin{definition}[Toric initial matroid]\label{def:toric+initial+tuple}
Let $M$ be an $F[\Gamma]$-matroid of rank $r$ with Pl\"ucker vector $P\colon {E \choose r} \rightarrow F[\Gamma]$.
Given some $\bu \in \Gamma^E$, we define its \emph{toric initial matroid} $M^\bu$ to be the $F$-matroid with Pl\"ucker vector $P^\bu \colon {E \choose r} \rightarrow F$ defined as
\[
P^\bu(I) = \begin{cases}
\theta(P(I)) & |P(I)| - \sum\limits_{i \in I} u_i \leq |P(J)| - \sum\limits_{j \in J} u_j \text{ for all } J \in {E \choose r} \, , \\
\0 & \text{otherwise} \, .
\end{cases}
\]
\end{definition}
We use the term toric initial matroid, as $\bu \in \Gamma^E$ is a point of the tropical torus i.e. $\bu$ has no coordinates equal to infinity, the tropical zero element.
%The toric initial matroid has a particularly nice description as the $F$-entries restricted only to coordinates whose `valuated bases' are minimal with respect to the weight vector $-\bu \in \Gamma^E$.

We will need a definition of initial matroid for $\bu \in (\Gamma \cup \{\infty\})^E$ where coordinates may be $\infty$.
However, we cannot immediately extend Definition~\ref{def:toric+initial+tuple} as the inequality given may include $\infty + (-\infty)$ terms.
As such, we must make a slightly more technical definition by contracting away coordinates with $u_i = \infty$ then readding them as coloops.

\begin{definition}[Initial matroid]\label{def:initial+tuple}
Let $M$ be an $F[\Gamma]$-matroid of rank $r$.
Given some $\bu \in (\Gamma \cup \{\infty\})^E$, we define the \emph{initial matroid} $M^\bu$ to be
\begin{equation} \label{eq:initial+nontoric}
M^\bu = (M/Z_\bu)^{\bu^\circ} \coloopsum Z_\bu \, .
\end{equation}
Its underlying matroid $\underline{M^\bu}$ has rank $r' = r - \rank_M(Z_\bu)+ |Z_\bu|$.
\end{definition}
%The Pl\"ucker vector $P^\bu \colon {E \choose r'} \rightarrow F$ of $M^\bu$ is more complex than in the toric case.
%We first note that by using Proposition~\ref{prop:loop+coloop}, we have for any $I \in {E \choose r'}$
%\begin{align*}
%P^\bu(I) &= P_{(M/Z_\bu)^{\bu^\circ} \coloopsum Z_\bu}(I) = 
%\begin{cases}
%P_{(M/Z_\bu)^{\bu^\circ}}(I \setminus Z_\bu) & Z_\bu \subseteq I \\ \0 & \text{otherwise}
%\end{cases} \, .
%\end{align*}
%As such, we can just consider subsets of the form $J \sqcup Z_\bu$ for $J \in {E \setminus Z_\bu \choose r - \rank_M(Z_\bu)}$.
%Let $Z \subseteq Z_\bu$ be any maximal independent subset of $Z^\bu$ in $\underline{M}$, then
%\begin{align*}
%P^\bu(J \sqcup Z_\bu) &= \begin{cases}
%\theta(P_{M/Z_\bu}(J)) & |P_{M/Z_\bu}(J)| - \sum_{j \in J} u_j \leq |P_{M/Z_\bu}(K)| - \sum_{k \in K} u_k \, , \quad \forall K \in {E \setminus Z_\bu \choose r - \rank_M(Z_\bu)} \\
%\0 & \text{otherwise}  
%\end{cases} \\
%&= \begin{cases}
%\theta(P_M(J\sqcup Z)) & |P_M(J\sqcup Z)| - \sum_{j \in J} u_j \leq |P_M(K\sqcup Z)| - \sum_{k \in K} u_k \, , \quad \forall K \in {E \setminus Z_\bu \choose r - \rank_M(Z_\bu)}  \\
%\0 & \text{otherwise}  
%\end{cases}
%\end{align*}
One can derive the Pl\"ucker vector for non-toric initial matroids using Definition~\ref{def:toric+initial+tuple}, Proposition~\ref{prop:loop+coloop} and the Pl\"ucker vector formulation of contractions in~\cite[Lemma 4.4]{Baker+Bowler:2019}.
However, the resulting function is too clunky to use and not particularly enlightening.
Moreover, non-toric initial matroids will have a much cleaner description in terms of circuits, which we shall derive in Section~\ref{sec:initial+circuits}.
Given this, we will mostly work with the toric case and extend to non-toric initial matroids via~\eqref{eq:initial+nontoric}.

Despite referring to them as initial matroids, it is a priori not clear that $P^\bu$ is the Pl\"ucker vector of an $F$-matroid.
The first result is that this is indeed the case.

\begin{proposition}\label{prop:ext+implies+initial}
    Let $M$ be a weak (resp. strong) $F[\Gamma]$-matroid with Pl\"ucker vector $P \colon {E \choose r} \rightarrow F[\Gamma]$.
    Then $P^\bu \colon {E \choose r} \rightarrow F$ is the Pl\"ucker vector of a weak (resp. strong) $F$-matroid $M^\bu$ for all $\bu \in (\Gamma \cup \{\infty\})^E$.
\end{proposition}
We shall defer the proof of Proposition~\ref{prop:ext+implies+initial} until Section~\ref{sec:flags}, where we will prove a more general statement for initial flag matroids.

\begin{example}
    The prototypical example of initial matroids is when $F = \KK$ and $\Gamma = \RR$, hence the tropical extension is the tropical hyperfield $\TT = F[\Gamma]$.
    $\TT$-matroids are valuated matroids, and so its initial matroids are ordinary matroids by Proposition~\ref{prop:ext+implies+initial}.
As discussed in the Introduction, this is an if and only if statement for valuated matroids: the rest of this section shall be dedicated to extending this to arbitrary tropical extensions. 
\end{example}

\begin{remark}
Toric initial matroids of $\TT$-matroids have an additional polyhedral interpretation rather than just algebraic, which we briefly recall; see~\cite{Speyer:2008,Maclagan+Sturmfels:2015} for full details.
Given a $\TT$-matroid $M$ with Pl\"ucker vector $P$, we consider the polyhedron
\[
\mathcal{E}(M) = \conv\left\{(e_B,P(B)) \mid B \in \cB(\underline{M})\right\} \subseteq \RR^E \times \RR \, , \quad e_B  = \sum_{i \in B} e_i \in \RR^E \, ,
\] 
and take the projection of its lower faces into $\RR^E$.
The resulting polyhedral complex is a regular subdivision $\cQ_M$ of the matroid base polytope $Q_{\underline{M}} := \conv(e_B \mid B \in \cB(\underline{M}))$.
The faces of $\cQ_M$ are precisely the matroid polytopes $Q_{M^{\bu}}$ of the toric initial matroids $M^{\bu}$, as $(-\bu,1)$ is the linear functional on $\RR^E \times \RR$ that is minimized on $Q_{M^{\bu}}$.
This construction can be repeated for any function $P \colon {E \choose r} \rightarrow \RR \cup \{\infty\}$, but only gives rise to a subdivision into matroid polytopes if $P$ is the Pl\"ucker vector of a $\TT$-matroid.

When $\Gamma = \RR$, the toric initial matroids of $F[\Gamma]$-matroids can be viewed as matroid polytopes with the $F$-data $\theta(P(B))$ attached at the vertex $e_B$.
%This is a viewpoint taken in~\cite{} where they view sign patterns on (not necessarily regular) matroid subdivisions.
When $\Gamma \neq \RR$, this polyhedral approach to initial matroids breaks down and so we must restrict to the algebraic approach.
\end{remark}

%\todo[inline]{Check if this is really necessary}
%Given an $F[\Gamma]$-matroid $M$, the image $\theta(M)$ in $\theta$ is not an $F$-matroid as $\theta$ is not a homomorphism.
%The initial matroid $M^\bu$ is an $F$-matroid, but is not the image of some homomorphism from $F[\Gamma]$ to $F$.
%If we replace $\theta$ with `taking initial matroids' in Diagram~\eqref{eq:diagram}, the next lemma shows it is a commuting diagram for matroids.
%\begin{lemma} \label{lem:initials+commute}
%Let $M$ be an $F[\Gamma]$-matroid and $\bu \in \Gamma^n$.
%Then $\underline{M^\bu} = |M|^\bu$ as $\KK$-matroids.
%\end{lemma}
%\begin{proof}
%It is straightforward to check that both $\underline{M^\bu}$ and $|M|^\bu$ have the same Pl\"ucker vector:
%\[
%|P|^\bu(I) = \underline{P^\bu}(I) = \begin{cases}
%1 \in \KK & |P(I)| - \sum_{i \in I} u_i \leq |P(J)| - \sum_{j \in J} u_j  \, , \\
%0 \in \KK & \text{otherwise} \, .
%\end{cases}
%\]
%\end{proof}

The first main theorem of this section is that the converse to Proposition~\ref{prop:ext+implies+initial} also holds for weak matroids.
\begin{theorem} \label{thm:matroid+subdivision}
Let $F$ be a tract and $\Gamma$ an ordered abelian group.
$M$ is a weak $F[\Gamma]$-matroid if and only if $M^\bu$ is a weak $F$-matroid for all $\bu \in \Gamma^E$.
\end{theorem}
\begin{proof}
Let $P \colon {E \choose r} \rightarrow F[\Gamma]$ be an arbitrary function.
Given Proposition~\ref{prop:ext+implies+initial}, it suffices to show that if $P$ does not satisfy some three term Pl\"ucker relation, then there exists some $\bu \in \Gamma^E$ such that $P^\bu$ also does not satisfy that relation.

Suppose that $P$ does not satisfy $\cP_{I, J}$ where $J \setminus I = \{j_0j_1j_2\}$:
\begin{align*} \label{eq:bad+relation}
\cP_{I, J} \, : \, \sum_{k=0}^2 (-1)^k\cdot P(I+j_k)\cdot P(J - j_k) \notin N_{F[\Gamma]} \quad 
\Rightarrow \quad \sum_{j_k \in K} (-1)^k \cdot \theta(P(I+j_k))\cdot \theta(P(J - j_k)) \notin N_F \, \\
\text{ where }  K = \left\{j_k \in J\setminus I \, \colon \, \infty \neq |P(I + j_k)| + |P(J - j_k)| \leq |P(I + j)| + |P(J - j)| \, \forall j \in J \setminus I \right\}.
\end{align*}
We note that $K \neq \emptyset$, else $P$ satisfies $\cP_{I,J}$.
Without loss of generality, let $j_0 \in K$, i.e. $|P(I + j_0)| + |P(J - j_0)| \neq \infty$ is minimal.
We aim to find a vector $\bu \in \Gamma^E$ such that $P^\bu$ also does not satisfy $\cP_{I,J}$ either.

Define $\gamma \in \Gamma^{E \choose r}$ as
\[
\gamma_{B} = \begin{cases} |P(B)| & |P(B)| \neq \infty \\ 0 & |P(B)| = \infty \end{cases} \, , \quad \forall B \in {E \choose r} \, ,
\]
and fix some arbitrarily large element $\Omega \in \Gamma$ such that $\Omega \gg \gamma_B$ for all $B$.
Let $i$ be the unique element of $I \setminus J$, and define $\bu \in \Gamma^E$ as
\[
u_s = \begin{cases}
    \gamma_{I+j_1} + \gamma_{I +j_2} - \gamma_{J - j_0} & s = i \\
    \gamma_{I + s} & s \in \{j_0,j_1,j_2\} \\
    \Omega & s \in I \cap J \\
    -\Omega & s \notin I \cup J
\end{cases} \, .
\]
We write $\psi(B) = |P(B)| - \sum_{k \in B}{u_k}$ and recall that $B \in \underline{P^\bu} := \{B' \in {E \choose r} \mid P^\bu(B') \neq \0\}$ if and only if it is minimized by $\psi$.
Providing that $\Omega$ is sufficiently large, if $B \in \underline{P^\bu}$ then $B = (I \cap J) \cup \{a,b\}$ where $a,b \in \{i,j_0,j_1,j_2\}$, as $\psi$ will be minimized by selecting all elements of $I \cap J$ and some subset of elements from $\{i,j_0,j_1,j_2\}$.
Let $W = \gamma_{J-j_0} - \gamma_{I+j_1} - \gamma_{I+j_2} - (r-2)\Omega$.
It is a straightforward check that:
\begin{align*}
\psi(I + j_0) &= \psi(J - j_0) = W \quad  \, , \\
\psi(I + j_k) &= \begin{cases} W & \text{ if } |P(I+j_k)| \neq \infty \\ \infty & \text{ if } |P(I+j_k)| = \infty \end{cases} \, , \quad k = 1,2 \\
\psi(J - j_1) &= |P(J-j_1)| - \gamma_{I+j_0} - \gamma_{I+j_2} - (r-2)\Omega \geq W \, . \\
\psi(J - j_2) &= |P(J-j_2)| - \gamma_{I+j_0} - \gamma_{I+j_1} - (r-2)\Omega \geq W \, , \\
\end{align*}
This implies that $\{I+j_0,I+j_1, I+j_2, J-j_0\} \cap \underline{P} \subset \underline{P^\bu}$, and that 
\begin{align*}
J-j_1 \in \underline{P^\bu} \, \Leftrightarrow \, |P(I+j_0)|+ |P(J-j_0)| = |P(I+j_1)|+ |P(J-j_1)| \\
J - j_2 \in \underline{P^\bu} \, \Leftrightarrow \, |P(I+j_0)|+ |P(J-j_0)| = |P(I+j_2)|+ |P(J - j_2)| 
\end{align*}
In particular, we can deduce that
\begin{align*}
\sum_{k=0}^2 (-1)^k \cdot P^\bu(I+j_k)\cdot P^\bu(J - j_k)
= \sum_{j_k \in K} (-1)^k \cdot \theta(P(I+j_k))\cdot \theta(P(J - j_k)) \notin N_F \, .
\end{align*}
\end{proof}

\begin{remark}\label{rem:toric+subdivision}
Although $M^\bu$ is an $F$-matroid for all $\bu \in (\Gamma \cup \{\infty\})^E$, Theorem~\ref{thm:matroid+subdivision} shows it is actually sufficient to check just the toric initial $F$-matroids to determine if $M$ is an $F[\Gamma]$-matroid.
This won't be the case when we consider flag matroids in Section~\ref{sec:flags}, and so it is necessary to introduce non-toric initial matroids.
\end{remark}

The natural next question is whether an analogue for Theorem~\ref{thm:matroid+subdivision} holds for strong matroids.
The answer depends on whether we are working over a perfect tract.
To this end, we first note that $F[\Gamma]$ is perfect if and only if $F$ is perfect.

\begin{theorem}[\cite{Bowler+Pendavingh:2019}]\label{thm:perfect+tract}
    Let $F$ be a tract and $\Gamma$ an ordered abelian group.
    Then $F$ is a perfect tract if and only if $F[\Gamma]$ is a perfect tract.
\end{theorem}
\begin{proof}
The direction that $F$ is perfect implies $F[\Gamma]$ is perfect is given in~\cite[Theorem 33]{Bowler+Pendavingh:2019} when $F$ is a hyperfield.
However, their proof does not use any special properties of hyperfields, and so can be generalised directly to tracts.
The opposite direction follows immediately from embedding $F \hookrightarrow F[\Gamma]$ by sending $a \mapsto (a,0)$.
\end{proof}

We get the immediate corollary in the case where $F$ is a perfect tact.
\begin{corollary*}(Theorem~\ref{thm:A})
    Let $F$ be a perfect tract and $\Gamma$ an ordered abelian group.
$M$ is a strong $F[\Gamma]$-matroid if and only if $M^\bu$ is a strong $F$-matroid for all $\bu \in \Gamma^E$.
\end{corollary*}
\begin{proof}
    By Theorem~\ref{thm:perfect+tract}, if $F$ is a perfect tract then so is $F[\Gamma]$.
    As weak matroids and strong matroids are equivalent over perfect tracts, the claim immediate follows from Theorem~\ref{thm:matroid+subdivision}.
\end{proof}

We end this section by showing that Theorem~\ref{thm:matroid+subdivision} does not hold for non-perfect tracts.
Explicitly, we construct a $F[\Gamma]$-matroid $M$ that is not strong but whose initial matroids $M^\bu$ are all strong.
This is done in the following variation of~\cite[Example 3.30]{Baker+Bowler:2019}.

\begin{example}[A weak matroid with strong initial matroids] 
\label{ex:strong+counterexample}
Let $\VV = (\RR_{\geq 0}, \boxplus, \odot)$ be the triangle hyperfield, whose operations are the usual multiplication on $\RR_{\geq 0}$ and the hyperaddition
\begin{align*}
    a \boxplus b = \left\{c \in \RR_{\geq 0} \mid |a -b| \leq c \leq a+b \right\} \, .
\end{align*}
As we have a well-defined (hyper)addition, we will work with it directly rather than just the null set.

Let $\Gamma \neq \{0\}$ be any non-trivial ordered abelian group and let $\delta \in \Gamma$ with $\delta < 0$.
Define the function $P \colon {6 \choose 3} \rightarrow \VV[\Gamma]$ by:
\begin{align*}
    \theta(P(ijk)) &= \begin{cases}
    4 & \{ijk\} = \{156\} \\
    2 & i = 1, \, j \in \{234\}, \, k \in \{56\} \\
    1 & \text{otherwise}
    \end{cases} \, ,
    &
    |P(I)| &= \begin{cases}
    0 & I \subset \{1234\} \text{ or } \{56\} \subset I \\
    \delta & \text{otherwise}
    \end{cases} \, .
\end{align*}
It is a straightforward case analysis to show that $P$ satisfies all the three-term Pl\"ucker relations, and is therefore the Pl\"ucker vector of a weak $\VV[\Gamma]$-matroid $M$.
%Our first claim about $P$ is that it satisfies all three-term Pl\"ucker relations, and is therefore the Pl\"ucker vector of a weak $\VV[\Gamma]$-matroid $M$.
%As the proof is straightforward but slightly tedious case analysis, we defer the proof to Appendix~\ref{sec:counterexample+proofs}.
%\begin{claim}\label{claim:weak}
%    $P$ is the Pl\"ucker vector of a weak $\VV[\Gamma]$-matroid $M$.
%\end{claim}
However, $M$ is not a strong $\VV[\Gamma]$-matroid, as $P$ does not satisfy the four-term Pl\"ucker relation $\cP_{56,1234}$:
\begin{align*}
\cP_{56,1234} \, &\colon \, P(156)\odot P(234) \, \boxplus \, P(256)\odot P(134) \, \boxplus \, P(356)\odot P(124) \, \boxplus \, P(456)\odot P(123) \\
&= (4,0) \boxplus (1,0)  \boxplus (1,0)  \boxplus (1,0) \\
&= (4,0) \boxplus \left(\{(\lambda,0) \mid \lambda \in (0,3]\} \cup (\RR_{>0} \times \RR_{>0})\cup \{\infty\}\right) \\
&= \left\{(\lambda,0) \mid \lambda \in [1,7]\right\} \not\ni \infty
\end{align*}
However, it is another straightforward check to verify that this is the only Pl\"ucker relation that $P$ does not satisfy.
%Recall that for a rank 3 matroid on six elements, the only non-trivial Pl\"ucker relations are the three-term relations and the four-term Pl\"ucker relations, those of the form $\cP_{I;[6]\setminus I}$ where $I \in {6\choose 2}$.
%Our second claim is that $\cP_{56;1234}$ is the only Pl\"ucker relation that $P$ does not satisfy.
%As with the previous claim, we defer the proof to Appendix~\ref{sec:counterexample+proofs}.
%\begin{claim}\label{claim:strong}
%    $P$ satisfies the four-term Pl\"ucker relation $\cP_{I;[6]\setminus I}$ if and only if $I \neq \{56\}$.
%\end{claim}

By Remark~\ref{rem:toric+subdivision}, it suffices to just check whether the toric initial matroids are strong.
Consider some arbitrary $\bu\in \Gamma^6$ and the initial matroid $M^\bu$ it defines.
By Proposition~\ref{prop:ext+implies+initial}, if $P$ satisfies the Pl\"ucker relation $\cP_{I;J}$, then so does each of its initial matroids.
Hence $P^\bu$ satisfies all Pl\"ucker relations aside from possibly $\cP_{56;1234}$: we show that this is also satisfied.

Assume that both $P^\bu(156) \neq \0$ and $P^\bu(234) \neq \0$, or equivalently that both $156, 234$ are contained in $\underline{P^\bu}$, the support of $P^\bu$.
From the definition of toric initial matroids, we have
\begin{align*}
|P(156)| + |P(234)| - \sum_{i=1}^6 u_i &\leq |P(I)| + |P([6] \setminus I)| - \sum_{i=1}^6 u_i \\
\Longrightarrow \quad |P(156)| + |P(234)| &\leq |P(I)| + |P([6] \setminus I)| \quad \forall I \in {6 \choose 3} \, .
%\gamma_{156} + \gamma_{234} \leq \gamma_{I} + \gamma_{[6] \setminus I} \quad \forall I \in {[6] \choose 3} \, .
\end{align*}
But this is not true: setting $I = 125$, we see $|P(125)| + |P(346)| = 2\delta < 0 = |P(156)| + |P(234)|$.
As such, there is no initial matroid $M^\bu$ with both $P^\bu(156) \neq \0$ and $P^\bu(234) \neq \0$.
This argument can be repeated for any pair of complementary subsets that appear in a monomial of $\cP_{56;1234}$.
Hence, $\cP_{56;1234}$ is just the zero expression for all $P^\bu$, and so trivially satisfied.
Therefore, all of the initial matroids $M^\bu$ of $M$ are strong $\VV$-matroids.
\end{example}

\subsection{Circuits, vectors and duality for initial matroids}\label{sec:initial+circuits}

In this section, we enrich the structure of initial $F$-matroids by describing their duals and (co)circuits.
We then use this to characterise the (co)vectors of $F[\Gamma]$-matroids.
Throughout we let $M$ be an $F[\Gamma]$-matroid with Pl\"ucker vector $P \colon {E \choose r} \rightarrow F[\Gamma]$ and $\bu \in (\Gamma \cup \{\infty\})^E$.
We will often highlight the (much simpler) special case where $M^\bu$ is toric, i.e. $\bu \in \Gamma^E$.
Note that all results in this section hold for both weak and strong $F[\Gamma]$-matroids.

We begin by describing the dual of $M^\bu$.

\begin{proposition}\label{prop:initial+dual}
Let $M$ be an $F[\Gamma]$-matroid and $\bu \in (\Gamma \cup \{\infty\})^E$.
Then
\[
(M^\bu)^* = (M^* \setminus Z_\bu)^{-\bu^\circ} \loopsum Z_\bu \, .
\]
If $M^\bu$ is toric, then $(M^\bu)^* = (M^*)^{-\bu}$.
\end{proposition}
\begin{proof}
We first show the toric initial matroid case where $\bu \in \Gamma^E$, i.e. $(M^\bu)^* = (M^*)^{-\bu}$.
We show that $(M^\bu)^*$ and $(M^*)^{-\bu}$ have the same Pl\"ucker vector.
Recall that the dual Pl\"ucker vector $P^*$ associated with $M^*$ is defined as
\[
P^*(I) = \sign(I,I^c)\cdot \overline{P(I^c)} \, .
\]
Its initial Pl\"ucker vector with respect to $-\bu$ is
\[
(P^*)^{-\bu}(I) =
\begin{cases}
\theta(P^*(I)) & |P^*(I)| + \sum\limits_{i \in I} u_i \leq |P^*(J)| + \sum\limits_{j \in J} u_j \; \forall J \in {E \choose |E|-r} \\
\0 & \text{otherwise}
\end{cases} \, .
\]
From the definition of $P^*$, and that signs and involutions do not appear over $\KK[\Gamma]$, we can rewrite the inequality
\begin{align}\label{eq:dual+initial}
|P^*(I)| + \sum_{i \in I} u_i \leq |P^*(J)| + \sum_{j \in J} u_j \quad \Longleftrightarrow \quad |P(I^c)| - \sum_{i \in I^c} u_i \leq |P(J^c)| - \sum_{j \in J^c}  u_j \, ,
\end{align}
where we have subtracted the sum of the coordinates of $\bu$ from either side.
Conversely, we can write $(P^\bu)^*$ as
\[
(P^\bu)^*(I) = \sign(I,I^c)\cdot \overline{P^\bu(I^c)} =
\begin{cases}
\sign(I,I^c)\cdot \overline{\theta(P(I^c))} & |P(I^c)| - \sum\limits_{i \in I^c} u_i \leq |P(J^c)| - \sum\limits_{j \in J^c}  u_j \forall J \in {E \choose |E|-r} \\
0 & \text{otherwise}
\end{cases} \, .
\]
As $\theta(P^*(I)) = \sign(I,I^c)\cdot \overline{\theta(P(I^c))}$, and the two inequality conditions are equivalent from~\eqref{eq:dual+initial}, the Pl\"ucker vectors $(P^*)^{-\bu}$ and $(P^\bu)^*$ are the same.

The non-toric case follows from the following chain of implications:
\begin{align*}
(M^\bu)^* &= \left((M/Z_\bu)^{\bu^\circ} \coloopsum Z_\bu\right)^* & & \\
&= \left((M/Z_\bu)^{\bu^\circ}\right)^* \loopsum Z_\bu & &\text{(Proposition~\ref{prop:loop+coloop})} \\
&= \left((M/Z_\bu)^*\right)^{-\bu^\circ} \loopsum Z_\bu & &\text{($(M^\bu)^* = (M^*)^{-\bu}$ in toric case)} \\
&= \left(M^* \setminus Z_\bu\right)^{-\bu^\circ} \loopsum Z_\bu & &\text{(Definition~\ref{def:minor})} \, .
\end{align*}
\end{proof}

We next characterise the circuits and cocircuits of $M^\bu$.
Given some $F[\Gamma]$-circuit $C \in \cC(M)$, define the \emph{initial circuit} $C^\bu \in F^E$ as
\begin{equation} \label{eq:initial+circuit}
C^\bu_i = \begin{cases}
\theta(C_i) & \infty \neq |C_i| + u_i \leq |C_j| + u_j \, \forall j \in E \\
\0 & \text{otherwise}
\end{cases} \, .
\end{equation}

\begin{proposition}\label{prop:initial+circuits}
Let $M$ be an $F[\Gamma]$-matroid and $\bu \in (\Gamma \cup \{\infty\})^E$.
The circuits and cocircuits of $M^\bu$ are precisely 
\begin{align*}
\cC(M^\bu) &= \MinSupp(\left\{C^\bu \mid C \in \cC(M)\right\}) \, , \\
\cC^*(M^\bu) &= \MinSupp\left(\left\{(D\setminus Z_u)^{-\bu^\circ} \times \{\underline{\0}_{Z_\bu}\} \mid D \in \cC^*(M) \, , \, \underline{D} \cap Z_\bu = \emptyset\right\}\right) \cup \{\alpha \cdot\chi_i \mid \alpha \in F^\times \, , \, i \in Z_\bu \} \, .
\end{align*}
If $M^\bu$ is toric, then the cocircuits are
$\cC^*(M^\bu) = \MinSupp(\left\{D^{-\bu} \mid D \in \cC^*(M)\right\})$ and the circuits are as above.
\end{proposition}

To prove this, we first need a lemma showing that the initial circuits are all orthogonal to the initial cocircuits of $M$.
\begin{lemma}\label{lem:init+circs+orth}
Let $M$ be an $F[\Gamma]$-matroid and $\bu \in (\Gamma \cup \{\infty\})^E$.
For all $C \in \cC(M)$ and $D \in \cC^*(M)$ with $\underline{D} \cap Z_\bu = \emptyset$, we have
\[
C^\bu \perp ((D \setminus Z_\bu)^{-\bu^\circ} \times \{\underline{\0}_{Z_\bu}\}) \, .
\]
If $\bu \in \Gamma^E$, then $C^\bu \perp D^{-\bu}$ for all $C \in \cC(M)$ and $D \in \cC^*(M)$.
\end{lemma}
\begin{proof}
We first show the toric case where $\bu \in \Gamma^E$.
Let $I = \SetOf{i \in \underline{C} \cap \underline{D}}{|C_i| + |D_i| \leq |C_j| + |D_j| \quad \forall j \in E}$.
As $C\perp D$, this implies
\[
\sum_{i \in I} \theta(C_i) \cdot \overline{\theta(D_i)} \in N_F \, .
\]
Let $I^\bu = \underline{C^\bu}\cap\underline{D^{-\bu}}$, then
\[
\langle C^\bu , D^{-\bu} \rangle = \sum_{i \in I^\bu} \theta(C_i) \cdot \overline{\theta(D_i)} \, .
\]
If $I^\bu = I$ or $I^\bu = \emptyset$, then $C^\bu \perp D^{-\bu}$.
We show that one of these cases always occurs.

We first observe that for all $i \in I^\bu$
\begin{align*}
\begin{split}
|C_i| + u_i \leq |C_j| + u_j \\
|D_i| - u_i \leq |D_j| - u_j \\
\end{split} \quad \forall j \in E \quad \Rightarrow \quad |C_i| + |D_i| \leq |C_j| + |D_j| \quad \forall j \in E \, .
\end{align*}
This implies $I^\bu \subseteq I$: if they are equal then we are done.
If they are not equal then there exists $i \in I \setminus I^\bu$.
Without loss of generality $i \notin \underline{C^\bu}$, and so we have $|C_i| + u_i > |C_k| + u_k$ for all $k \in \underline{C^\bu}$.
This implies for all $k \in \underline{C^\bu}$:
\begin{align*}
(|C_k| + u_k) + (|D_k| - u_k) &= |C_k| + |D_k| \geq |C_i| + |D_i| = (|C_i| + u_i) + (|D_i| - u_i) \\
&> (|C_k| + u_k) + (|D_i| - u_i) \\
\Longrightarrow \quad |D_k| - u_k &> |D_i| - u_i  \quad \forall k \in \underline{C^\bu} \, .
\end{align*}
In particular, $k \notin \underline{D^{-\bu}}$.
This implies $I^\bu = \emptyset$.

In the non-toric case, we note that $\underline{D} \cap Z_\bu = \emptyset$ implies that $I \cap Z_\bu = \emptyset$.
Moreover, we have that 
\[
\langle C^\bu, (D \setminus Z_\bu)^{-\bu^\circ} \times \{\underline{\0}_{Z_\bu}\} \rangle = \sum_{i \in E \setminus Z_\bu} C^\bu_i \cdot \overline{D^{-\bu}_i} 
= \langle (C \setminus Z_\bu)^{\bu^\circ}, (D \setminus Z_\bu)^{-\bu^\circ}\rangle
\]
as $C^\bu_i = \0$ for all $i \in Z_\bu$.
The proof now follows from the toric case applied to $(C \setminus Z_\bu)^{\bu^\circ}$ and $(D \setminus Z_\bu)^{-\bu^\circ}$.
\end{proof}

\begin{proof}[Proof of Proposition~\ref{prop:initial+circuits}]
We first prove the statement for toric initial matroids, and then deduce the non-toric case from that.

We first show the containment $(\subseteq)$ for circuits.
Recall that the Pl\"ucker vectors of $M^\bu$ are of the form
\[
P^\bu(I) = \begin{cases}
\theta(P(I)) & |P(I)| - \sum\limits_{i\in I} u_i \leq |P(J)| - \sum\limits_{j\in J} u_j  \quad \forall J \in {E \choose r} \\
\0 & \text{otherwise}
\end{cases} \, .
\]
Using~\eqref{eq:circuits}, we see that the circuits $X \in \cC(M^\bu)$ of $M^\bu$ are all of of the form
\begin{equation} \label{eq:initial+circuit+1}
X_{i_s} = 
\begin{cases}
(-1)^s \alpha\cdot\theta(P(\tau -i_s)) & |P(\tau - i_s)| - \sum\limits_{k \neq s} u_{i_k} \leq |P(J)| - \sum\limits_{j\in J} u_j \quad \forall J \in {E \choose r} \\
\0 & \text{otherwise}
\end{cases}
\end{equation}
for some $\tau \in {E \choose r+1}$ and $\alpha \in F^\times$.
Conversely, as each $C \in \cC(M)$ is equal to $\alpha \cdot C_\tau$ for $\tau \supset \underline{C}$ and some $\alpha \in F^\times$, its initial circuit $C^\bu$ is of the form
\begin{equation} \label{eq:initial+circuit+2}
(C^\bu)_{i_s} = \begin{cases}
(-1)^s \theta(\alpha \cdot P(\tau -i_s)) & |P(\tau - i_s)| - \sum\limits_{k \neq s} u_{i_k} \leq |P(\tau - j)| - \sum\limits_{k \in\tau - j} u_{k} \quad \forall j \in \tau \\
\0 & \text{otherwise}
\end{cases} \, .
\end{equation}
%for some $\tau \supset \underline{C}, \alpha \in F^\times$.
Note that we obtain the inequality above by subtracting $\sum_{i \in \tau} u_i$ from both sides of the inequality in~\eqref{eq:initial+circuit}.
The only difference between \eqref{eq:initial+circuit+1} and \eqref{eq:initial+circuit+2} is that the inequality in the latter ranges only over $r$-subsets of $\tau$ rather than all $r$-subsets.
This immediately tells us that
\[
\cC(M^\bu) \subseteq \{C^\bu \mid C \in \cC(M)\} \, .
\]
Moreover, the $F$-circuit axioms~\cite[Definition 3.8]{Baker+Bowler:2019} insist that elements of $\cC(M^\bu)$ must be non-zero and have minimal support, giving the containment $(\subseteq)$ of the claim for circuits.
The containment for cocircuits follows from the circuit claim along with Proposition~\ref{prop:initial+dual}, observing that:
\begin{align*}
\cC^*(M^\bu) &= \cC((M^\bu)^*) = \cC((M^*)^{-\bu}) \\
&\subseteq \MinSupp(\left\{D^{-\bu} \mid D \in \cC(M^*)\right\}) \\
&= \MinSupp(\left\{D^{-\bu} \mid D \in \cC^*(M)\right\}) \, .
\end{align*}

For the converse, Theorem~\ref{thm:dual+matroid} states that the circuits are precisely the vectors orthogonal to cocircuits of minimal support.
Lemma~\ref{lem:init+circs+orth} tells us that all initial circuits are orthogonal to all initial cocircuits $D^{-\bu}$.
As this includes all cocircuits of $M^\bu$, we deduce that the initial circuits of minimal support are a subset of $\cC(M)$.
A dual argument holds for cocircuits.
%
%Consider some $C \in \cC(M)$ such that $C^\bu$ is not contained in $\cC(M^\bu)$, we show that it is either zero or has non-minimal support.
%Observe from Lemma~\ref{lem:matroid+hom} that $|C|$ is a circuit of $|M|$, but $\underline{C^\bu}$ is not a circuit of $\underline{M^\bu}$.
%Moreover, it is straightforward to check from definitions that $\underline{C^\bu} = |C|^\bu$.
%As $\underline{M^\bu} = |M|^\bu$ by Lemma~\ref{lem:initials+commute}, it follows that $|C|^\bu$ is not a circuit of $|M|^\bu$.
%By~\cite[Theorem 3.7]{Murota+Tamura:2001}, which proves the claim for $F = \KK$, this implies that $|C|^\bu$ is either the all zeros vector or has non-minimal support.
%It follows that $\underline{C^\bu}$, and therefore $C^\bu$, are also either zero or have non-minimal support.
%This completes the proof.

For the non-toric case, the statement holds for circuits by the following chain of implications:
\begin{align*}
\cC(M^\bu) &= \cC((M/Z_\bu)^{u^\circ} \coloopsum Z_\bu) & & \\
&= \cC((M/Z_\bu)^{\bu^\circ}) \times \{\underline{\0}_{Z_\bu}\} & &\text{(Proposition~\ref{prop:loop+coloop})} \\
&= \MinSupp((\widetilde{C})^{\bu^\circ} \mid \widetilde{C} \in \cC(M/Z_\bu)) \times \{\underline{\0}_{Z_\bu}\} & &\text{(Toric case)} \\
&= \MinSupp((C \setminus Z_\bu)^{\bu^\circ} \mid C \in \cC(M)) \times \{\underline{\0}_{Z_\bu}\} & &\text{(Definition~\ref{def:minor})} \\
&= \MinSupp(C^\bu \mid C \in \cC(M)) 
\end{align*}
where the final equality holds as $u_i = \infty$ implies that $C^\bu_i = \0$.
For cocircuits, a similar deduction holds:
\begin{align*}
\cC^*(M^\bu) &= \cC((M/Z_\bu)^{u^\circ} \coloopsum Z_\bu)^*) & & \\
&= \cC((M^* \setminus Z_\bu)^{-\bu^\circ} \loopsum Z_\bu) & &\text{(Proposition~\ref{prop:initial+dual})} \\
&= \cC((M^* \setminus Z_\bu)^{-\bu^\circ}) \cup \{\alpha \cdot\chi_i \mid \alpha \in F^\times \, , \, i \in Z_\bu \} & &\text{(Proposition~\ref{prop:loop+coloop})} \\
&= \MinSupp((\widetilde{D})^{-\bu^\circ} \mid \widetilde{D} \in \cC(M^* \setminus Z_\bu))  \cup \{\alpha \cdot\chi_i \mid \alpha \in F^\times \, , \, i \in Z_\bu \} & &\text{(Toric case)} \\
&= \MinSupp((D\setminus Z_\bu)^{-\bu^\circ} \mid D\setminus Z_\bu \in \cC^*(M) \setminus Z_\bu)  \cup \{\alpha \cdot\chi_i \mid \alpha \in F^\times \, , \, i \in Z_\bu \} & &\text{(Definition~\ref{def:minor})} \\
&= \MinSupp\left(\left\{(D\setminus Z_u)^{-\bu^\circ} \times \{\underline{\0}_{Z_\bu}\} \mid D \in \cC^*(M)\, , \, \underline{D} \cap Z_\bu = \emptyset\right\}\right) \cup \{\alpha \cdot\chi_i \mid \alpha \in F^\times \, , \, i \in Z_\bu \} \, .
\end{align*}
\end{proof}

With a characterisation of (co)circuits of initial matroids complete, we close the section with a description of the (co)vectors of an $F[\Gamma]$-matroid.
\begin{proposition}\label{prop:covectors}
Let $M$ be an $F[\Gamma]$-matroid.
The vectors and covectors of $M$ are subsets
\begin{align*}
\cV(M) &\subseteq \SetOf{Y \in F[\Gamma]^E}{\theta(Y) \in \cV(M^{|Y|})} \, , \\
\cV^*(M) &\subseteq \SetOf{X \in F[\Gamma]^E}{\theta(X) \in \cV^*(M^{|X|})} \, .
\end{align*}
Moreover, if $F$ is a perfect tract then these are equalities.
\end{proposition}
\begin{proof}
We show the claim for covectors.
Writing $X_i = (\theta(X_i), |X_i|)$ for $i \in \underline{X}$, we see from definitions that
\begin{align}\label{eq:covector+containment}
X \in \cV^*(M) &\Longleftrightarrow \sum_{i\in \underline{X} \cap \underline{C}}( \theta(X_i)\cdot \overline{\theta(C_i)}, |X_i| + |C_i|) \in N_{F[\Gamma]} \quad \forall C \in \cC(M)\nonumber \\
&\Longleftrightarrow \sum_{i \in \underline{C^{|X|}}} \theta(X_i)\cdot \overline{\theta(C_i)} \in N_F \quad \forall C \in \cC(M) \nonumber \\
&\Longleftrightarrow \theta(X) \perp C^{|X|}  \quad  \forall C \in \cC(M) \, .
\end{align}
By Proposition~\ref{prop:initial+circuits}, each circuit of $M^{|X|}$ is some initial circuit $C^{|X|}$ for $C \in \cC(M)$.
As such we have $\theta(X) \in \cV^*(M^{|X|})$.
The claim for vectors holds by duality, as we perform the same argument everywhere replacing $M$ with $M^*$, noting that $\cV(M) = \cV^*(M^*)$.

Now suppose $F$ is a perfect tract, and $X \in F[\Gamma]^E$ such that $\theta(X) \in \cV^*(M^{|X|})$.
By Lemma~\ref{lem:init+circs+orth} and Proposition~\ref{prop:initial+circuits}, each initial circuit $C^{|X|}$ is orthogonal to all the cocircuits of $M^{|X|}$, and so $C^{|X|} \in \cV(M^{|X|})$.
As vectors and covectors are always orthogonal over perfect tracts, we therefore have $\theta(X) \perp C^{|X|}$ for all $C \in \cC(M)$.
This implies $X \in \cV^*(M)$, completing the proof.
\end{proof}

\begin{remark}
From a geometric viewpoint, the circuits of an $F$-matroid correspond to the (support minimal) linear forms that cut out a `linear space' over $F$.
As such, the covectors can be viewed as an analogue of a linear space over a tract.
For example, projective tropical linear spaces are precisely the covectors of a $\TT$-matroid~\cite{Murota+Tamura:2001,Brandt+Eur+Zhang:2021}.
With this viewpoint, Proposition~\ref{prop:covectors} allows us to describe the linear spaces that arise over tropical extensions of (perfect) tracts.
We shall explore this further in Section~\ref{sec:tropical+linear+spaces}.
\end{remark}

%\begin{remark}
%When $F = \KK$, many of the results in this section were shown in~\cite{Murota+Tamura:2001}.
%These use entirely combinatorial proofs that do not require the algebraic formalism of tracts.
%\end{remark}

\section{Further matroid structures over tracts} \label{sec:further+matroids}

In this section, we utilise the theory we built up in Section~\ref{sec:trop+ext+matroids} to examine the structure of flag matroids and positroids over $F[\Gamma]$.
Our goal is to establish characterisations in terms of their initial (flag) matroids, analogous to Theorem~\ref{thm:A}.

\subsection{Flag matroids over tracts} \label{sec:flags}

In this section, we extend Theorem~\ref{thm:matroid+subdivision} to flag matroids over tropical extensions of tracts.
We first recall the necessary preliminaries for flag matroids over tracts.

\begin{definition} \label{def:flag+matroid}
Let $0 < r \leq s < |E|$ and let $M$ and $N$ be strong $F$-matroids of ranks $r$ and $s$ respectively with Pl\"ucker vectors $P \colon {E \choose r} \rightarrow F$ and $Q \colon {E \choose s} \rightarrow F$.

Given subsets $I \in {E \choose r-1}$ and $J \in {E \choose s+1}$, $M$ and $N$ satisfy the \emph{Pl\"ucker incidence relation} $\cP^{r,s}_{I;J}$ if
\begin{equation}\label{eq:incidence+relations}
    \cP^{r,s}_{I,J} \; \colon \; \sum_{j \in J \setminus I} \sign(j;I,J) \cdot P(I+j) \cdot Q(J-j) \in N_F \, ,
\end{equation}
where $\sign(j;I,J) = (-1)^\ell$ with $\ell$ equal to the number of elements $j' \in J$ with $j< j'$ plus the number of elements $i \in I$ with $j<i$.

We say that $M$ is a \emph{quotient} of $N$, and write $N \quot M$, if it satisfies the Pl\"ucker incidence relations $\cP^{r,s}_{I;J}$ for all $I \in {E \choose r-1}$ and $J \in {E \choose s+1}$.

A \emph{flag $F$-matroid} is a sequence $\bM = (M_1, \dots, M_k)$ of strong $F$-matroids such that $M_j \quot M_i$ for all $1 \leq i < j \leq k$.
\end{definition}

Note that if $r=s$, then~\eqref{eq:incidence+relations} reduces to the usual Pl\"ucker relations \eqref{eq:plucker+relation}.
In particular, the two-term relations where $|J \setminus I| = 2$ just verify that $P = \lambda \cdot Q$ for some $\lambda \in F$, and so represent the same matroid.

As with $F$-matroids, there are numerous cryptomorphic definitions of flag matroids and matroid quotients; we refer the reader to~\cite{Jarra+Lorscheid:2022} for further details.
For our purposes, we will suffice with the following structure theorem that implies a cryptomorphism over perfect tracts.

\begin{theorem}\cite[Theorem 2.16]{Jarra+Lorscheid:2022} \label{thm:flag+equivalence}
Let $F$ be a tract and $M_1, \dots, M_k$ a sequence of $F$-matroids.
Consider the following properties:
\begin{enumerate}
\item The covector sets form a chain $\cV^*(M_1) \subseteq \dots \subseteq \cV^*(M_k)$; \label{fm1}
\item $(M_1, \dots, M_k)$ is a flag $F$-matroid (i.e. $M_j \quot M_i$ for all $1 \leq i < j \leq k$); \label{fm2}
\item $M_i$ is a quotient of $M_{i+1}$ for all $1 \leq i \leq k-1$. \label{fm3}
\end{enumerate}
Then the implications $\eqref{fm1} \Rightarrow \eqref{fm2} \Rightarrow \eqref{fm3}$ hold in general, and $\eqref{fm3} \Rightarrow \eqref{fm1}$ holds if $F$ is perfect.
\end{theorem}

We note a couple of useful lemmas on operations that preserve quotients (and hence flag matroids).

\begin{lemma}[{\cite[Theorem 2.14]{Jarra+Lorscheid:2022}}]\label{lem:minor+quotient}
Let $M,N$ be strong $F$-matroids on $E$ and $A \subseteq E$.
If $N \quotient M$, then $N \setminus A \quotient M \setminus A$ and $N / A \quotient M / A$.
\end{lemma}

\begin{lemma}\label{lem:loop+coloop+quotient}
Let $M,N$ be strong $F$-matroids on $E$ and $A$ a set disjoint from $E$.
If $N \quotient M$, then $N \loopsum A \quotient M \loopsum A$ and $N \coloopsum A \quotient M \coloopsum A$.
\end{lemma}
\begin{proof}
Let $M, N$ be $F$-matroids with respective Pl\"ucker vectors $P \colon {E \choose r}\rightarrow F$ and $Q :{E \choose s}\rightarrow F$.
As $N \quotient M$, they satisfy the Pl\"ucker incidence relations \eqref{eq:incidence+relations}.

Consider $M\loopsum A$ and $N\loopsum A$ with respective Pl\"ucker vectors $\tP \colon{E\sqcup A \choose r} \rightarrow F$ and $\tQ :{E\sqcup A \choose s} \rightarrow F$.
By Proposition~\ref{prop:loop+coloop}, these satisfy $\tP(I) = P(I)$ if $I \cap A = \emptyset$ and $\0$ otherwise; similarly for $\tQ$.
Consider the incidence relation $\cP^{r,s}_{I;J}$ where $I, J \in E \sqcup A$ of cardinality $r-1$ and $s+1$ respectively.
If $(I \cup J) \cap A = \emptyset$, then
\[
\sum_{j \in J \setminus I} \sign(j;I,J) \cdot \tP(I+j) \cdot \tQ(J-j) = \sum_{j \in J \setminus I} \sign(j;I,J) \cdot P(I+j) \cdot Q(J-j) \in N_F \, .
\]
If $(I \cup J) \cap A \neq \emptyset$, then at least one of $\tP(I+j)$ and $\tQ(J-j)$ is equal to $\0$ for each term of $\cP^{r,s}_{I;J}$.
In either case, this incidence relation is satisfied by $\tP$ and $\tQ$.
Iterating over all such incidence relations gives that $N \loopsum A \quotient M \loopsum A$.

Now consider $M\coloopsum A$ and $N\coloopsum A$ with respective Pl\"ucker vectors $\tP\colon {E\sqcup A \choose r+ |A|} \rightarrow F$ and $\tQ \colon {E\sqcup A \choose s+ |A|}\rightarrow F$.
By Proposition~\ref{prop:loop+coloop}, these satisfy $\tP(I) = P(I\setminus A)$ if $A\subseteq I$ and $\0$ otherwise; similarly for $\tQ$.
Consider the incidence relation $\cP^{r+|A|,s+|A|}_{I;J}$ where $I, J \in E \sqcup A$ of cardinality $r + |A| -1$ and $s+ |A| + 1$ respectively.
If $A \subseteq I \cap J$, then 
\[
\sum_{j \in J \setminus I} \sign(j;I,J) \cdot \tP(I+j) \cdot \tQ(J-j) = \sum_{j \in (J\setminus A) \setminus (I\setminus A)} \sign(j;I,J) \cdot P(I\setminus A+j) \cdot Q(J\setminus A-j) \in N_F \, ,
\]
which is the relation $\cP^{r,s}_{I\setminus A;J\setminus A}$ up to a possible sign factor.
If $A \nsubseteq I \cap J$, then at least one of $\tP(I+j)$ and $\tQ(J-j)$ is equal to $\0$ for each term of $\cP^{r+|A|,s+|A|}_{I;J}$.
In either case, this incidence relation is satisfied by $\tP$ and $\tQ$.
Iterating over all such incidence relations gives that $N \coloopsum A \quotient M \coloopsum A$.
\end{proof}

As with Section~\ref{sec:initial+matroid}, we will focus on flag matroids over tropical extensions of tracts.
Again, a key notion in our toolkit will be initial flag matroids.

\begin{definition}
Let $\bM = (M_1, \dots, M_k)$ be a flag $F[\Gamma]$-matroid. % with Pl\"ucker vector sequence $\bP = (P_1, \dots, P_k)$.
%Let $\bP = (P_1,\dots, P_k)$ be a sequence of tuples $P_i \in F[\Gamma]^{n \choose r_i}$ where $0 \leq r_1 < \dots < r_k \leq n$.
Given some $\bu \in (\Gamma \cup \{\infty\})^E$, we define its \emph{initial flag matroid} $\bM^\bu = (M_1^\bu, \dots, M_k^\bu)$ to be the flag $F$-matroid obtained as a sequence of initial matroids of constituents of $\bM$.%, as defined in Definitions~\ref{def:toric+initial+tuple} and \ref{def:initial+tuple}.
%the initial sequence $\bP^\bu = (P_1^\bu, \dots, P_k^\bu)$ where $P_i^\bu \in F^{n \choose r_i}$ is the initial tuple defined in Definition~\ref{def:initial+tuple}.
\end{definition}

As in the matroid case, it is not a priori clear that $\bM^\bu$ is a flag $F$-matroid.
%As in the matroid case, we see that if $\bP$ is the Pl\"ucker vector sequence of a flag $F[\Gamma]$-matroid $\bM$, then $\bP^\bu$ is the Pl\"ucker vectors sequence of an flag $F$-matroid $\bM^\bu$, that we call the \emph{initial flag matroid} of $M$ with respect to $\bu$.
The following proposition demonstrates that this is the case.
\begin{proposition}\label{prop:ext+implies+initial+flag}
    Let $\bM = (M_1, \dots, M_k)$ be a flag $F[\Gamma]$-matroid. % with Pl\"ucker vector sequence $\bP =(P_1, \dots, P_k)$.
    Then $\bM^\bu = (M_1^\bu, \dots, M_k^\bu)$ is a flag $F$-matroid for all $\bu \in (\Gamma \cup \{\infty\})^E$.
%    Then $\bP^\bu =(P_1^\bu, \dots, P_k^\bu)$ is the Pl\"ucker vector sequence of the flag $F$-matroid $\bM^\bu = (M_1^\bu, \dots, M_k^\bu)$ for all $\bu \in (\Gamma \cup \{\infty\})^n$.
\end{proposition}
To prove Proposition~\ref{prop:ext+implies+initial+flag} (and the delayed  Proposition~\ref{prop:ext+implies+initial}), we require the following two technical lemmas.

\begin{lemma}\label{lem:incidence+pairs}
    Let $P \colon {E \choose r} \rightarrow F[\Gamma]$ and $Q \colon {E \choose s} \rightarrow F[\Gamma]$ be arbitrary functions.
    Given some $\bu \in \Gamma^E$, suppose that $I \in \underline{P^\bu}$ and $J \in \underline{Q^\bu}$.
    For any sets $I' \in {E \choose r}, J' \in {E \choose s}$ such that $I \cup J = I'\cup J'$ and $I \cap J = I' \cap J'$, we have
    \[
    |P(I)| + |Q(J)| \leq |P(I')| + |Q(J')| \, ,
    \]
    with equality if and only if $I' \in \underline{P^\bu}$ and $J' \in \underline{Q^\bu}$.
\end{lemma}

%\begin{lemma}\label{lem:incidence+pairs}
%    Let $P = (a_I, \gamma_I) \in F[\Gamma]^{n \choose r}$ and $\bq = (b_J, \delta_J) \in F[\Gamma]^{n \choose s}$.
%    Given some $\bu \in \Gamma^n$, suppose that $I \in D_\bu(P)$ and $J \in D_\bu(\bq)$.
%    For any sets $I' \in {n \choose r}, J' \in {n \choose s}$ such that $I + j = I'\cup J'$ and $I \cap J = I' \cap J'$, we have
%    \[
%    \gamma_{I} + \delta_{J} \leq \gamma_{I'} + \delta_{J'} \, ,
%    \]
%    with equality if and only if $I' \in D_\bu(P)$ and $J' \in D_\bu(\bq)$.
%\end{lemma}
\begin{proof}
As $I \in \underline{P^\bu}$ and $J \in \underline{Q^\bu}$, we immediately have $|P(I)| + |Q(J)|$ is finite.
From the definitions of $P^\bu$ and $Q^\bu$, we have
    \begin{align}\label{eq:minimizer+rels}
    |P(I)| - \sum_{i \in I} u_i + |Q(J)| - \sum_{j \in J} u_j \leq |P(I')| - \sum_{i \in I'} u_i + |Q(J')| - \sum_{j \in J'} u_j \, ,
    \end{align}
with equality if and only if $I' \in \underline{P^\bu}$ and $J' \in \underline{Q^\bu}$.
From the relations between $I,J$ and $I', J'$, we deduce that
\[
\sum_{i \in I} u_i + \sum_{j \in J} u_j = \sum_{k \in I \cup J} u_k + \sum_{\ell \in I \cap J} u_\ell = \sum_{k \in I' \cup J'} u_k + \sum_{\ell \in I' \cap J'} u_\ell = \sum_{i \in I'} u_i + \sum_{j \in J'} u_j \, .
\]
Therefore we can freely cancel the $\bu$ values from either side of \eqref{eq:minimizer+rels} to get the claim.
\end{proof}

\begin{lemma}\label{lem:ext+implies+initial}
    Let $P \colon {E \choose r} \rightarrow F[\Gamma]$ and $Q \colon {E \choose s} \rightarrow F[\Gamma]$ be arbitrary functions with $r \leq s$.
    Suppose $P$ and $Q$ satisfy the Pl\"ucker incidence relation $\cP^{r,s}_{I;J}$ for some $I \in {E \choose r-1}$ and $J \in {E \choose s+1}$, i.e.
\[
\sum_{j \in J \setminus I} \sign(j;I,J) \cdot P(I+j)\cdot Q(J-j) \in N_{F[\Gamma]} \, .
\]
    Then $P^\bu \colon {E \choose r} \rightarrow F$ and $Q^\bu \colon {E \choose s} \rightarrow F$ satisfy $\cP^{r,s}_{I;J}$ for all $\bu \in \Gamma^E$, i.e.
    \[
\sum_{j \in J \setminus I} \sign(j;I,J) \cdot P^\bu(I+j)\cdot Q^\bu(J-j) \in N_F \, .
\]
\end{lemma}

\begin{proof}
By definition, $P$ and $Q$ satisfying $\cP^{r,s}_{I,J}$ implies
    \begin{align*} \label{eq:minimizer}
    &\sum_{j \in K} \sign(j;I,J) \cdot \theta(P(I+ j)) \cdot \theta(Q(J - j)) \in N_{F} \nonumber\\
    \text{ where } \quad &K = \left\{ j \in J \setminus I \mid |P(I + j)| + |Q(J - j)|  \text{ finite and minimal}\right\} \, .
    \end{align*}
    Fix some arbitrary $\bu \in \Gamma^E$.
    If there exists no $j \in K$ such that $I + j \in \underline{P^\bu}$ and $J -j \in \underline{Q^\bu}$, then every term of $\sum_{j \in J \setminus I} \sign(j;I,J) \cdot P^\bu(I + j) \cdot Q^\bu(J -j)$ is zero and so is trivially in $N_F$.
    Suppose there exists $j^*$ such that $I + j^* \in \underline{P^\bu}$ and $J -j^* \in \underline{Q^\bu}$.
    Then by Lemma~\ref{lem:incidence+pairs}, we have
    \[
    K = \big\{ j \in J \setminus I \; \big| \; |P(I + j)| + |Q(J -j)| = |P(I + j^*)| + |Q(J -j^*)| \big\} \, .
    \]
    We can deduce that
    \[
    \sum_{j \in J \setminus I} \sign(j;I,J) \cdot P^\bu(I + j) \cdot Q^\bu(J -j) = \sum_{j \in K} \sign(j;I,J)\cdot \theta(P(I + j)) \cdot \theta(Q(J -j)) \in N_{F} \, ,
    \]
    as for any $j \notin K$, either $I + j \notin \underline{P^{\bu}}$ or $J -j \notin \underline{Q^{\bu}}$, and so the corresponding term is zero.
\end{proof}

\begin{proof}[Proof of Proposition~\ref{prop:ext+implies+initial}]
First consider the toric case where $\bu \in \Gamma^E$.
Recall that the Pl\"ucker relation $\cP^r_{I,J}$ is just the Pl\"ucker incidence relation $\cP^{r,r}_{I,J}$.
As such, Lemma~\ref{lem:ext+implies+initial} implies if $P$ satisfies all the Pl\"ucker relations, then so does $P^\bu$.

For the non-toric case, if $M$ is a $F[\Gamma]$-matroid then $(M/Z_\bu)^{\bu^\circ}$ is an $F$-matroid by Definition~\ref{def:minor} and the toric case.
This combined with Proposition~\ref{prop:loop+coloop} implies that $M^\bu$ is an $F$-matroid.
\end{proof}

\begin{proof}[Proof of Proposition~\ref{prop:ext+implies+initial+flag}]
As each part of $\bM^\bu$ is an $F$-matroid by Proposition~\ref{prop:ext+implies+initial}, suffices to show they form pairwise quotients.
We show that $M_j \quot M_i$ implies that $M_j^\bu \quot M_i^\bu$.
In the toric case, this follows directly from Lemma~\ref{lem:ext+implies+initial}.

In the non-toric case, Lemma~\ref{lem:minor+quotient} states that contracting $Z^\bu$ preserves the quotient, i.e. $M_j/Z_\bu \quot M_i/Z_\bu$.
The toric case tells us that $(M_j/Z_\bu)^{\bu^\circ} \quot (M_i/Z_\bu)^{\bu^\circ}$.
Applying Lemma~\ref{lem:loop+coloop+quotient} gives us the final step
\[
M_j^\bu = (M_j/Z_\bu)^{\bu^\circ} \coloopsum Z_\bu \quot (M_i/Z_\bu)^{\bu^\circ} \coloopsum Z_\bu = M_i^\bu \, .
\]
\end{proof}

The converse to Proposition~\ref{prop:ext+implies+initial+flag} cannot hold in general.
A strong $F$-matroid is a flag $F$-matroid with only one constituent, and so Example~\ref{ex:strong+counterexample} gives a counterexample to this.
However, as with strong matroids, we can show that the converse always holds over perfect tracts.

\begin{theorem} (Theorem~\ref{thm:B})
Let $F$ be a perfect tract and $\Gamma$ an ordered abelian group.
Then $\bM = (M_1, \dots, M_k)$ is a flag $F[\Gamma]$-matroid if and only if $\bM^\bu = (M_1^\bu, \dots, M_k^\bu)$ is a flag $F$-matroid  for all $\bu \in (\Gamma \cup \{\infty\})^E$.
\end{theorem}

\begin{proof}
One direction follows from Proposition~\ref{prop:ext+implies+initial+flag}.

For the converse, suppose that $\bM$ is not a flag matroid, and so there exists $i < j$ such that $M_i$ is not a quotient of $M_j$.
As $F$ is perfect, Theorem~\ref{thm:perfect+tract} implies $F[\Gamma]$ is also perfect and so by Theorem~\ref{thm:flag+equivalence} there exists $X \in \cV^*(M_i) \setminus \cV^*(M_j)$.
Applying Proposition~\ref{prop:covectors}, we see that $\theta(X) \in \cV^*(M_i^{|X|}) \setminus \cV^*(M_j^{|X|})$, hence $\bM^{|X|}$ is not a flag matroid.
\end{proof} 

\begin{remark}
Valuated flag matroids are studied in both~\cite{Haque:2012} and \cite{Brandt+Eur+Zhang:2021}.
In particular, \cite[Theorem A]{Brandt+Eur+Zhang:2021} gives four equivalent characterisations of valuated flag matroids, in terms of (a) Pl\"ucker incidence relations, (b) consecutive valuated matroid quotients, (c) initial flag matroids and (d) chains of covectors.
Combining Theorem~\ref{thm:B} with results from \cite{Jarra+Lorscheid:2022} gives an extension of this theorem to all flag $F[\Gamma]$-matroids where $F$ is perfect.

Jarra and Lorscheid's Theorem~\ref{thm:flag+equivalence} recovers the equivalences (a), (b) and (d) of this theorem for all flag $F$-matroids where $F$ is perfect; in fact, they make this connection to flag $\TT$-matroids in \cite[Section 2.8]{Jarra+Lorscheid:2022}.
However, coupled with Theorem~\ref{thm:perfect+tract}, we recover the equivalences (a), (b) and (d) for all flag $F[\Gamma]$-matroids where $F$ is perfect.
The remaining characterisation (c) is extended to all flag $F[\Gamma]$-matroids where $F$ is perfect by Theorem~\ref{thm:B}.
\end{remark}

\subsection{Positroids over tracts} \label{sec:positroids}

We briefly discuss how one can consider positroids over tracts.
We begin by recalling the usual definition of a positroid, before reframing it slightly to fit into the framework of tracts.

Recall that the usual definition of a positroid is a matroid that is representable by a totally non-negative real matrix, i.e. all of its maximal minors are non-negative.
In our language, it is a $\KK$-matroid $M$ with Pl\"ucker vector $P_M \colon{E \choose r} \rightarrow \KK$ such that there exists an $\RR$-matroid $N$ with non-negative Pl\"ucker vector $P_M\colon {E \choose r} \rightarrow \RR_{\geq 0}$ such that $t(P_M) =P_N$, where $t \colon \RR \rightarrow \KK$ the trivial homomorphism.

The trivial homomorphism is very coarse as it forgets all of the non-negativity of $P_N$.
Instead, we preserve this sign information by considering the image of $P_N$ in the sign homomorphism
\[
\sgn \colon \RR \rightarrow \SS \quad , \quad \sgn(a) = \begin{cases}
\1 & a \in \RR_{>0} \\ -\1 & a \in \RR_{<0} \\ \0 & a = 0
\end{cases} \, .
\]
Its image $\sgn(P_N) \colon {E \choose r} \rightarrow \SS_{\geq \0}$ gives a non-negative Pl\"ucker vector over the sign hyperfield, i.e. $\SS_{\geq \0} = \{\0,\1\}$.
As such, we will consider a positroid to be an $\SS$-matroid whose Pl\"ucker vector is non-negative.
Such matroids are known in the literature as \emph{positively oriented matroids}, and their equivalence to positroids was already known~\cite{Ardila+Rincon+Williams:2017}.

To extend this definition to other tracts, we need the notion of an ordering on a tract.
\begin{definition}\label{def:ordering}
    An \emph{ordering} on a tract $F$ is a subset $F_{>\0} \subseteq F^\times$ such that 
    \[
    F_{>\0} \cdot F_{>\0} \subseteq F_{>\0} \quad , \quad F^\times = F_{>\0} \sqcup -F_{>\0} \quad , \quad N_F \cap \NN[F_{>\0}] = \{\0\} \, .
    \]
    An \emph{ordered tract} is a tract $F$ along with distinguished ordering $F_{>\0}$.
    We define $F_{\geq \0} = F_{>\0} \cup \{\0\}$ and $F_{<\0} = -F_{>\0}$.
\end{definition}
Orderings on tracts are a direct generalisation of orderings on (hyper)fields.
When $F$ is a hyperfield, the first two conditions on $F_{>\0}$ are the same as over arbitrary tracts, and the third condition $F_{>\0} \boxplus F_{>\0} \subseteq F_{>\0}$ insists sums of positive elements are positive.
This condition breaks down when $F$ is a tract, but $N_F \cap \NN[F_{>\0}] = \{\0\}$ replaces it by ensuring we cannot have a null sum of positive elements.

\begin{remark}
As discussed for hyperfields in~\cite{Anderson+Davis:2019}, orderings can be equivalently characterised by tract homomorphisms to the sign hyperfield.
Explicitly, $F_{>\0}$ is an ordering on $F$ if and only if there exists a tract homomorphism $f\colon F \rightarrow \SS$ such that $F_{>\0} = f^{-1}(\1)$.
Not all tracts can be ordered, in the same way not all fields can be ordered.
For example, there is no homomorphism from the phase hyperfield $\Phi$ to the sign hyperfield.
\end{remark}

\begin{definition} \label{defn:positroid}
    Let $F$ be an ordered tract with ordering $F_{>\0}$.
    A weak (resp. strong) \emph{$F$-positroid} $M$ is a weak (resp. strong) $F$-matroid whose Pl\"ucker vector $P_M\colon {E \choose r} \rightarrow F_{\geq \0}$ is non-negative.
    A \emph{flag $F$-positroid} is a flag $F$-matroid $\bM = (M_1, \dots, M_k)$ whose parts $M_i$ are all strong $F$-positroids.
\end{definition}

%
%
%Note that for any ordered tract $F$, there is a `modulus' homomorphism
%\todo{new notation here}
%\[
%|\cdot| \colon F^\times \mapsto F_{>0} \, , \, |g| = \begin{cases}
%    g & g \in F_{>0} \\
%    -g & g \in F_{<0}
%\end{cases} \, .
%\]
%We let $F^+$ denote the tract in the image of the homomorphism.
%
%\begin{example}
%    A positroid is a vector $P \in \SS_{\geq 0}^{n \choose r}$ satisfying the signed Pl\"ucker relations.
%    While formally an $\SS$-matroid, we can consider its image in the modulus map to obtain a $\KK$-matroid, as $|\SS| = \KK$.
%\end{example}
%
%\begin{example}
%    $\RR$-positroids are elements of the totally nonnegative Grassmannian.
%    However, we can also consider them as $|\RR|$-matroids: this is the factor hyperfield $\RR/\langle -1, 1 \rangle$.
%    This shows we at least need to consider matroids over hyperfields to allow for this machinery.
%    \end{example}

\begin{example}\label{ex:tropical+positroid}
The signed tropical hyperfield $\TT_\pm \cong \SS[\RR]$ is an ordered tract, where the ordering is inherited from the ordering on $\SS$.
Explicitly
\[
(\TT_\pm)_{>\0} = \SetOf{(\1,\gamma)}{\gamma \in \RR} \, .
\]
As such, $\TT_\pm$-positroids are correspond to non-negative Pl\"ucker vectors $P\colon {E \choose r} \rightarrow (\TT_\pm)_{\geq \0}$ satisfying the Pl\"ucker relations over $\TT_\pm$.

$\TT_\pm$-positroids have been studied in the literature via tropicalisation of the non-negative Grassmannian as positive tropical Pl\"ucker vectors~\cite{Arkani-Hamed+Lam+Spradlin:2020, Lukowski+Parisi+Williams:2020}.
Here, they are viewed as Pl\"ucker vectors in the tropical semiring satisfying the `positive Pl\"ucker relations'.
These are equivalent to the signed tropical Pl\"ucker relations we use, but imposed on the tropical semiring.
Moreover, the tropical semiring can be viewed as the positive part of the signed tropical hyperfield.
\end{example}

Note that in Example~\ref{ex:tropical+positroid}, the ordering on $\SS[\RR]$ was inherited from the ordering on $\SS$.
This can be extended to arbitrary extensions of ordered tracts.
Explicitly, if $F$ is an ordered tract with ordering $F_{>\0}$, it is straightforward to check that $F[\Gamma]$ is also an ordered tract with ordering
\[
F[\Gamma]_{>\0} = \{(a,\gamma) \mid a \in F_{>\0} \, , \, \gamma \in \Gamma\} \, .
\]
From now on, we will always assume this ordering on $F[\Gamma]$.

Given this setup, the positroid versions of Theorems~\ref{thm:A} and~\ref{thm:B} are straightforward.

\begin{theorem}(Theorem \ref{thm:C}) \label{thm:positroid}
Let $F$ be an ordered tract and $\Gamma$ an ordered abelian group.
$M$ is a weak $F[\Gamma]$-positroid if and only if $M^\bu$ is a weak $F$-positroid for all $\bu \in \Gamma^E$.

If $M$ is a strong $F[\Gamma]$-positroid, then $M^\bu$ is a strong $F$-positroid for all $\bu \in \Gamma^E$.
Conversely, if $F$ is perfect and $M^\bu$ is a strong $F$-positroid for all $\bu \in \Gamma^E$, then $M$ is a strong $F[\Gamma]$-positroid.

The sequence $\bM = (M_1, \dots, M_k)$ is a flag $F[\Gamma]$-positroid if and only if $\bM^\bu = (M_1^\bu, \dots, M_k^\bu)$ is a flag $F$-positroid  for all $\bu \in (\Gamma \cup \{\infty\})^E$.
\end{theorem}
\begin{proof}
    It is straightforward from the definitions that $P \colon {E \choose r} \rightarrow F[\Gamma]_{\geq \0}$ if and only if $P^\bu\colon {E \choose r} \rightarrow F_{\geq \0}$.
    As such, the results just follow from Theorem~\ref{thm:matroid+subdivision}, Proposition~\ref{prop:ext+implies+initial}, Theorem~\ref{thm:A} and Theorem~\ref{thm:B} respectively.
\end{proof}

\begin{remark}
In the case where $F[\Gamma] = \SS[\RR] = \TT_\pm$, these results were already known by a number of authors.
The claim that $\SS[\RR]$-positroids have decompose into $\SS$-positroids was proved independently by~\cite{Arkani-Hamed+Lam+Spradlin:2020} and~\cite{Lukowski+Parisi+Williams:2020}.
The claim for flag $\SS[\RR]$-positroids was can be deduced for full flags from~\cite{Joswig+Loho+Luber+Olarte:2023} and for partial flags from~\cite{Boretsky+Eur+Williams:2023}.
\end{remark}

\begin{remark}
While our definition of an $\SS$-positroid agrees with the usual definition of positroid, our definition of flag $\SS$-positroid may differ from the definition of a flag positroid given in \cite[Definition 1.2]{Boretsky+Eur+Williams:2023}.
Explicitly, a flag matroid $(M_1, \dots, M_k)$ on $[n]$ of ranks $(r_1, \dots, r_k)$ is a flag positroid if it has a realization by a real matrix $A$ where each $r_i \times n$ submatrix formed by the first $r_i$ rows has non-negative maximal minors.
This implies that a flag positroid can always be extended to a full flag positroid, i.e. one where $r_{i+1} = r_{i} + 1$.
It is unknown whether this holds for flag $\SS$-positroids~\cite[Question 1.6]{Boretsky+Eur+Williams:2023}.

In general, it is an open question for which tracts $F$ one can extend a flag $F$-matroid to a full flag $F$-matroid~\cite[Question 2]{Jarra+Lorscheid:2022}.
For example, it is always possible for $\KK$-matroids.
However, \cite{Richter-Gebert:1993} demonstrated an example of two $\SS$-matroids $M$ and $N$ with $M \quotient N$ that could not be completed to a full flag of $\SS$-matroids.
Even for $\TT$-matroids, this remain unresolved.
\end{remark}

\section{Enriched tropical linear spaces} \label{sec:tropical+linear+spaces}

In this section, we apply the theory developed in the previous sections to tropical linear spaces.
We briefly recall some necessary preliminaries of tropical geometry, framed in the language of tracts.
We then introduce the notion of an \emph{enriched tropical linear space}, a linear space over a tropical extension of a tract.
We deduce a structure theorem for enriched tropical linear spaces in Theorem~\ref{thm:D}, and show how images of linear spaces in enriched valuation maps take on the structure of an enriched tropical linear space.
We finish by giving a number of key examples of enriched tropical linear spaces that have arisen in the literature, and give intuition on how one can view and describe them.

\subsection{Tropical geometry}\label{sec:tropical+geometry}

    A \emph{valuation} on a field $\FF$ is a surjective map $\val\colon \FF \rightarrow \Gamma \cup \{\infty\}$ to an ordered abelian group $(\Gamma,+)$ satisfying
    \begin{itemize}
        \item $\val(a) = \infty \Leftrightarrow a = 0$,
        \item $\val(ab) = \val(a) + \val(b)$,
        \item $\val(a+b) \geq \min(\val(a),\val(b))$, with equality if $\val(a) \neq \val(b)$.
    \end{itemize}
    The pair $(\FF, \val)$ is called a \emph{valued field}.
    The group $\Gamma$ is called the \emph{value group} of $\FF$.
    
Valuations naturally fit in the framework of tracts in the following way.
By identifying $\Gamma \cup \{\infty\}$ with $\KK[\Gamma]$, we can instead consider $\val$ as a map to $\KK[\Gamma]$.
It is easy to verify that $(\FF,\val)$ is a valued field if and only if $\val\colon \FF \rightarrow \KK[\Gamma]$ is a tract homomorphism.

% a+b -a - b \in N_\FF \Leftrightarrow v(a+b) + v(a) + v(b) \in N_\TT
    
\begin{example}\label{ex:val}
    Our standard example of a valued field will be the field of \emph{Hahn series} with value group $\Gamma$ over a field $\FF$.
    These are formal power series whose exponents are elements of the ordered abelian group $\Gamma$
    \[
    \hseries{\FF}{\Gamma} = \SetOf{\sum_{g \in G} c_g t^g}{c_g \in \FF^\times \, , \, G \subseteq \Gamma \text{ well-ordered } } \, ,
    \]
    along with the zero element.

    Given a Hahn series $\omega = \sum_{g \in G}c_g t^g$, its `first' term $c_\gamma t^{\gamma}$ where $\gamma = \min(g \mid g \in G)$ is referred to as the \emph{leading term} $\lt(\omega)$ of $\omega$, where $c_\gamma$ is the \emph{leading coefficient} $\lc(\omega)$ and $\gamma$ is the \emph{leading power} $\lp(\omega)$.
    The valuation on $\hseries{\FF}{\Gamma}$ maps zero to $\infty$, and a non-zero series to its leading power:
   \begin{align} \label{eqn:val-trop-ext}
       \val\colon \hseries{\FF}{\Gamma} &\rightarrow \Gamma \cup\{\infty\} \\ 
    \sum_{g \in G} c_g t^g &\mapsto \gamma = \min(g \mid g \in G) \, .\nonumber 
   \end{align} 
\end{example}

In the framework of tracts, we can generalise valuations by considering homomorphisms to other tropical extensions.
These encode additional information about field elements, which we refer to as \emph{enriched valuations}.

\begin{definition}
An \emph{enriched valuation} on a field $\FF$ is a surjective homomorphism $\nu \colon \FF \rightarrow F[\Gamma]$ to a tropical extension of a tract $F$.
\end{definition}

\begin{example}\label{ex:sval}
    For a ordered field $\FF$, the valuation map~\eqref{eqn:val-trop-ext} can be enriched to the \emph{signed valuation}, which records the sign of a Hahn series, or equivalently the sign of its leading coefficient, i.e.
\begin{align*}
\sval\colon \hseries{\FF}{\Gamma} & \rightarrow \SS[\Gamma] ,\\ \sum_{g \in G} c_g t^g  &\mapsto (\sgn(c_{\gamma}), \gamma) \, , \quad \gamma = \min(g \mid g \in G) \, ,
\end{align*}
where $\sgn$ sends positive field elements to $\1_\SS$ and negative field elements to $-\1_\SS$.
When $\Gamma = \RR$, the image of the signed valuation map is precisely the signed tropical hyperfield $\TT_{\pm}$.
\end{example}

\begin{example}\label{ex:fine+val}
    Consider the field of Hahn series $\hseries{\FF}{\Gamma}$ over an arbitrary field.
    We can enrich the usual valuation map~\eqref{eqn:val-trop-ext} on $\hseries{\FF}{\Gamma}$ by defining the \emph{fine valuation} $\fval$ that remembers the leading coefficient as well as the leading power, i.e.
    \begin{align*}
        \fval\colon \hseries{\FF}{\Gamma} &\rightarrow \FF[\Gamma] \\ \sum_{g \in G} c_g t^g &\mapsto (c_\gamma , \gamma)\, , \quad \gamma = \min(g \mid g \in G) \, .
    \end{align*}
    We can view the fine valuation of a Hahn series as a first-order approximation that recalls only the leading term.
\end{example}

\begin{remark}\label{rem:stringent}
Examples~\ref{ex:val},~\ref{ex:sval} and~\ref{ex:fine+val} are particularly nice enriched valuations, as their targets are \emph{stringent hyperfields}, those hyperfields where the only multivalued sums are sums of additive inverses, i.e. $|a \boxplus b| > 1$ implies that $b = -a$.
In fact, these are all of the stringent hyperfields by the classification \cite[Theorem 4.10]{Bowler+Su:2021}.
Additionally, matroids over stringent hyperfields have especially nice vector axioms as given in~\cite{Bowler+Pendavingh:2019}.
This will be invaluable when we consider the `enriched tropicalisation' of linear spaces (Proposition~\ref{prop:tropicalisation}).
\end{remark}

There are other very natural examples of enriched valuations to non-stringent hyperfields, but their behaviour may not be as controlled when we consider their geometry.
\begin{example}
Fixing $\FF = \CC$, the valuation map~\eqref{eqn:val-trop-ext} can be enriched to the \emph{phased valuation}, which records the phase or argument of a Hahn series, or equivalently the phase of its leading coefficient, i.e.
\begin{align*}
\phval\colon \hseries{\CC}{\Gamma} & \rightarrow \Theta[\Gamma] ,\\ \sum_{g \in G} c_g t^g  &\mapsto (\arg(c_{\gamma}), \gamma)\, , \quad \gamma = \min(g \mid g \in G) \, ,
\end{align*}
where $\arg$ maps a complex number to its argument as in Example~\ref{ex:phase}.
We emphasise that $\Theta$, and therefore $\Theta[\Gamma]$, is neither stringent nor perfect.
As such, the phased valuation is often badly behaved.
\end{example}
    
Finally, we review some basics of tropical geometry over the tropical hyperfield; for further details see~\cite{Maclagan+Sturmfels:2015}.
Let $(\FF,\val)$ be a valued field.
Given some polynomial $p = \sum_{a \in \ZZ^n} c_a Z^a \in \FF[Z_1, \dots, Z_n]$, its corresponding tropical polynomial is the function from $\TT^n$ to the power set $\cP(\TT)$ of $\TT$ defined as
\begin{align*}
\val_*(p)(X) = \bigboxplus_{a \in \ZZ^n} \val(c_a)\odot X_1^{\odot a_1} \odot \cdots \odot  X_n^{\odot a_n} \, .
\end{align*}
The \emph{tropical hypersurface} associated to $p$ is the set of points whose evaluation by $\val_*(p)$ contains $\infty$, i.e.
\begin{align*}
V(\val_*(p)) = \SetOf{X \in \TT^n}{\infty \in \val_*(p)(X)} \, .
\end{align*}
We call the coordinatewise image $\val(V(p))$ of the hypersurface $V(p) \subseteq \FF^n$ the \emph{tropicalisation of $V(p)$}.
The tropicalisation $\val(V(p))$ is always contained in the tropical hypersurface $V(\val_*(p))$, i.e. $\val(V(p)) \subseteq V(\val_*(p))$, with equality when $\FF$ is algebraically closed.

Given an ideal $I \subseteq \FF[Z_1,\dots, Z_n]$, its corresponding tropical ideal is
\[
\val_*(I) = \SetOf{\val_*(p)}{p \in I} \, .
\]
The \emph{tropical variety} associated to $I$ is
\[
V(\val_*(I)) = \SetOf{X \in \TT^n}{\infty \in \val_*(p)(X) \, \forall p \in I} = \bigcap_{p \in I} V(\val_*(p)) \, .
\]
Similarly to hypersurfaces, we call the coordinatewise image $\val(V(I))$ of the algebraic variety $V(I) \subseteq \FF^n$ the \emph{tropicalisation of $V(I)$}.
Again, we always have the tropicalisation is contained in the tropical variety $\val(V(I)) \subseteq V(\val_*(I))$ with equality when $\FF$ is algebraically closed.
We finally note that the restriction of the tropical variety $V(\val_*(I))$ to $\RR^n$ is a polyhedral complex with additional structure, see~\cite[Theorem 3.3.6]{Maclagan+Sturmfels:2015} for full details.

\subsection{Enriched tropical linear spaces} \label{sec:enriched+tls}

In this section, we use the results from Sections~\ref{sec:trop+ext+matroids} and~\ref{sec:further+matroids} to understand the geometry of linear spaces over various examples of tropical extensions.
We first motivate this by considering the relationship between linear spaces and matroids over a field

Linear spaces over a field $\FF$ are in one-to-one correspondence with $\FF$-matroids.
Let $L \subset \FF^E$ be an $r$-dimensional linear space; we give multiple equivalent ways to describe the $\FF$-matroid on ground set $E$ of rank $r$ associated to $L$.
\begin{itemize}
\item The simplest characterisation is in terms of covectors.
By~\cite[Proposition 2.19]{Anderson:2019}, there exists some $\FF$-matroid $M$ of rank $r$ such that $L = \cV^*(M)$.
\item As $\FF$ is perfect, the dual space $L^\perp$ to $L$ is precisely the set of vectors of $M$, i.e. $L^\perp = \cV(M)$.
Equivalently, this is the set of covectors of the dual matroid $M^*$.
\item The linear space is cut out by a linear ideal $I_M$, i.e. $L = V(I_M)$ where $I_M$ generated by linear forms.
The set of all linear forms in $I_M$ is precisely the set of vectors of $M$, i.e.
\begin{equation} \label{eq:circuit+ideal}
\sum_{i\in E} Y_i \cdot Z_i \in I_M \, \Longleftrightarrow \, Y \in \cV(M) \, .
\end{equation}
The linear forms of minimal support in $I_M$ are precisely the circuits $\cC(M)$ of $M$.
Moreover, picking exactly one circuit for each possible support gives us a minimal universal Gr\"obner basis for $I_M$~\cite{Sturmfels:1996}.
\item The Pl\"ucker embedding maps a $r$-dimensional linear space in $\FF^E$ to a point in projective space  via the map
\begin{align} \label{eq:plucker+embedding}
\psi\colon\FF^E \rightarrow \PP^{{E \choose r}-1}_\FF \, , \quad L = \spn(A) \mapsto \left[\det(A|_J) \colon J \in {E \choose r}\right]
\end{align}
where $A \in \FF^{d \times |E|}$ is any matrix whose rowspan is $L$, and $A|_J$ the square matrix obtained by restricting to the columns labelled by $J$.
This map is well-defined, as varying $A$ only varies $\det(A|_J)$ by scalars.
The image $\psi(L)$ in the Pl\"ucker embedding is precisely the Pl\"ucker vector $P_M$ of $M$.

The image of $\psi$ is the Grassmannian $\Gr_\FF^{r,E}$, the moduli space of $r$-dimensional linear spaces in $\FF^E$.
It can be viewed as a projective variety or scheme cut out by the polynomial equations given by the Pl\"ucker relations $\cP^r_{I,J}$.
\end{itemize}

%Geometrically, Grassmannians are naturally associated to linear spaces.
%When $F = \FF$ is a field, $\Gr_\FF^{r,n}$ is the moduli space of $r$-dimensional linear spaces in $\FF^n$.
%It can be viewed as a projective variety or scheme cut out by the polynomial equations given by the Pl\"ucker relations $\cP^r_{I,J}$.
%Given an $r$-dimensional linear space $L \subseteq \FF^n$, its embedding inside the $\FF$-Grassmannian associates it with a Pl\"ucker vector and therefore an $\FF$-matroid $M$.
%The linear space $L$ and the matroid $M$ are related in the following way.
%The linear ideal that cuts out $L$ is generated by the circuits $\cC(M)$ of $M$~\cite[Lemma 4.1.4]{Maclagan+Sturmfels:2015}, i.e.
%\[
%L = V(I)\subseteq \FF^n \quad , \quad  I = \left\langle \sum_{i=1}^n C_i Z_i  \; \bigg| \; C \in \cC(M)\right\rangle \subseteq \FF[Z_1, \dots, Z_n] \, .
%\]
%As such, it follows that $L = \cV^*(M)$ is precisely the set of covectors of $M$, the set of tuples orthogonal to the circuits.
%Using the definition of perfect tract, one can also see that the orthogonal complement $L^\perp = \cV(M)$ is the linear space associated to the dual matroid $M^*$. \todo{Citation for this?}
%
%\begin{proposition}[\cite{Anderson:2019}]
%Let $\FF$ be a field, and $V$ a linear subspace of $\FF^n$.
%Then there is a strong $\FF$-matroid $M$ such that
%\begin{itemize}
%\item $V = \cV^*(M)$
%\item $V^\perp = \cV(M)$
%\item The projective coordinates of the Pl\"ucker embedding for $V$ constitute a Pl\"ucker vector $P_M$ for $M$.
%\end{itemize}
%Furthermore, every $\FF$-matroid arises this way.
%\end{proposition}

\begin{example} \label{ex:linear+space}
    Consider the two dimensional linear space $L \subset\hseries{\RR}{\RR}^4$ and dual linear space $L^\perp \subset\hseries{\RR}{\RR}^4$ defined as
    \[
    L = \spn
    \begin{bmatrix}
        1 & 0 & -2 & 2 \\
        0 & 1 &-1 & 1-t
    \end{bmatrix} \quad , \quad L^\perp = \spn
    \begin{bmatrix}
        2 & 1 & 1 & 0 \\
        2 & 1-t & 0 & -1
    \end{bmatrix} \, .
    \]
    This linear space defines a $\hseries{\RR}{\RR}$-matroid $N$ where $L = \cV^*(N)$ and $L^\perp = \cV(N)$.
    Its corresponding Pl\"ucker vector is
    \[
    P_N = [1 : -1 : 1-t : 2 : -2 : 2t] \in \PP^5_{\hseries{\RR}{\RR}} \, ,
    \]
    given by the Pl\"ucker embedding~\eqref{eq:plucker+embedding}.
%    The circuits $\cC(M)$ of $M$ are
%    \[
%    \left\{(1,1,1,0),(1,1-t,0,-1), (t,0,-1+t,-1), (0,t,1,1)\right\} \, 
%    \]
%    and all their scalings.
%    These can be computed directly from $L^\perp$ as veectors of minimal support via row reduction, or from the Pl\"ucker vector via \eqref{eq:circuits}.
    The ideal that cuts out the linear space is
    \begin{equation*}
    I_N = \langle \,
    2Z_1 + Z_2 + Z_3 \, , \,
    2Z_1 + (1-t)Z_2 -Z_4 \, , \,
    2tZ_1 - (1-t)Z_3 -Z_4 \, , \,
    tZ_2 + Z_3 + Z_4 \, 
    \rangle \, ,
    \end{equation*}
    where the generators listed are precisely (up to scaling) the circuits $\cC(N)$ of $N$.
    While any two of the forms minimally generate the ideal, the circuits form a minimal universal Gr\"obner basis for $I$.
\end{example}

We employ the same viewpoint to define linear spaces over tropical extensions of tracts via strong matroids.

\begin{definition}\label{def:linear+space}
Let $F[\Gamma]$ be the tropical extension of a tract $F$ by the ordered abelian group $\Gamma$.
An \emph{enriched tropical linear space} $\cL_M$ is the set of covectors of a strong $F[\Gamma]$-matroid $M$, i.e.
\begin{align*}
\cL_M &= \cV^*(M) = \cC(M)^\perp \, ,
\end{align*}
Equivalently, it is the set of tuples orthogonal to the circuits $\cC(M)$.
The \emph{dimension} of $\cL_M$ is the rank of $M$. % maximum cardinality of the support of a circuit in $\cC(M)$.
The \emph{toric part} $\cL_M^\circ$ of $\cL_M$ is the restriction to $(F[\Gamma] \setminus \{\infty\})^E$, i.e.
\[
\cL_M^\circ = \SetOf{X \in \cV^*(M)}{X_i \neq \infty \, \forall i \in E} \, .
\]
\end{definition}
If $F = \KK$, Definition~\ref{def:linear+space} agrees with the usual definition of a tropical linear space (albeit allowing for more general value groups than just $\RR$).
In these cases, we will drop the adjective enriched and just refer to $\cL_M$ as a tropical linear space.
Analogously to how enriched valuations can be viewed as valuations encoding additional information about field elements, enriched tropical linear spaces can be viewed as tropical linear spaces encoding additional information about the linear space over the field.
%This viewpoint will be made more explicit in Section~\ref{sec:enriched+examples} when we consider a number of examples of enriched tropical linear spaces.

Brandt-Eur-Zhang showed that tropical linear spaces can be characterised in a number of different ways~\cite[Theorem B]{Brandt+Eur+Zhang:2021}.
To highlight the structural similarities with tropical linear spaces, Theorem~\ref{thm:D} is an analogous characterisation for enriched tropical linear spaces over perfect tracts.
To give such a characterisation, we need a few additional notions.

Let $M$ be an $F$-matroid on ground set $E$.
Given a finite subset $\cX \subseteq F^E$, its \emph{span} is
\begin{equation}\label{eq:span}
{\rm span}(\cX) = \left\{Y \in F^E \ \Bigg| \ \exists \, \alpha_X \in F \text{ s.t. } Y_k - \sum_{X \in \cX} \alpha_X \cdot X_k \in N_F \; \forall k \in E\right\}
\end{equation}
We would like to describe the covectors of $M$ as a span of cocircuits.
However, this does not hold in general as spans are too coarse over tracts.
Instead, we follow the approach laid out in~\cite{Anderson:2019} by only considering spans of fundamental cocircuits.

For each $B \in \cB(\underline{M})$ of the underlying matroid $\underline{M}$ and $j \in B$, there exists a unique cocircuit of $\cC^*(M)$ denoted $D(j,B)$ such that $\underline{D(j,B)} \subseteq E \setminus B \cup \{j\}$ and $D_j(j,B) = \1$~\cite[Lemma 2.5]{Anderson:2019}.
This is called the \emph{fundamental $F$-cocircuit} with respect to $j$ and $B$, as its support is equal to the usual fundamental cocircuit of $\underline{M}$.
We write $\cC^*_B(M) := \{D(j,B) \in \cC^*(M) \mid j \in B\}$ for the collection of fundamental $F$-cocircuits with respect to $B$.

\begin{theorem}[Theorem \ref{thm:D}]
Let $M$ be a strong $F[\Gamma]$-matroid on $E$ where $F$ is a perfect tract.
The enriched tropical linear space $\cL_M$ corresponding to $M$ is equal to any of the following sets in $F[\Gamma]^E$:
\begin{enumerate}[label=(\roman*)]
\item $\bigcup\limits_{\emptyset \subseteq A \subseteq E} \cL_{M/A}^\circ \times \{\infty\}^A \, , $ \label{eq:i}
\item $\cC(M)^\perp \, , $ \label{eq:ii}
\item $\left\{X \in F[\Gamma]^E \mid \theta(X) \in \cV^*(M^{|X|}) \right\}\, , $ \label{eq:iii}
\item $\bigcap\limits_{B \in \cB(\underline{M})} {\rm{span}}(\cC^*_B(M)) \, .$ \label{eq:iv}
\end{enumerate}
\end{theorem}

\begin{proof}
The definition of enriched tropical linear space is given by the covectors of $M$, or equivalently \ref{eq:ii}.
Proposition~\ref{prop:covectors} shows this is equivalent to \ref{eq:iii} when $F$ is perfect.
The equivalence with \ref{eq:iv} follows from an alternative characterisation of covectors as `spans' of cocircuits~\cite[Lemma 2.16]{Anderson:2019}.

For the final equivalence with \ref{eq:i}, observe that we can write
\[
\cV^*(M) = \bigcup_{\emptyset \subseteq A \subseteq E} \left(\cV^*(M) \cap \left((F [\Gamma] \setminus \{\infty\})^{E \setminus A} \times \{\infty\}^A \right)\right) \, .
\]
By \cite[Proposition 4.5 (i)]{Anderson:2019}, the collection of covectors with infinite coordinates precisely at $A$ is $\cL_{M/A}^{\circ}$, giving the result.
\end{proof}

\begin{remark}\label{rem:span+bad}
When $F$ is a field, the covectors of an $F$-matroid are the span of the cocircuits.
Given this, one may be a bit dissatisfied with \ref{eq:iv}.
This is an unavoidable quirk of spans over hyperstructures that emerges already for $\TT$ and $\TT_{\pm}$.
We investigate this further in Section~\ref{sec:enriched+examples}.
\end{remark}

\begin{remark}
When $F$ is a perfect tract, the vectors and covectors of an $F$-matroid are orthogonal to one another.
As such, $\cL_M$ is also orthogonal to its dual linear space $\cL_{M^*} = \cV(M)$, as
\[
\cL_M = \cV^*(M) = \cV(M)^\perp = \cL_{M^*}^\perp \, .
\]
While we don't have an ideal structure over perfect tracts, we do have the next best thing where the set of `linear forms that $\cL_M$ vanishes on' has a linear space structure.
If $F$ is not a perfect tract, this is no longer the case.
\end{remark}

We next connect enriched tropical linear spaces with images of linear spaces in enriched valuation maps.
Recall that for tropical varieties, we have $\val(V(I)) \subseteq V(\val_*(I))$ for an ideal $I$ over a field $\FF$, with equality when $\FF$ is algebraically closed.
A similar paradigm holds for enriched tropical linear spaces as follows.
Let $L = \cV^*(N) \subseteq \FF^E$ be the linear space corresponding to some $\FF$-matroid $N$, and consider an enriched valuation $\nu \colon \FF \rightarrow F[\Gamma]$.
We call the coordinatewise image $\nu(L) \subseteq F[\Gamma]^E$ the \emph{tropicalisation of $L$}.
Moreover, there is an enriched tropical linear space $\cL_M = \cV^*(M)$ corresponding to the $\FF[\Gamma]$-matroid $M = \nu_*(N)$.
It follows from~\cite[Proposition 4.6]{Anderson:2019} that
\[
\nu(L) = \nu(\cV^*(N)) \subseteq \cV^*(\nu_*(N)) = \cL_M \, .
\]
The following proposition shows that this is an equality for various nice examples of enriched valuations.

\begin{proposition}\label{prop:tropicalisation}
Let $\FF$ be an infinite field and $\nu \colon \hseries{\FF}{\Gamma} \rightarrow F[\Gamma]$ be an enriched valuation equal to either the valuation $\val$, signed valuation $\sval$ or fine valuation $\fval$.
Then for any $\hseries{\FF}{\Gamma}$-matroid $N$, we have 
\[
\nu(\cV^*(N)) = \cV^*(\nu_*(N)) \, .
\]
\end{proposition}
\begin{proof}
Suppose $\nu(\cV^*(N))$ forms the set of covectors $\cV^*(M)$ of an $F[\Gamma]$-matroid $M$.
As $F[\Gamma]$ can be written as a quotient of a field by a subgroup of the units \cite[Theorem 7.5]{Bowler+Su:2021}, we can apply~\cite[Theorem 4.15]{Baker+Zhang:2022} to deduce $M$ and $\nu_*(N)$ have the same matroid rank.
Combining \cite[Corollary 2.6, Theorem 2.16]{Jarra+Lorscheid:2022}, we deduce that $M = \nu_*(N)$, and hence $\nu(\cV^*(N)) = \cV^*(\nu_*(N))$.
It remains to show that $\nu(\cV^*(N))$ forms the set of covectors of an $F[\Gamma]$-matroid.

We utilise the vector axioms in~\cite[Theorem 44]{Bowler+Pendavingh:2019} to show that $\nu(\cV^*(N))$ are the (co)vectors of an $F[\Gamma]$-matroid.
Their axioms (V0), (V1) and (V3) follow immediately from properties of linear spaces and surjective tract homomorphisms, hence it remains only to show (V2).
Define the composition map $\circ \colon F[\Gamma] \times F[\Gamma] \rightarrow F[\Gamma]$ by
\[
X \circ Y = \begin{cases}
Z & \text{ if } X \boxplus Y = \{Z\} \\
X & \text{ if } X = -Y \text{ and } F \in \{\KK,\SS\} \\
\0 & \text{ if } X = -Y \text{ and } F =\FF
\end{cases} \, ,
\] 
and extend it componentwise to a binary operation on $F[\Gamma]^E$.
To prove (V2), we fix some $X, Y \in \nu(\cV^*(N))$ such that $\underline{X \circ Y} = \underline{X} \cup \underline{Y}$, and show that $X \circ Y \in \nu(\cV^*(N))$.

Consider $\tilde{X},\tilde{Y} \in \cV^*(N)$ such that $\nu(\tilde{X}) = X$ and $\nu(\tilde{Y}) = Y$.
If either $X_i$ or $Y_i$ equals $\0$, the sum $X_i \boxplus Y_i$ is a singleton and so it immediately follows that $\nu(\tilde{X}_i + \tilde{Y}_i) = X_i \circ Y_i$.
Hence we may assume that $X_i, Y_i \neq \0$ for all $i \in E$.
If $X_i \boxplus Y_i = \{Z_i\}$, the tract homomorphism definition implies that 
$\nu(\tilde{X}_i + \tilde{Y}_i) = Z_i$.
In the case that $F = \FF$ (i.e. $\nu = \fval$), this immediately proves that $X \circ Y \in \nu(\cV^*(N))$, as the hypothesis that $\underline{X \circ Y} = \underline{X} \cup \underline{Y}$ implies that $X_i \neq -Y_i$ for all $i \in E$.

In the case where $F \in \{\KK,\SS\}$, we always have $\underline{X \circ Y} = \underline{X} \cup \underline{Y}$.
As $\tilde{X}$ and $\tilde{Y}$ are tuples of non-zero Hahn series, we write their leading terms as
\[
\lt(\tilde{X}_i) = c_i t^{\gamma_i} \quad , \quad \lt(\tilde{Y}_i) = d_{i} t^{\eta_i} \quad , \quad c_i, d_i \in \FF \, , \: \gamma_i, \eta_i \in \Gamma \, , \: i \in E \, .
\]
In the case that $F = \KK$ (i.e. $\nu = \val$), as $\FF$ an infinite field there exists a non-zero scalar $\lambda$ such that $\lambda \neq -c_i/d_i$ for all $i \in E$.
Then 
\[
\nu(\tilde{X}_i + \lambda\tilde{Y}_i) = \begin{cases} \min(\gamma_i, \eta_i) & \text{ if } \gamma_i \neq \eta_i \\ \gamma_i & \text{ if } \gamma_i = \eta_i \end{cases} \: = X_i\circ Y_i \in \nu(\cV^*(N)) \, .
\]
In the case where $F = \SS$ (i.e. $\nu = \sval$), we must have $\FF$ is an ordered field as there exists a homomorphism to the sign hyperfield via
\[
\FF \hookrightarrow \hseries{\FF}{\Gamma} \xrightarrow{\nu} \SS[\Gamma] \xrightarrow{\theta} \SS \, .
\]
As such, there exists a positive element $\lambda \in \FF_{>0}$ such that $\lambda < |c_i|/|d_i|$ for all $i \in E$.
Then $\sgn(c_i) = \sgn(c_i + \lambda d_i)$ and hence
\[
\nu(\tilde{X}_i + \lambda\tilde{Y}_i) = \begin{cases} (\sgn(c_i),\gamma_i) & \text{ if } \gamma_i < \eta_i \\
(\sgn(d_i),\eta_i) & \text{ if } \gamma_i > \eta_i \\ (\sgn(c_i),\gamma_i) & \text{ if } \gamma_i = \eta_i \end{cases} \, = X_i\circ Y_i \in \nu(\cV^*(N)) \, .
\]
\end{proof}

We note that equality will not hold for all enriched valuations.
For example, \cite[Example 4.7]{Anderson:2019} demonstrates a $\CC$-matroid whose image in the phase map $\ph\colon \CC \rightarrow\Theta$ satisfies $\ph(\cV^*(N)) \subsetneq \cV^*(\ph_*(N))$.
It is straightforward to extend this example to a $\hseries{\CC}{\RR}$-matroid by extending by scalars such that its image in the phase valuation $\phval\colon \hseries{\CC}{\RR} \rightarrow\Theta[\Gamma]$ satisfies $\phval(\cV^*(N)) \subsetneq \cV^*(\ph_*(N))$.

%\begin{remark}
%We note that equality does not hold for non-perfect hyperfields.
%If $\HH[\Gamma]$ is non-perfect, there exists a $\HH[\Gamma]$-matroid $M$ such that $X \in \cV^*(M), Y \in \cV(M)$ but $X \not\perp Y$...
%\todo[inline]{Needs to realizable}
%\end{remark}

\begin{remark}
In~\cite{Baker+Zhang:2022}, the authors describe a morphism of tracts $f \colon F \rightarrow F'$ as \emph{epic} if for any $F$-matroid $M$, we have $f(\cV(M)) = \cV(f_*(M))$.
They study epic morphisms to obtain results on various notions of matrix rank over tracts.
Proposition~\ref{prop:tropicalisation} can be rephrased as showing $\val$, $\sval$ and $\fval$ are all epic morphisms, the first of which was already shown in~\cite{Baker+Zhang:2022}.
\end{remark}

\begin{remark}
For varieties of degree greater than one, the equality in Proposition~\ref{prop:tropicalisation} no longer holds in general.
As with standard tropical varieties, one has to worry about various notions of algebraic closure; see~\cite{Maxwell+Smith:2023} for further details in the context of tropical extensions.
\end{remark}

\subsection{Examples of enriched tropical linear spaces} \label{sec:enriched+examples}

The remainder of this section is dedicated to studying tropical linear spaces over various tracts.
We show how our perspective recovers already known theory and extends it further to enriched  tropical linear spaces.

\paragraph{Tropical linear spaces}
Consider the tropical hyperfield $\TT = \KK[\RR]$.
The modulus map $|\cdot|$ is just the identity and so we shall drop it.
Moreover, recall that the phase map $\theta \colon \KK[\RR] \rightarrow \KK$ sends a tuple $X \in \TT^E$ to its support $\theta(X) = \underline{X} \in \KK^E$.
As such, initial matroids of $\TT$-matroids are ordinary matroids.
To avoid confusion, we shall always write $\0$ and $\1$ for the identity elements of $\KK$ and $\infty$ and $0$ for the identity elements of $\TT$.

When specialised to $\TT$, Theorem~\ref{thm:D} recovers (i)-(iii) of the characterisation of tropical linear spaces in \cite[Theorem B]{Brandt+Eur+Zhang:2021}.
For (i) and (ii), this is immediate, but (iii) requires some additional explanation.
Characterisation \ref{eq:iii} and Equation~\ref{eq:covector+containment} imply that $X \in \cL_M$ if and only if $\theta(X) \perp C^{X}$  for all $C \in \cC(M)$.
Note that if $\theta(X_i) = \0$, then $X_i = \infty$ and so $C_i^{X} = \0$ also.
This, along with $N_\KK$ being all summations other than the singleton element $\1$, implies that
\[
\theta(X) \perp C^{X} \, \Longleftrightarrow \, \sum_{i\in E} \theta(X_i)\cdot C_i^{X} \in N_{\KK} \, \Longleftrightarrow \, C^{X} \text{ not a singleton } \, .
\]
As the circuits of $M^{X}$ are the set of $C^{X}$ of minimal support, this holds if and only if $M^{X}$ is loopless, i.e. no single element circuits.
As such, we recover the following familiar characterisation of a tropical linear space associated to a $\TT$-matroid
\begin{equation} \label{eq:loopless}
\cL_M = \SetOf{X \in \TT^E}{M^{X} \text{ is loopless }} \, .
\end{equation}
This characterisation first appears for tropical linear spaces restricted to $\RR^E$ in~\cite{Speyer:2008,Rincon:2013}, and all of $\TT^E$ in~\cite{Brandt+Eur+Zhang:2021}.
To avoid issues with $+\infty$ and $-\infty$, they instead look at an equivalent condition where $(M^{X})^*$ is coloopless, which is required due to their polyhedral definition of initial matroids.
Our algebraic definition means we can simply state the usual characterisation.

Characterisations (iv) of Theorem~\ref{thm:D} and \cite[Theorem B]{Brandt+Eur+Zhang:2021} are not immediately comparable, as the former is phrased as a span over the tropical hyperfield and the latter as a span over the tropical semiring.
These two notions are usually not directly comparable, but the equivalence between the characterisations (i)-(iii) allows us to deduce some of the interplay between these two notions.
Given a finite subset $\cX \subseteq \TT^E$, we write its tropical semiring span as
\[
\tspan(\cX) := \SetOf{\min(\alpha_X + X) \in \TT^E}{X \in \cX \, , \, \alpha_X \in \RR \cup \{\infty\}}
\]
where $\min$ is applied coordinatewise.
We define the tropical hyperfield span as the usual $\spn(\cX)$ over $\TT$ as in \eqref{eq:span}.
Clearly $\tspan(\cX) \subseteq \spn(\cX)$ and generally this containment is strict.
However, when $M$ is a $\TT$-matroid we can deduce the following equality relating $\spn$ and $\tspan$ of its cocircuits from Theorem~\ref{thm:D} and \cite[Theorem B]{Brandt+Eur+Zhang:2021}:
\begin{equation} \label{eq:tspan}
\cL_M = \bigcap_{B \in \cB(\underline{M})} \spn(\cC^*_B(M)) = \tspan(\cC^*(M)) \, .
\end{equation}
While the latter notion is seemingly more natural, we shall see that it becomes increasingly  hard to define as we move to more general tropical extensions.

Finally, we note that \cite[Theorem B]{Brandt+Eur+Zhang:2021} has a fifth characterisation in terms of the topological closure of $\cL_{M/\ell} \times \{\infty\}^\ell$ inside $\TT^E$, where $\ell$ is the set of loops of $M$.
As it is unclear what topology can be imbued on more general tropical extensions, we do not try to generalise this characterisation.

\begin{example}\label{ex:trop+linear+space}
Recall the linear space $L \subseteq \hseries{\RR}{\RR}^4$ from Example~\ref{ex:linear+space}, and denote its associated $\hseries{\RR}{\RR}$-matroid $N$.
As defined in Example~\ref{ex:val}, $(\hseries{\RR}{\RR}, \val)$ is a valued field.
As such, Proposition~\ref{prop:tropicalisation} implies the image $\val(L)$ in the valuation map is the tropical linear space $\cL_M$ where $M = \val_*(N)$ is a $\TT$-matroid.
    Its points are precisely those elements orthogonal to $\val(\cC(N))$, namely
\[
\cC(M) = \val(\cC(N)) = \left\{
    (a,a,a,\infty) \, , \,
    (a,a,\infty,a) \, , \,
    (a+1,\infty,a,a) \, , \,
    (\infty,a+1,a,a) \, \mid a \in \RR \right\} \, .
    \]
    It can be be checked that $\cL_M$ is the set of tuples $X \in \TT^4$ satisfying one of the following seven conditions:
    \begin{align*}
    [1] \quad &\colon \quad X_1 > X_2 = X_3 = X_4  &
    [2] \quad &\colon \quad X_2 > X_1 = X_3 = X_4  \\
    [3] \quad &\colon \quad X_1 = X_2 = X_3 = X_4  &
    [4] \quad &\colon \quad X_3 = X_4 > X_1 = X_2 > X_4 -1 \\
    [5] \quad &\colon \quad X_3 > X_4 = X_1 + 1 = X_2 +1  &
    [6] \quad &\colon \quad X_3 = X_4 = X_1 +1 = X_2 +1 \\
    [7] \quad &\colon \quad X_4 > X_3 = X_1 +1 = X_2 +1 &&
    \end{align*}
    The tropical linear space is displayed in Figure~\ref{fig:trop+linear+space}.
    There are seven non-trivial initial $\KK$-matroids associated to this linear space, one corresponding to each of the seven `polyhedral' pieces of $\cL_M = \val(L)$.
    For example, any $X$ in the ray labelled $[1]$ gives rise to the same loopless initial matroid with circuits and covectors
    \[
    \cC(M^{X}) = \{(\0,\1,\1,\0), (\0,\1,\0,\1), (\0,\0,\1,\1) \} \quad , \quad \cV^*(M^{X}) = \{(\bullet,\1,\1,\1), (\bullet,\0,\0,\0) \mid \bullet \in \KK\} \, .
    \]
    Note that by \ref{eq:iii} of Theorem~\ref{thm:D}, the covectors of $\cV^*(M^{X})$ give rise to legitimate points in $\cL_M$ if and only if they have the same support as $X$.
    For example, $\theta(X) = (\1,\1,\1,\1)$ only when $X_1 \neq \infty$ and $\theta(X) = (\0,\1,\1,\1)$ only when $X_1 = \infty$.
    There is no such $X$ in this ray with $\theta(X) = (\bullet,\0,\0,\0)$ as $X_1$ must be strictly greater than the other coordinates.
    
\end{example}

\begin{figure}[ht]
\centering
\includegraphics[width=0.4\textwidth]{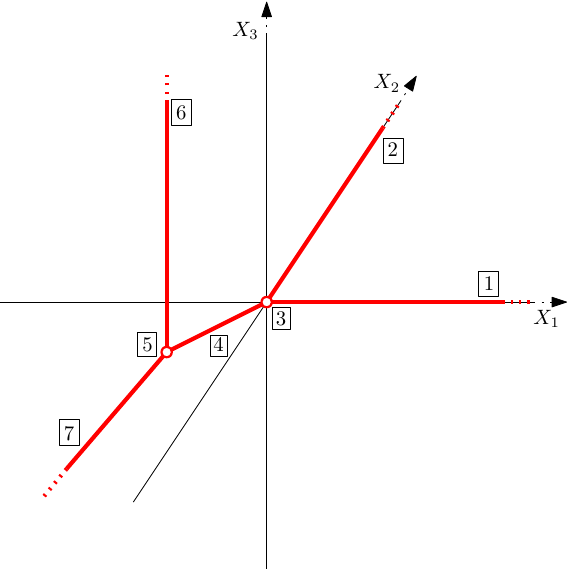}
\caption{The tropical linear space $\cL_M = \val(L)$ from Example~\ref{ex:trop+linear+space}, normalised to $X_4 = \1_{\TT} = 0$.
The labels correspond to the seven non-trivial initial $\KK$-matroids of $M$.}
\label{fig:trop+linear+space}
\end{figure}

\begin{example}
We give a brief example of a tropical linear space whose underlying matroid contains loops.
Let $M$ be the $\TT$-matroid with circuits
\[
\cC(M) = \left\{
    (a,a,\infty,\infty) \, , \,
    (a,\infty,a,\infty) \, , \,
    (\infty,a,a,\infty) \, , \,
    (\infty,\infty,\infty,a) \, \mid a \in \RR \right\} \, .
\]
The underlying matroid $\underline{M}$ has circuits $\{12,13,23,4\}$: in particular it contains $4$ as a loop.
The resulting tropical linear space is
\[
\cL_M = \{(\lambda, \lambda, \lambda, \infty) \, \mid \, \lambda \in \TT\} \, ,
\]
where $X_4 = \infty$ for all $X \in \cL_M$ to ensure orthogonality with the final circuit.
As such, for any $X \in \cL_M$ we have $M^{X} = \{12,13,23\}$ is loopless; in fact, $4$ has become a coloop.
\end{example}

\paragraph{Signed tropical linear spaces}
Consider the signed tropical hyperfield $\TT_\pm \cong \SS[\RR]$.
We denote the phase map $\theta$ by $\sgn\colon \TT_\pm \rightarrow \SS$ as it records sign data of an element, and recall that the modulus map $|\cdot|\colon \TT_\pm \rightarrow \TT$ forgets the sign data.
Given a $\TT_\pm$-matroid $M$, its initial matroids are $\SS$-matroids, i.e. oriented matroids.
As such, we can view the signed tropical linear space $\cL_M \subseteq \TT_\pm^E$ as a tropical linear space with an oriented matroid associated to each point.
To again avoid confusion between $\SS$ and $\TT_\pm$, we will write the elements of $\SS$ as $\{+, \0, -\}$ and elements of $\TT_\pm$ as $(\pm, a)$ for $a \in \RR$, and $\infty$ as the zero element of $\TT_\pm$.
Recall that $\1_{\TT_\pm} = (+,0)$.

We consider a couple of descriptions of signed tropical linear spaces arising from Theorem~\ref{thm:D}.
We first deduce a description analogous to \eqref{eq:loopless} from characterisation \ref{eq:iii} by considering when $\sgn(X) \in \cV^*(M^{|X|})$.
We say that two sign vectors $Y, Z \in \SS^E$ are \emph{conformal} if $Y_e \cdot Z_e \neq -$ for all $e \in E$, and \emph{non-conformal} otherwise.
Characterisation \ref{eq:iii} and Equation~\eqref{eq:covector+containment} imply that $X \in \cL_M$ if and only if $\sgn(X) \perp C^{|X|}$  for all $C \in \cC(M)$, i.e.
\[
\sgn(X) \perp C^{|X|} \, \Longleftrightarrow \, \sum_{i\in E} \sgn(X_i)\cdot C_i^{|X|} \in N_{\SS} \, \Longleftrightarrow \, \exists j,k \in E \text{ s.t } \sgn(X_j)\cdot C_j^{|X|} = - \sgn(X_k)\cdot C_k^{|X|} \neq \0 \, .
\]
There exists some $C \in \cC(M)$ that does not satisfy this if and only if either $C^{|X|}$ or $-C^{|X|}$ is conformal with $\sgn(X)$.
Hence, we can quantify the corresponding signed tropical linear space as
\begin{equation*}\label{eq:nonconformal}
\cL_M = \SetOf{X \in \TT_\pm^E}{\sgn(X) \text{ and } C^{|X|} \text{ are non-conformal } \forall C \in \cC(M)} \, .
\end{equation*}
Note that it is necessary that $M^{|X|}$ is loopless for all $X \in \cL_M$, though this is no longer sufficient.

Characterisation \ref{eq:iv} again describes $\cL_M$ as (the intersection of) the span of fundamental cocircuits $\cC^*_B(M)$.
In their study of flavours of convexity over the signed tropical hyperfield, Loho and Skomra~\cite[Theorem 7.8]{Loho+Skomra:2024} give a different description of $\cL_M$, showing it is the \emph{tropical-closed span} ${\rm tcspan}(\cC^*(M)) \subseteq \TT_\pm^E$ of all of the cocircuits.
Our notion of span is equivalent to their \emph{tropical-open span}, which is a superset of the tropical-closed span.
As with $\TT$, this gives us a relation between these two different notions of span over $\TT_\pm$:
\begin{equation}\label{eq:tcspan}
\cL_M = \bigcap_{B \in \cB(\underline{M})} \spn(\cC^*_B(M)) = {\rm tcspan}(\cC^*(M)) \, .
\end{equation}
Despite being the cleaner formulation, the tropical-closed span is much more delicate to define and to work with.
It would be interesting to generalise Equations~\eqref{eq:tspan} and~\eqref{eq:tcspan} to other tropical extensions of tracts by finding finer notions of span over these tracts, analogous $\tspan$ and ${\rm tcspan}$.

\begin{example}\label{ex:signed+tropical+linear+space+1}
Consider the $\TT_\pm$-matroid $M$ with the unique (up to scaling) circuit $[(+,0),(+,0),(-,0)] \in \TT_\pm^3$.
The corresponding signed tropical linear space $\cL_M$ is the set of covectors of $M$:
\begin{align*}
& & R_1 &= \SetOf{X \in \TT_\pm^3}{|X_2| = |X_3| \leq |X_1| \, , \, \sgn(X_2) = \sgn(X_3) \neq \0} \\
\cL_M = \cV^*(M) &=  R_1 \cup R_2 \cup R_3 \, , & R_2 &= \SetOf{X \in \TT_\pm^3}{|X_1| = |X_3| \leq |X_2| \, , \, \sgn(X_1) = \sgn(X_3) \neq \0} \\
& & R_3 &= \SetOf{X \in \TT_\pm^3}{|X_1| = |X_2| \leq |X_3| \, , \, \sgn(X_1) = -\sgn(X_2) \neq \0}
\end{align*}
Two different visualisations of $\cL_M$ are given in Figure~\ref{fig:sign+trop+linear+space+1}.
On the left, we simply view $\cL_M$ in $\TT_\pm^3$ normalised to $X_3 = \1_{\TT_\pm} = (+,0)$, where $R_1$ is red, $R_2$ is blue and $R_3$ is black.
\begin{figure}
\centering
\includegraphics[width=\textwidth]{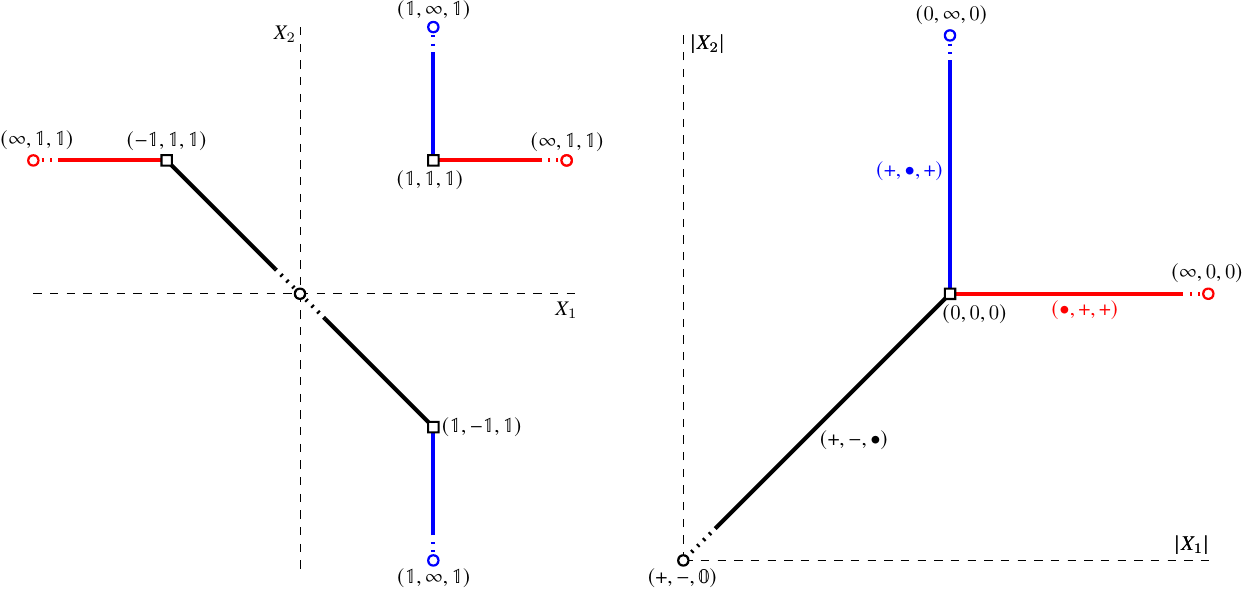}
\caption{The signed tropical linear space from Example~\ref{ex:signed+tropical+linear+space+1}.
The left figure shows $\cL_M$ in $\TT_\pm^3$, normalised to $X_3 = \1 = (+,0)$.
The right figure shows $\cL_M$ as a tropical linear space in $\TT^3$, normalised to $|X_3| = 0$, with an initial $\SS$-matroid attached to each point.}
\label{fig:sign+trop+linear+space+1}
\end{figure}

On the right, we take an initial matroid viewpoint by viewing $\cL_M$ in $\TT^3$ (normalised to $|X_3| = 0$) where each point $|X|$ has an initial $\SS$-matroid $M^{|X|}$ attached.
Write $|R_i|$ for the image of $R_i$ in $\TT^3$, we again have $|R_1|$ is red, $|R_2|$ is blue and $|R_3|$ is black.
Any red point $|X| \in |R_1|$ has the initial oriented matroid $M^{|X|}$ with circuits and covectors
\[
\cC(M^{|X|}) = \{(\0,+,-),(\0,-,+)\} \quad , \quad \cV^*(M^{|X|}) = \SetOf{(a,b,b)}{a,b \in \SS} \subseteq \SS^3 \, .
\]
A covector of $M^{|X|}$ gives rise to a point of $\cL_M$ if and only if it has the same support as $X$, i.e.
\begin{align*}
X_1 \neq \infty \, &\Rightarrow \, \sgn(X) \in \{(+,+,+), (-,+,+), (+,-,-), (-,-,-)\} \, , \\
X_1 = \infty \, &\Rightarrow \, \sgn(X) \in \{(\0,+,+), (\0,-,-)\}  \, .
\end{align*}
Note that the latter case is the point at infinity closing the open ray $|R_1|$.
In Figure~\ref{fig:sign+trop+linear+space+1}, we label the ray $|R_1|$ with $(\bullet,+,+)$ to denote the initial oriented matroid has covectors given by non-zero scalings of $(\bullet,+,+)$ where $\bullet \in \SS$ is non-zero if and only if $|X_1| \neq \infty$.
The initial matroids of other rays are labelled similarly.
It is straightforward to verify that the initial matroid at the `origin' $\bm{0} = (0,0,0) \in \TT^3$ has circuits and covectors
\begin{equation}\label{eq:oriented+initial+matroid}
\cC(M^{\bm{0}}) = \{(+,+,-),(-,-,+)\} \, , \quad \cV^*(M^{\bm{0}}) = \{(a,-a,b), (a,b,a), (b,a,a) \mid a, b \in \SS \} \subseteq \SS^3 \, .
\end{equation}
However, we again note that only those covectors with the same support as $\bm{0}$ arise as $\sgn(X)$ for some $|X| = \bm{0}$, i.e. those covectors where $a, b \neq \0$.

The relationship between these two viewpoints can be seen through the wider lens of Viro's patchworking method~\cite{Viro:1984}.
The covectors of the initial matroid $M^{|X|}$ tell us which orthants in $\TT_\pm^3$ contain the point $|X|$.
For example, Equation~\eqref{eq:oriented+initial+matroid} tells us that $[(+,0),(+,0),(+,0)]$ and $
[(-,0),(+,0),(+,0)]$ appear in $\cL_M$, but $[(-,0),(-,0),(+,0)]$ does not.
One can view moving between the two viewpoints in Figure~\ref{fig:sign+trop+linear+space+1} by folding or unfolding the orthants onto the positive orthant.
\end{example}

From now on, we will take the second approach of viewing enriched tropical linear spaces as tropical linear spaces with initial matroids attached to each point.
This will make figures more approachable, but know that we can always take the former viewpoint by some analogue of patchworking.

\begin{example}\label{ex:signed+tropical+linear+space+2}
    Recall the linear space $L \subseteq \hseries{\RR}{\RR}^4$ from Example~\ref{ex:linear+space}.
    This linear space is the set of covectors $\cV^*(N)$ of an $\hseries{\RR}{\RR}$-matroid $N$.
    By Proposition~\ref{prop:tropicalisation}, the signed tropicalisation of this linear space can be given as both $\sval(L)$ and $\cL_M$ where $M = \sval_*(N)$ is a $\TT_\pm$-matroid.
    Its points are precisely those elements orthogonal to the $\TT_\pm$-circuits
    \[
    \cC(M) = \sval(\cC(N)) =
    \left\{ \begin{aligned}
    &\alpha \cdot [(+,0),(+,0),(+,0),\infty]\, , \,
    \alpha \cdot [(+,0),(+,0),\infty,(-,0)]\, , \\
    &\alpha \cdot [(+,1),\infty,(-,0),(-,0)]\, , \,
    \alpha \cdot [\infty,(+,1),(+,0),(+,0)] \, 
    \end{aligned}
    \; \bigg| \; \alpha \in \TT_\pm \setminus \{\infty\}
    \right\} \, .
    \]
    We view the corresponding signed tropical linear space as a tropical linear space in $\TT^4$ with an initial $\SS$-matroid associated to each point.
    This is the tropical linear space in Figure~\ref{fig:trop+linear+space} with an initial $\SS$-matroid associated to each of the seven pieces of $\val(L)$.
    
    As an explicit example, write $R = \{|X| \in \val(L) \, \mid \, |X_1| > |X_2| = |X_3| = |X_4|\}$ for the ray [1] of $\val(L)$, and consider $X \in \cL_M$ such that $|X| \in R$.
    From Proposition~\ref{prop:initial+circuits}, we can calculate that the initial $\SS$-matroid $M^{|X|}$ has circuits 
    \[
 \cC(M^{|X|}) = \{(\0,+,+,\0), (\0,-,-,\0), (\0,+,\0,-),(\0,-,\0,+), (\0,\0,+,+), (\0,\0,-,-)\} \, .
    \]
    From this, we deduce that the covectors of the initial matroid are $\cV^*(M^{|X|}) = \left\{(a, b, -b,b) \, \mid \, a,b \in \SS \right\}$.
    Using \ref{eq:iii} from Theorem~\ref{thm:D}, we deduce that the subset $\tilde{R} \subseteq \cL_M$ that maps onto $R$ under $|\cdot|$ is precisely
    \begin{align*}
    \tilde{R} &= \tilde{R}^\circ \cup \tilde{R}^\infty = \left\{X \in \cL_M \, \mid \, |X| \in R \right\} \\
\text{where} \quad   \tilde{R}^\circ &= \{[(a,\eta), (b,\gamma), (-b,\gamma), (b,\gamma)] \in \TT_\pm^4 \, \mid \, \gamma, \eta \in \RR\, , \,\gamma < \eta \, , \, a, b \in \{+,-\}\} \\
       \tilde{R}^\infty &=  \{[\infty, (b,\gamma), (-b,\gamma), (b,\gamma)] \in \TT_\pm^4 \, \mid \, \gamma \in \RR \, , \, b \in \{+,-\}\} \, .
    \end{align*}
    The set $\tilde{R}^\circ$ maps precisely onto the open part of $R$, while $\tilde{R}^\infty$ maps onto the closed point at infinity.
Via Viro's patchworking, we can consider $\tilde{R}$ as an unfolding of $R$ governed by the initial oriented matroid.

Repeating this calculation across all pieces of $\val(L)$, we obtain the covectors of all seven initial oriented matroids:
        \begin{align*}
    [1] \quad &\colon \quad \cV^*(M^{|X|}) = \left\{(a, b, -b,b) \, \mid \, a,b \in \SS \right\} \\
    [2] \quad &\colon \quad \cV^*(M^{|X|}) = \left\{(b, a, -b,b) \, \mid \, a,b \in \SS \right\} \\
    [3] \quad &\colon \quad \cV^*(M^{|X|}) = \left\{(a, b, -b,b), (b, a, -b,b) \, \mid \, a,b \in \SS \right\} \\
    [4] \quad &\colon \quad \cV^*(M^{|X|}) = \left\{(a,-a,b,-b) \, \mid \, a,b \in \SS \right\} \\
    [5] \quad &\colon \quad \cV^*(M^{|X|}) = \left\{(a,-a,b,a) \, \mid \, a,b \in \SS \right\} \\
    [6] \quad &\colon \quad \cV^*(M^{|X|}) = \left\{(a,-a,b,a), (a,-a,a,b) \, \mid \, a,b \in \SS \right\} \\
    [7] \quad &\colon \quad \cV^*(M^{|X|}) = \left\{(a,-a,a,b) \, \mid \, a,b \in \SS \right\} \\
    \end{align*}
    These completely characterise $\cL_M = \sval(L)$ and give us an unfolding of $\val(L) \subset \TT^4$ into $\sval(L) \subset \TT_\pm^4$.
\end{example}

\begin{remark}\label{rem:signed+survey}
There is already a fairly substantial literature surrounding signed tropical linear spaces, as highlighted in Section~\ref{sec:related+work}.
We give a brief survey of what is known and how it relates to our results.

Tabera \cite{Tabera:2015} studies tropicalisations of varieties over real Puiseux series via real tropical bases.
While real tropical bases do not exist for all real varieties, they do for sufficiently simple varieties including linear spaces.
In particular, one can deduce from \cite[Theorem 3.14]{Tabera:2015} that a real tropical basis for a linear space is precisely the circuits of an oriented valuated matroid.
This agrees with our definition of a signed tropical linear space as the set of covectors of an oriented valuated matroid.

A number of authors have considered sign structures on Bergman fans from different perspectives.
J\"urgens \cite{Jurgens:2018} introduces the \emph{signed Bergman fan} of an oriented matroid with respect to a sign vector.
Following the framework of \cite{Tabera:2015}, these arise as a subset of the tropicalisation of linear spaces over $\RR$, namely those points with the specified sign vector.
Restricting to the positive sign vector recovers the positive Bergman fan of~\cite{ArdilaKlivansWilliams:2006}.
Celaya \cite{Celaya:2019} introduces the \emph{real Bergman fan} as a subfan of the normal fan of the signed matroid polytope, which can be viewed as a gluing of  J\"urgens' signed Bergman fans.
They give two different fan structures on this fan: a fine structure in terms of flags of conformal vectors, and a coarse structure in terms of loopless initial matroids.
Finally, \cite{Rau+Renaudineau+Shaw:2022} study real phase structures on Bergman fans, giving a cryptomorphism with oriented matroids.
This is directly connected to the previous two constructions, viewed either as a union of J\"urgens' signed Bergman fans over all sign vectors, or as a `folding' of Celaya's real Bergman fan.
All three of these constructions may be viewed as signed tropical linear spaces, arising as the signed tropicalisation of a linear space over $\RR$.
Their corresponding oriented valuated matroids have Pl\"ucker vectors that take values only in $\{(+,0),\infty, (-,0)\}$.

\cite{Celaya+Loho+Yuen:2024} consider an abstract analogue of Viro's patchworking for oriented matroids, where patchworking is governed by a matroid subdivision of a matroid polytope.
When restricted to regular matroid subdivisions, their results precisely give Theorem~\ref{thm:A} in the case where $F[\Gamma] = \TT_\pm$.
Finally, we note again that \cite{Loho+Skomra:2024} describe a signed tropical linear space as the tropical-closed span of all its cocircuits.
\end{remark}

\paragraph{Fine tropical linear spaces}
%As an example, in~\cite{Maxwell+Smith:2023} the authors consider varieties over tropical extensions of hyperfields.
%In particular, they give a fundamental theorem of tropical geometry for fine tropical varieties, varieties over the hyperfield $\FF \rtimes \Gamma$ where $\FF$ is an algebraically closed field.
%Any field $\FF$ is perfect, and so Proposition~\ref{prop:covectors} applies.
%We can describe linear spaces over $\FF \rtimes \Gamma$ as a tropical linear space where at each point we have a copy of a genuine linear space over $\FF$ determined by the initial matroid.

Our final example is the fine valuation $\fval \colon \hseries{\FF}{\RR} \rightarrow \FF[\RR]$ introduced in Example~\ref{ex:fine+val}.
Unlike the previous two examples, this enriched valuation and its tropical extension $\FF[\RR]$ have not been widely studied outside of~\cite{Maxwell+Smith:2023}.
However, Proposition~\ref{prop:tropicalisation} and Remark~\ref{rem:stringent} make it an exceedingly natural candidate for further study.

Recall that $\fval$ remembers the leading term $\lt(\omega)$ of a Hahn series $\omega$, and so elements of $\FF[\RR]$ can be viewed as a first-order approximation of elements of $\hseries{\FF}{\RR}$.
The phase map $\theta\colon \FF[\RR] \rightarrow \FF$ records the leading coefficient $\lc(\omega)$ of $\omega$, and as previously the modulus map $| \cdot | \colon \FF[\RR] \rightarrow \TT$ records the leading power $\lp(\omega)$ of $\omega$.
The initial matroids of an $\FF[\Gamma]$-matroid $M$ are $\FF$-matroids, linear spaces over $\FF$, and so we can view the \emph{fine tropical linear space} $\cL_M$ as a tropical linear space with a linear space over $\FF$ attached to each point.
Moreover, if $\cL_M = \fval(L)$ arises as the fine tropicalisation of the linear space $L \subseteq \hseries{\FF}{\RR}^E$, we can view $\cL_M$ as a first-order approximation of $L$.
This gives us an explicit description of $\cL_M$ in terms of initial degenerations of the original linear space $L$.

We very briefly recall some Gr\"obner theory over valued fields, see~\cite{Maclagan+Sturmfels:2015} for more details.
Given some $\bu \in (\RR \cup \{\infty\})^E$, consider the map of polynomial rings
\begin{align*} \label{eq:degeneration}
\initial_\bu \colon\hseries{\FF}{\RR}[Z_e : e \in E] &\longrightarrow \FF[Z_e : e \in E] \nonumber\\
f = \sum_{\ba \in A} \omega_\ba \cdot Z^\ba &\longmapsto \initial_\bu(f) := \sum_{\ba \in A_{\min}} \lc(\omega_\ba)\cdot Z^\ba \, , \quad A_{\min} = \underset{\ba \in A}{\argmin}(\lp(\omega_\ba) + \ba \cdot \bu)
\end{align*}
Viewing $t$ as a variable in the first ring, this is the \emph{initial form} of the polynomial $f$ with respect to the term order $(1, \bu)$.
Given an ideal $I \subseteq\hseries{\FF}{\RR}[Z_e : e \in E]$, its \emph{initial ideal} $\initial_\bu(I)$ with respect to $\bu$ is 
\[
\initial_\bu(I) = \langle \initial_\bu(f) \mid f \in I\rangle \subset \FF[Z_e : e \in E] \, ,
\]
the ideal generated by all initial forms of polynomials in $I$.
If $V := V(I) \subseteq \hseries{\FF}{\RR}^E$ is the associated variety of $I$, we define its \emph{initial degeneration} with respect to $\bu$ to be $\initial_\bu(V) := V(\initial_\bu(I)) \subseteq \FF^E$.

Now let $L \subseteq \hseries{\FF}{\RR}^E$ be the linear space with corresponding $\hseries{\FF}{\RR}$-matroid $N$ such that $\fval(L) = \cL_M$, or equivalently $M = \fval_*(N)$.
As defined in~\eqref{eq:circuit+ideal}, $L$ is the linear space cut out by the ideal $I_N$ where 
\[
I_N = \langle f_C \mid C \in \cC(N) \rangle \, , \quad f_C = \sum_{i\in E} C_i \cdot Z_i \in \hseries{\FF}{\RR}[Z_e : e \in E] \, .
\]
Moreover, picking one $f_C$ for each possible support $\underline{C}$ gives a universal Gr\"obner basis for $I_N$.
Hence the initial degeneration of $L$ with respect to $\bu$ is
\begin{align} \label{eq:initial+degen}
\initial_\bu(L) &= V(\langle \initial_\bu(f_C) \mid C \in \cC(N) \rangle) \subseteq \FF^E \, , \nonumber \\
\text{ where } \initial_\bu(f_C) &= \sum_{i \in E_{\min}} \lc(C_i) \cdot Z_i \, , \quad E_{\min} =  \underset{i \in E}{\argmin}(\lp(C_i) + u_i) \, .
%\initial_\bu(L) = V(\initial_\bu(I_N)) \subseteq \FF^E \, , \quad 
%\initial_\bu(I_N) &= \langle  \initial_\bu(f_C)\mid C \in \cC(N) \rangle \, , \\
%\initial_\bu(f_C)&=\sum_{i \in \argmin(\lp(C_i) + u_i)} \lc(C_i) \cdot Z_i
\end{align}
We now relate these initial degenerations back to $\cL_M$ and initial matroids of $M$.
Given some $C \in \cC(N)$, we write $\widetilde{C} = \fval(C) = ((\lc(C_i),\lp(C_i))_{i \in E}$ for the corresponding circuit of $M$ in the image of $\fval$.
Given some $\bu \in (\RR \cup \{\infty\})^E$, the initial circuit $\widetilde{C}^\bu$ of $\widetilde{C}$ as defined in \eqref{eq:initial+circuit} is
\[
\widetilde{C}^\bu = \begin{cases}
\lc(C_i) & i \in \underset{i \in E}{\argmin}(\lp(C_i) + u_i) \\
\0_\FF & \text{ otherwise }
\end{cases} \, .
\]
Hence $\widetilde{C}^\bu$ is precisely the coefficients of the linear form $\initial_\bu(f_C)$ from \eqref{eq:initial+degen}, and therefore $\cV^*(M^\bu) = \initial_\bu(L)$.
Applying characterisation \ref{eq:iii} from Theorem~\ref{thm:D} gives a description of $\cL_M$ in terms of initial degenerations of $L$:
\begin{equation} \label{eq:linear+degen}
\cL_M = \SetOf{X \in \FF[\RR]^E}{\theta(X) \in \initial_{|X|}(L)} \, .
\end{equation}
%
%
%\[
%\theta(X) \perp C^{|X|} = \lc(\widetilde{X})\perp C^{\lp(\widetilde{X})} \, \Longleftrightarrow \, \sum_{i\in E} \lc(\widetilde{X}_i)\cdot C_i^{\lp(\widetilde{X})} \in N_{\FF} \, .
%\]
%Fixing the leading powers $|X| = \lp(\widetilde{X})$ of $\widetilde{X}$, each initial circuit $C^{\lp(\widetilde{X})}$ gives a linear equation over $\FF$ that the leading coefficients $\lc(\widetilde{X})$ must satisfy for $\widetilde{X} \in L$.
%These equations cut out the initial degeneration of $L$ with respect to the term weight $|X|$.

\begin{example}
    Again, recall the linear space $L \subseteq \hseries{\RR}{\RR}^4$ from Example~\ref{ex:linear+space}, the set of covectors $\cV^*(N)$ of an $\hseries{\RR}{\RR}$-matroid $N$.
    Alternatively, it is the linear space cut out by the ideal
    \begin{equation}\label{eq:real+ideal}
    I_N = \langle \,
    2Z_1 + Z_2 + Z_3 \, , \,
    2Z_1 + (1-t)Z_2 -Z_4 \, , \,
    2tZ_1 - (1-t)Z_3 -Z_4 \, , \,
    tZ_2 + Z_3 + Z_4 \, 
    \rangle \subseteq \hseries{\RR}{\RR}[Z_1,Z_2,Z_3,Z_4] \, .
    \end{equation}
    By Proposition~\ref{prop:tropicalisation}, the image of this linear space $\fval(L)$ in the fine valuation map is equal to the covectors $\cL_M = \cV^*(M)$ of the $\RR[\RR]$-matroid $M = \fval_*(N)$.
    We demonstrate how we can view $\cL_M$ in terms of initial degenerations of the original linear space $L$ using \eqref{eq:linear+degen}.
%    These are precisely those elements orthogonal to (scalings of) the $\RR[\RR]$-circuits
%    \[
%    \left\{ \, [(2,0),(1,0),(1,0),\infty]\, , \,
%    [(2,0),(1,0),\infty,(-1,0)]\, , \,
%    [(2,1),\infty,(-1,0),(-1,0)]\, , \,
%    [\infty,(1,1),(1,0),(1,0)] \, \right\} \, .
%    \]
%        We view the corresponding fine tropical linear space $\cL_M = \fval(L)$ as the tropical linear space $\val(L) \subset \TT^4$ with an initial $\RR$-matroid, i.e. a real linear space, associated to each point.%, though these linear spaces remain constant on the seven pieces of $\val(L)$.
    
    As an explicit example, write $R = \{|X| \in \val(L) \, \mid \, |X_1| > |X_2| = |X_3| = |X_4|\}$ for the ray [1] of $\val(L)$, and consider $X \in \cL_M$ such that $|X| \in R$.
    Rather than directly calculating the circuits of $M^{|X|}$ from Proposition~\ref{prop:initial+circuits}, we instead consider the initial degeneration $\initial_{|X|}(L)$ of $L$ with respect to $|X|$.
    By taking the initial ideal of \eqref{eq:real+ideal}, we see that for all $|X| \in R$ this is a variety in $\RR^4$ cut out by the ideal
    \[
    \initial_{|X|}(I_N) = \langle Z_2 + Z_3, Z_2 - Z_4, Z_3 + Z_4 \rangle \subseteq \RR[Z_1, Z_2,Z_3,Z_4] \, .
    \]
%    we can calculate that the circuits of $M^{|X|}$ are all scalar multiples of
%    \[
%\{(0,1,1,0), (0,1,0,-1), (0,0,1,1)\} \subset \RR^4 \, .
%    \]
%    From this, we deduce that the covectors of the initial matroid form the real linear space 
    From this, we deduce that $\initial_{|X|}(L)$ is the real linear space 
    \[
%    \cV^*(M^{|X|}) = \rm{span}_\RR \begin{bmatrix}
%    1 & 0 & 0 & 0 \\
%    0 & 1 & -1 & 1
%\end{bmatrix}     \, .
%\cV^*(M^{|X|})  = \rm{span}_\RR\{(1, 0, 0, 0) , (0, 1, -1, 1)\} \subset \RR^4 \, .
\initial_{|X|}(L) = \cV^*(M^{|X|})  = \{a\cdot(1, 0, 0, 0) + b \cdot(0, 1, -1, 1) \, \mid \, a,b \in \RR\} \subset \RR^4 \, .
    \]
    Combined with \eqref{eq:linear+degen}, we deduce that the subset $\tilde{R} \subseteq \cL_M$ that maps onto $R$ under $|\cdot|$ is
    \begin{align*}
    \tilde{R} &= \tilde{R}^\circ \cup \tilde{R}^\infty = \left\{X \in \cL_M \, \mid \, |X| \in R \right\} \\
\text{where} \quad   \tilde{R}^\circ &= \{[(a,\eta), (b,\gamma), (-b,\gamma), (b,\gamma)] \in \TT_\pm^4 \, \mid \, \gamma, \eta \in \RR\, , \,\gamma < \eta \, , \, a, b \in \RR^\times\} \\
   \tilde{R}^\infty &= \{[\infty, (b,\gamma), (-b,\gamma), (b,\gamma)] \in \TT_\pm^4 \, \mid \, \gamma \in \RR \, , \, b \in \RR^\times\} \, .
    \end{align*}
    Again, the set $\tilde{R}^\circ$ maps precisely onto the open part of $R$, while $\tilde{R}^\infty$ maps onto the closed point at infinity.
      We can again consider this as a `field analogue' of Viro's patchworking, where we unfold $R$ to $\tilde{R}$ through $\RR$.

Repeating this calculation across all pieces of $\val(L)$, we obtain the seven linear spaces over $\RR$ associated to each of the seven pieces of $\val(L)$, each an initial degeneration of the original linear space $L$:
        \begin{align*}
    [1] \quad &\colon \quad \initial_{|X|}(L) = \rm{span}_\RR\{(1, 0, 0, 0) , (0, 1, -1, 1)\} \\
    [2] \quad &\colon \quad \initial_{|X|}(L)  = \rm{span}_\RR\{(0, 1, 0, 0) , (1, 0, -2, 2)\} \\
    [3] \quad &\colon \quad \initial_{|X|}(L)  = \rm{span}_\RR\{(1, -2, 0, 0) , (1, 0, -2, 2)\} \\
    [4] \quad &\colon \quad \initial_{|X|}(L)  = \rm{span}_\RR\{(1, -2, 0, 0) , (0, 0, 1, -1)\} \\
    [5] \quad &\colon \quad \initial_{|X|}(L)  = \rm{span}_\RR\{(1, -2, 0, 1) , (0, 0, 1, 0)\} \\
    [6] \quad &\colon \quad \initial_{|X|}(L)  = \rm{span}_\RR\{(1, -2, 1, 0) , (1, -2, 0, 1)\} \\
    [7] \quad &\colon \quad \initial_{|X|}(L)  = \rm{span}_\RR\{(1, -2, 1, 0) , (0, 0, 0, 1)\} \\
    \end{align*}
    These completely characterise $\cL_M = \fval(L)$ and gives us an `unfolding' of $\val(L) \subset \TT^4$ into $\fval(L) \subset \RR[\RR]^4$.

\end{example}

%\section{Further work}
%
%\begin{itemize}
%\item Enriched Dressians
%\item Characterising when Prop~\ref{prop:tropicalisation} is an equality
%\item Transversal matroids?
%\end{itemize}

%\subsection{Representability}
%
%\todo[inline]{Matroids over tropical extensions of partial fields?}

\bibliographystyle{plain}
\bibliography{references}

\end{document}